\documentclass[smallextended]{svjour3}       
\smartqed  
\usepackage{graphicx}
\usepackage{amsfonts}
\usepackage{amsmath}
\usepackage{amssymb}
\usepackage{fancyhdr}
\usepackage{indentfirst}
\usepackage{placeins}
\usepackage{listings}
\usepackage{caption}
\usepackage{subfigure}
\usepackage{url}
\usepackage{bm}
\usepackage{bbm}
\usepackage{comment}
\usepackage{multirow}
\usepackage{xcolor}
\usepackage{mathtools}

\newcommand{\cL}{{\mathcal L}}
\newcommand{\bR}{{\mathbb R}}
\newcommand{\bE}{{\mathbf E}}

\newcommand{\wrt}{with respect to }

\newtheorem{assumption}{Assumption}
\newtheorem{clm}{Claim}

\DeclarePairedDelimiter{\floor}{\lfloor}{\rfloor}

\numberwithin{equation}{section}
\numberwithin{theorem}{section}
\numberwithin{lemma}{section}
\numberwithin{clm}{section}

\begin{document}

\title{Importance sampling in path space for diffusion processes
 with slow-fast variables 
}

\author{Carsten Hartmann \and Christof Sch\"{u}tte \and Marcus Weber \and Wei Zhang}

\authorrunning{C. Hartmann et al.} 

\institute{C. Hartmann, W. Zhang \at
              Institute of Mathematics, Freie Universit\"{a}t Berlin, Arnimallee 6, 14195 Berlin, Germany \\
              \email{carsten.hartmann@fu-berlin.de, wei.zhang@fu-berlin.de}           
       %
           \and
           C. Sch\"utte, M. Weber  \at
              Zuse Institute Berlin, Takustrasse 7, 14195 Berlin, Germany\\
              \email{schuette@zib.de, weber@zib.de}    
                                }

\date{Received: date / Accepted: date}

\maketitle

\begin{abstract}

Importance sampling is a widely used technique to reduce the variance of a 
Monte Carlo estimator by an appropriate change of measure. In this work, we study importance sampling in the framework
of diffusion process and consider the change of measure which is 
realized by adding a control force to the original dynamics. For certain exponential type expectation, the corresponding control force of the optimal change of measure leads to a zero-variance estimator and is 
 related to the solution of a Hamilton-Jacobi-Bellmann equation. 
We focus on certain diffusions with both slow and fast variables, and the main
result is that we obtain an upper bound of the relative error for the importance sampling estimators with control obtained from the limiting dynamics. 
We demonstrate our approximation strategy with an illustrative numerical example. 
\end{abstract}

\section{Introduction}
\label{sec-intro}
Monte Carlo (MC) methods are powerful tools to solve high-dimensional problems that are not amenable
to grid-based numerical schemes \cite{junliu_mc}. Despite their quite long history since the invention of the computer, 
the development of MC method and applications thereof are a field of active research. 
Variants of the standard Monte Carlo method include Metropolis MC \cite{mcmc-hastings,mcmc-brooks}, Hybrid
MC \cite{hmc-duane,hmc-Schutte}, Sequential
MC \cite{smc-liu,smc-Doucet2001}, to mention just a few.

A key issue for many MC methods is variance reduction in order to improve the convergence of the 
corresponding MC estimators. Although all unbiased MC estimators share the 
same $\mathcal{O}(N^{-\frac{1}{2}})$ decay of their variances with the sample size $N$, the prefactor matters a lot 
for the performance of the MC method.  
Therefore variance reduction techniques  (see, e.g., \cite{mc-asmussen-2007,junliu_mc}) seek to decrease the constant prefactor
 and thus to increase the accuracy and efficiency of the estimators. 
 
 In this paper, we focus on the importance sampling method for variance reduction. The basic idea is to 
generate samples from an alternative probability distribution (rather than 
sampling from the original probability distribution), so that the ``important'' regions in
state space are more frequently sampled. To give an example, consider a real-valued random variable $X$ on some probability space $(\Omega,\mathcal{F},\mathbf{P})$ and the calculation of a probability 
\[
\mathbf{P}(X\in B)=\mathbf{E}\!\left(\chi_{B}(X)\right)
\]
of the event $\{\omega\in\Omega\colon X(\omega)\in B\}$ that is rare. When the set $B$ is rarely hit by the random variable $X$, it may be a good idea to draw samples from another probability distribution, say, $\mathbf{Q}$ so that the event $\{X \in B\}$ has larger probability under $\mathbf{Q}$. An unbiased estimator of $\mathbf{P}(X\in B)$ can then be based on the appropriately reweighted expectation under $\mathbf{Q}$, i.e., 
\[
\mathbf{E}\!\left(\chi_{B}(X)\right) =
\mathbf{E}_{\mathbf{Q}}\!\left(\chi_{B}(X)\Psi\right)\,,
\]    
with $\Psi(\omega)=(\rm{d}\mathbf{P}/\rm{d}\mathbf{Q})(\omega)$ being the Radon-Nikodym derivative of $\mathbf{P}$ \wrt $\mathbf{Q}$. The difficulty now lies in a clever choice of $\mathbf{Q}$, because not every probability measure $\mathbf{Q}$ that puts more weight on  the ``important'' region $B$ leads to a variance reduction of the corresponding estimator. Especially in cases when the two probability distributions are too different from each other so that  the Radon-Nikodym derivative $\Psi$ (or likelihood ratio) becomes almost degenerate, the variance typically grows and one is better off with the plain vanilla MC estimator that is based on drawing samples from the original distribution $\mathbf{P}$. Importance sampling thus deals with clever choices of $\mathbf{Q}$ that enhance the sampling of events like $\{X\in B\}$ while mimicking the behaviour of the original distribution in the relevant regions. Often such a choice can be based on large deviation asymptotics that provides estimates for the probability of the event $\{X\in B\}$ as a function of a smallness parameter; see, e.g., \cite{ip-blanchet-2008,ip-glasserman,ip-asmussen-2008,ip-dupuis,ip-dupuis-multiscale,ip-eric}.

Here we focus on the path sampling problem for diffusion processes.
Specifically, given a diffusion process $(X_{t})_{t\ge 0}$ governed by a stochastic
differential equation (SDE), our aim is to compute the expectation of some
path functional of $X_{t}$ with respect to the underlying probability measure $\mathbf{P}$ generated by the Brownian motion. 
In this setting, we want to apply importance sampling and draw samples
(i.e.~trajectories) from a modified SDE to which a control force has been
added that drives the dynamics to the important regions in state space. The
control force generates a new probability measure on the space of trajectories
$(X_{t})_{t\ge 0}$, and estimating the expectation of the path functional \wrt
the original probability measure by sampling from the controlled SDE is
possible if the trajectories are reweighted according to the Girsanov theorem~\cite{oksendalSDE}. 
We confine ourselves to certain exponential path functionals which will be explicitly given below. For this type of path
functionals, the optimal change of measure exists that admits importance sampling estimator with zero
variance. Furthermore, the path sampling problem admits a dual formulation in terms of a 
stochastic optimal control problem, in which case finding the optimal change of measure is equivalent to 
solving the Hamilton-Jacobi-Bellmann (HJB) equation associated with 
the stochastic control problem. 

\textbf{Relevant work and contribution of this paper.}
While in general it is impractical to find the exact optimal control force by solving an optimal control problem, 
there is some hope to find computable approximations 
to the optimal control that yield importance sampling estimators which are sufficiently
accurate in that they have small variance. 
A general theoretical framework has been established by Dupuis and Wang in \cite{dupuis_isaaces,ip-dupuis},
where they connected the subsolutions of HJB equation and the rate of variance decay for the corresponding importance sampling estimators.
This theoretical framework has been further applied by Dupuis, Spiliopoulos and Wang
in a series of papers \cite{rare_event_rough,ip-dupuis-multiscale,ip-kostas1,ip-kostas3}
to study systems of quite general forms and several adaptive importance
sampling schemes were suggested based on large deviation analysis.
In many cases, these importance sampling schemes were shown to be
asymptotically optimal in logarithmic sense. Also see discussions in  \cite{ip-eric,spiliopoulos2015}. 
Closely related to our present work, dynamics involving two parameters $\delta
, \epsilon>0$, that represent time scale separation between slow and fast
variables and the noise intensity, were studied in \cite{ip-kostas1}. Therein
the author carried out a systematic analysis for dynamics within different
regimes that are expressed 
by the ratio $\frac{\epsilon}{\delta}$ as $\epsilon \rightarrow 0$,
where $\delta = \delta(\epsilon)$.
Importance sampling for systems in the regime when $\frac{\epsilon}{\delta} \rightarrow +\infty$ with
random environment was studied in \cite{ip-kostas3}.
A numerical scheme that leads to importance sampling estimators with vanishing relative error 
for diffusion processes in the small noise limit has been proposed in \cite{ip-eric}. 
On the other hand, while importance sampling is crucial in the small noise limit
when $\epsilon \rightarrow 0$, some recent
work~\cite{ip-kostas2,spiliopoulos2015} also considered the performance of
importance sampling estimators when $\epsilon$ is small but fixed
(pre-asymptotic regime), especially when systems' metastability is involved~\cite{ip-kostas2}.

Inspired by these previous studies, in the present work we consider importance
sampling for diffusions with both slow and fast time scales. 
See equation (\ref{averaging-dynamics}) in Section~\ref{sec-main}.
Instead of studying importance sampling estimators associated with general
subsolutions of the HJB equation as in \cite{ip-dupuis,rare_event_rough,ip-dupuis-multiscale,ip-kostas1,ip-kostas3}, we consider a specific control which can be constructed from the low-dimensional limiting dynamics.
The main contribution of the present work is Theorem~\ref{main_thm} in
Section~\ref{sec-main} which states that, under certain assumptions, 
the importance sampling estimator associated to this specific control is asymptotically optimal in
the time scale separation limit and an upper bound on the relative error of the corresponding estimator is obtained. 
To the best of our knowledge, this is the first result about the explicit dependence of
the relative error of the importance sampling estimator on the time-scale separation parameter. 
As a secondary contribution, since the proof is based on a careful study of the multiscale
process and the limiting process, several error estimates 
for the strong approximation of the original process by the limiting process are obtained as a
by-product. See Theorem~\ref{main-result-2}-\ref{main-result-4} in Section~\ref{sec-proof}.

Before concluding the introduction, we compare our results with the previous
work in more details and discuss some limitations. First of all, the two-scale dynamics (\ref{averaging-dynamics}) considered in the present work is a special case of the dynamics considered in \cite{ip-kostas1,ip-kostas3}
 (corresponding to coefficients $b=g=\tau_1=0$ there).
This specialization allows us to prove strong convergence of the dynamics towards the limit dynamics. 
Secondly, instead of considering asymptotic regime for both $\epsilon,
\delta \rightarrow 0$ as in \cite{ip-dupuis-multiscale,ip-kostas1,ip-kostas3}, here we only
consider the time-scale separation limit and assume the other parameter
$\beta$ in (\ref{averaging-dynamics}), which is related to system's
temperature, is fixed. (Roughly speaking, this corresponds to the case when $\delta \rightarrow 0$
with fixed $\epsilon$ in \cite{ip-kostas1,ip-kostas3}). As a consequence, the constant in
Theorem~\ref{main_thm} depends on $\beta$. Thirdly, we assume Lipschitz
conditions on system's coefficients, which
may be restrictive in many applications. 
Generalizing the theoretical results to non-Lipschitz case is possible but not trivial and
will be considered in future work. We refer to \cite{cerrai_nonlip} for a related
studies of reaction-diffusion equations. 

Nevertheless, the two-scale dynamics (\ref{averaging-dynamics}) is an interesting 
mathematical paradigm for many applications that involve both slow and fast
time scales 
(we refer to \cite{asymptotic_analysis,book_PS08} for general references about averaging and homogenization). 
And our results are of
different type comparing to the above mentioned literatures.
In applications, especially in climate sciences and molecular
dynamics~\cite{meta_simple_climate_model,majda_math_pers,msm},
systems may have a few degrees of freedom which evolves on a large time scale
and exhibits \emph{metastability} feature, while the other degrees of freedom are rapidly evolving.  
In this situation, due to the presence of metastability, standard
Monte Carlo sampling may become inefficient and shows large sample variance even for
moderate temperatures $\beta$ (also see \cite{ip-kostas2}). We expect our
results will be relevant for developing efficient importance sampling schemes in this situation.
A more detailed discussion based on an illustrative numerical example will be
presented in Section~\ref{sec-examples}.

\textbf{Organization of the article. }
This paper is organized as follows. In Section~\ref{sec-setup},
we briefly introduce the importance sampling method in the diffusion setting and discuss the variance of 
Monte Carlo estimators corresponding to a general
control force. Section~\ref{sec-main} states the assumptions and
our main result: an upper bound of the relative error 
for the importance sampling estimator based on suboptimal controls for the multiscale diffusions; the result is proved in
Section~\ref{sec-proof}, but we provide some heuristic arguments based on formal asymptotic 
expansions already in Section \ref{sec-main}. Section~\ref{sec-examples}
shows an illustrative numerical example that demonstrate the performance of the
importance sampling method. Appendix~\ref{app-1} and \ref{app-2} contain 
technical results that are used in the proof.

\section{Importance sampling of diffusions}
\label{sec-setup}

We consider the conditional expectation 
\begin{align}
I = \mathbf{E} \Big[\exp \Big(-\beta\int_t^T h(z_s) \,ds\Big)~\Big|~z_t = z\Big]
\label{exp-I}
\end{align}
on a finite time interval $[t,T]$, where $\beta > 0$, $h : \mathbb{R}^n
\rightarrow \mathbb{R}^+$, and $z_s \in \mathbb{R}^n$ satisfies the dynamics
\begin{align}
\begin{split}
  d z_s &= b(z_s) ds + \beta^{-1/2}\sigma(z_s) dw_s, \quad t \le s \le T \\
  z_t&=z
  \end{split}
\label{dynamics-1}
\end{align}
with $b : \mathbb{R}^n\rightarrow \mathbb{R}^n$, $\sigma :
\mathbb{R}^n\rightarrow \mathbb{R}^{n\times m}$, $w_s$ is a
standard $m$-dimensional Wiener process. 
Exponential expectations similar to (\ref{exp-I}) may arise either in connection with importance sampling 
\cite{ip-dupuis-multiscale,ip-kostas1,ip-kostas3,ip-eric}, or due to its
close relationship with certain optimal control problem \cite{var_rep1998,fleming2006}.
In recent years, it has also been exploited by physicists to study phase
transitions~\cite{Jack-1,Hedges06032009}.
\subsection{Importance sampling method}
\label{sub-sec-setup-1}
In this subsection we introduce the importance sampling
method to compute quantify (\ref{exp-I}). To simplify matters, we assume all the coefficients are
smooth and the controls satisfy the Novikov condition such that the
Girsanov theorem can be applied \cite{oksendalSDE}. Specific assumptions and
the concrete form of dynamics will be given in Section~\ref{sec-main}. 

It is known that dynamics (\ref{dynamics-1}) induces a probability measure $\mathbf{P}$ over the path ensembles $z_s,
t \le s \le T$ starting from $z$.
To apply the importance sampling method, we introduce
\begin{align}
  d\bar{w}_s = \beta^{1/2} u_s \,ds + dw_s, 
\label{girsanov-bs}
\end{align}
where $u_s \in \mathbb{R}^m$ will be referred to as the \emph{control force}. 
Then it follows from Girsanov theorem \cite{oksendalSDE} that $\bar{w}_s$ is a
standard $m$-dimensional Wiener process under
probability measure $\mathbf{\bar{P}}$, with Radon-Nikodym derivative 
\begin{align}
\frac{d\mathbf{\bar{P}}}{d\mathbf{P}} = Z_t = \exp
\Big(-\beta^{1/2} \int_t^T u_s \,dw_s -
\frac{\beta}{2}\int_t^T |u_s|^2 ds \Big).
\label{girsanov-zt}
\end{align}
In the following, we will omit the conditioning on the initial value at time $t$ . 
Letting $\mathbf{\bar{E}}$ denote the expectation under 
$\mathbf{\bar{P}}$, we have 
\begin{align}
I = \mathbf{E} \Big[\exp \Big(-\beta\int_t^T h(z_s) \,ds\Big)\Big] =
\mathbf{\bar{E}} \Big[\exp \Big(-\beta\int_t^T h(z^u_s) \,ds\Big) Z_t^{-1}\Big],
\label{target-exp}
\end{align}
with variance 
\begin{align}
  \mathrm{Var}_{u} I = \mathbf{\bar{E}}
\Big[\exp\Big(-2\beta\int_t^T h(z^u_s) \,ds\Big) (Z_t)^{-2}\Big]-I^2.
\label{variance}
\end{align}
Moreover, under $\mathbf{\bar{P}}$, we have 
\begin{align}
\begin{split}
  d z^u_s &= b(z^u_s) ds - \sigma(z^u_s) u_s\, ds + \beta^{-1/2}\sigma(z^u_s)
  d\bar{w}_s\,,   
  \quad t \le s \le T \\
  z^u_t &= z. \\
  \end{split}
\label{dynamics-2}
\end{align}

Now consider the calculation of (\ref{target-exp}) by a Monte Carlo sampling in path space, and suppose that $N$
independent trajectories $\{z^{u,i}_s, t \le s \le T\}$ of (\ref{dynamics-2}) have been generated
where $i = 1, 2, \cdots, N$. An unbiased estimator of (\ref{exp-I}) is now given by 
\begin{align}
I_N = \frac{1}{N}\sum_{i=1}^{N} \Big[\exp\Big(-\beta\int_t^T
h(z^{u,i}_s)\, ds\Big) (Z^{u,i}_t)^{-1} \Big]\,,
\label{ip_mc}
\end{align}
whose variance is 
\begin{align}
  \mathrm{Var}_u I_N = \frac{\mathrm{Var}_u I}{N} = \frac{1}{N} \Big[\mathbf{\bar{E}}
\Big(\exp\Big(-2\beta\int_t^T h(z^u_s) \,ds\Big) (Z_t)^{-2}\Big)-I^2\Big].
\label{variance_In}
\end{align}
Notice that $Z_t = 1$ when $u_s \equiv 0$, and we recover the standard Monte
Carlo method. In order to quantify the efficiency of the Monte Carlo method,
we introduce the \emph{relative error}~\cite{ip-dupuis,ip-eric}
\begin{align}
  \mbox{RE}_u(I) = \frac{\sqrt{\mathrm{Var}_u I}}{I}\,.
  \label{relative-error}
\end{align}
The advantage of introducing the control force $u_s$ is that we may
choose $u_s$ to reduce the relative error of the estimator (\ref{ip_mc}).
From (\ref{variance}) and (\ref{variance_In}), we can see that minimizing the
relative error of the new estimator is equivalent to choosing $u_s$ such that  
\begin{align}
  \frac{1}{I^2}\mathbf{\bar{E}} \Big[\exp\Big(-2\beta\int_t^T h(z^u_s) \,ds\Big) (Z_t)^{-2}\Big]
\label{var-exp}
\end{align}
is as close as possible to $1$.\\

\subsection{Dual optimal control problem and estimate of relative error}
\label{sub-sec-setup-2}
\noindent
To proceed, we make use of the following duality relation \cite{var_rep1998}:
\begin{align}
\ln\mathbf{E} \Big[\exp \Big(-\beta\int_t^T h(z_s) \,ds\Big)\Big] =
-\beta\inf_{u_s} \mathbf{\bar{E}} \Big\{\int_t^T h(z^u_s) \,ds + \frac{1}{2}\int_t^T
|u_s|^2 ds \Big\}\,,
\label{dual-relation}
\end{align}
where the infimum is over all processes $u_s$ which are progressively
measurable with respect to the augmented filtration generated by the Brownian motion.
See \cite{var_rep1998} for more discussions.
It is known that there is a feedback control $\hat{u}_s$ such that the infimum on the
right-hand side (RHS) of (\ref{dual-relation}) is attained (see \cite[Sec.~VI, Thm.~3.1]{fleming2006}).  
We will call $\hat{u}_s$ the \emph{optimal control force}. Accordingly we define
$\hat{w}_s, \hat{Z}_t, \mathbf{\hat{P}}$ to be the respective quantities in 
(\ref{girsanov-bs}) and (\ref{girsanov-zt}) with
$u_s$ replaced by $\hat{u}_s$, and we denote $\hat{z}_s=\hat{z}_{s}^{\hat{u}}$
as the solution of (\ref{dynamics-2}) with control force $\hat{u}_s$. 
Using Jensen's inequality one can show that (\ref{dual-relation}) implies 
\begin{align}
\exp \Big(-\beta\int_t^T h(\hat{z}_s) \,ds\Big) \hat{Z}^{-1}_t = I,
\qquad \mathbf{\hat{P}}-a.s. 
\label{as-const}
\end{align}
Combining the above equality with (\ref{variance_In}), it follows that the change of measure induced by $\hat{u}_s$ is optimal in the sense that the variance of the importance sampling estimator (\ref{ip_mc}) vanishes.  

It is helpful to note that the RHS of (\ref{dual-relation}) has an 
interpretation as the value function of the stochastic control problem:  
\begin{align}
U(t,z) = \inf\limits_{u_s} \mathbf{\bar{E}}\!\left(\int_t^T h(z^u_s) \,ds + \frac{1}{2}\int_t^T
|u_s|^2 ds ~\Big|~ z_t = z \right).
\label{value-fun}
\end{align}
From the dynamic programming principle \cite{fleming2006}, we know that
$U(t,z)$ satisfies the following \emph{Hamilton-Jacobi-Bellman} (HJB) or \emph{dynamic programming} equation:
\begin{align}
\begin{split}
  &\frac{\partial U}{\partial t} + \min\limits_{c \in \mathbb{R}^m} \Big\{h + \frac{1}{2}
|c|^2 + (b-\sigma c) \cdot \nabla U+
\frac{1}{2\beta}\sigma\sigma^T\colon\nabla^2 U\Big\} = 0 \\
&U(T, z) = 0\,.
\end{split}
\label{hjb}
\end{align}
The latter implies that the optimal control force $\hat{u}_s$ is of feedback form and satisfies  
\begin{align}
\hat{u}_s= \sigma^T(\hat{z}_{s}) \nabla U(s, \hat{z}_{s}).
\label{optimal-u}
\end{align}

Now we estimate (\ref{var-exp}) and thus the relative error (\ref{relative-error}) for
a general control $u_s$. To this end we suppose that the probability measures $\mathbf{\bar{P}}$ and
$\mathbf{\hat{P}}$ are mutually equivalent.
Then, using (\ref{as-const}),  
we can conclude that 
\begin{align}
\exp \Big(-\beta\int_t^T h(\hat{z}_s) \,ds\Big) \hat{Z}^{-1}_t = I,
\qquad \mathbf{\bar{P}}-a.s. 
\end{align}
and therefore
\begin{equation}\label{zt-ratio-exp}
\begin{aligned}
  &\frac{1}{I^2}\mathbf{\bar{E}} \Big[\exp\Big(-2\beta\int_t^T h(z^u_s) ds\Big)
(Z_t)^{-2}\Big]  \\
= & \frac{1}{I^2}\mathbf{\bar{E}} \Big[\exp\Big(-2\beta\int_t^T h(\hat{z}_s) ds\Big)
(\hat{Z}_t)^{-2}\Big(\frac{\hat{Z}_t}{Z_t}\Big)^2\Big] 
=  \mathbf{\bar{E}} \Big[\Big(\frac{\hat{Z}_t}{Z_t}\Big)^2\Big],
\end{aligned}
\end{equation}
where by Girsanov's theorem~(\ref{girsanov-zt}), we have
\begin{align}
\Big(\frac{\hat{Z}_t}{Z_t}\Big)^2 =& 
\exp \Big(-2\beta^{1/2}\int_t^T (\hat{u}_s - u_s) dw_s -
\beta \int_t^T (|\hat{u}_s|^2 - |u_s|^2) ds\Big).
\label{zt-ratio}
\end{align}
In order to simplify (\ref{zt-ratio-exp}), we follow
\cite{ip-dupuis-multiscale} and introduce another control force
$\tilde{\bar{u}}_s$ and change
the measure again. Specifically, we choose 
$\tilde{\bar{u}}_s = 2\hat{u}_s-u_s$ and define
$\tilde{\bar{w}}_t,\mathbf{\tilde{\bar{P}}}, \tilde{\bar{Z}}_t$ as
in (\ref{girsanov-bs})--(\ref{girsanov-zt}), with $u_s$ being replaced by 
$\tilde{\bar{u}}_s$. If we now let $\mathbf{\tilde{\bar{E}}}$ denote the expectation \wrt 
$\mathbf{\tilde{\bar{P}}}$ then, using equations (\ref{zt-ratio-exp}) and (\ref{zt-ratio}), we obtain 
\begin{align}
\mathbf{\bar{E}} \Big[\Big(\frac{\hat{Z}_t}{Z_t}\Big)^2\Big] =
\mathbf{\tilde{\bar{E}}} \Big[\Big(\frac{\hat{Z}_t}{Z_t}\Big)^2 \tilde{\bar{Z}}_t^{-1}Z_t\Big]
= 
\mathbf{\tilde{\bar{E}}} \Big[\exp\Big(\beta\int_t^T |\hat{u}_s -
u_s|^2 ds\Big)\Big].
\label{likelihood-exp}
\end{align}
Roughly speaking, the last equation indicates that the relative error
(\ref{relative-error}) 
of the importance sampling estimator associated to a general control $u$ depends on 
the difference between control $u$ and the optimal control $\hat{u}$. This
relation will be further used in Section~\ref{sec-proof} to prove the upper bound for
the relative error of importance sampling estimator. \\

\section{Importance sampling of multiscale diffusions}
\label{sec-main}

Our main result in this paper concerns dynamics with two time scales.
Specifically, we consider the case when the state variable $z \in
\mathbb{R}^n$ can be split into a slow variable $x \in \mathbb{R}^k$ and
a fast variable $y \in \mathbb{R}^l$, i.e. $z=(x,y), \,k + l = n$, and 
we assume that (\ref{dynamics-1}) is of the form 
\begin{align}
  \begin{split}
    dx_s &= f(x_s,y_s) ds + \beta^{-1/2} \alpha_1(x_s,y_s) dw_s^1\\
    dy_s &= \frac{1}{\epsilon} g(x_s,y_s) ds+
   \beta^{-1/2} \frac{1}{\sqrt{\epsilon}}\alpha_2(x_s,y_s) dw_s^2 
\end{split}
\label{averaging-dynamics}
\end{align}
where $f\colon\mathbb{R}^n\rightarrow \mathbb{R}^k$, $g\colon
\mathbb{R}^n\rightarrow \mathbb{R}^l$ are smooth vector fields, 
 $\alpha_1\colon \mathbb{R}^n\rightarrow
\bR^{k\times m_1}$, $\alpha_2\colon \mathbb{R}^n\rightarrow \bR^{l\times
m_2}$  
are smooth noise coefficients and $w_s^1\in\bR^{m_1}$, $w_s^2\in\bR^{m_2}$
 are independent Wiener processes with $m_1, m_2 > 0$. The parameter $\epsilon \ll 1$ describes the
time-scale separation between processes $x_s$ and $y_s$.

Let $x\in \mathbb{R}^k$ be given and suppose that the fast subsystem 
\begin{align}
    dy_s &= \frac{1}{\epsilon} g(x,y_s) ds+
 \beta^{-1/2} \frac{1}{\sqrt{\epsilon}}\alpha_2(x,y_s) dw_s^2 ,\qquad y_0 = y
 \in \mathbb{R}^l\,,
  \label{frozen-fast-dynamics}
\end{align}
 is ergodic with a unique invariant measure whose density \wrt 
 Lebesgue measure is denoted by  $\rho_x(y)$ (see Appendix~\ref{app-2} for more details). Then it is well known that when $\epsilon \rightarrow 0$, under some mild conditions on the coefficients, the slow component of  
(\ref{averaging-dynamics}) converges in probability to the averaged
dynamics \cite{freidlin2012random,khasminskii,book_PS08,liu2010}
\begin{align}
\begin{split}
  d \widetilde{x}_s &= \widetilde{f}(\widetilde{x}_s) ds +
  \beta^{-1/2}\widetilde\alpha(\widetilde{x}_s) dw_s ,
  \quad t \le s \le T \\
  \widetilde{x}_t &= x\,,
 \end{split}
\label{dynamics-averaged}
\end{align}
where for every $x \in \mathbb{R}^k$, we have  
\begin{align}
  \widetilde{f}(x) = \int_{\mathbb{R}^l} f(x,y) \rho_x(y) \,dy, \qquad
  \widetilde{\alpha}(x)\widetilde{\alpha}(x)^T = \int_{\mathbb{R}^l} \alpha_1(x,y)\alpha_1(x,y)^T
\rho_x(y) \,dy.
\label{averaging-coeff}
\end{align}
Further define 
\begin{align}
  \widetilde{h}(x) = \int_{\mathbb{R}^l} h(x,y) \rho_x(y) \,dy\,,
  \label{averaging-h}
\end{align}
and consider the averaged value function 
\begin{align}
U_0(t,x) = \inf\limits_{u} \mathbf{\bar{E}} \Big\{\int_t^T
  \widetilde{h}(\widetilde{x}^u_s) \,ds + \frac{1}{2}\int_t^T
|u_s|^2 ds \Big\},
\label{oc-reduced}
\end{align}
where $\widetilde{x}^u_s \in \mathbb{R}^k$ is the solution of 
\begin{align}
\begin{split}
  d \widetilde{x}^u_s &= \widetilde{f}(\widetilde{x}^u_s) ds -
  \widetilde{\alpha}(\widetilde{x}^u_s) u_s ds +
  \beta^{-1/2}\widetilde\alpha(\widetilde{x}^u_s) dw_s , \quad t \le s \le T \\
  \widetilde{x}^u_t &= x\,.
 \end{split}
\label{dynamics-averaged-controlled}
\end{align}

The idea of using suboptimal controls for importance sampling of multiscale systems such as (\ref{averaging-dynamics}) is to use the solution of the limiting control problem (\ref{oc-reduced})--(\ref{dynamics-averaged-controlled}) to construct an asymptotically optimal control of the form
 \begin{align}
\hat{u}^{0}_{s} = \left(\alpha_{1}^T(x^{u}_s,y^{u}_{s})\nabla_x
U_0(x^{u}_s), 0\right)\,,
\label{sub-control-average}
\end{align}
for the full system. Comparing (\ref{sub-control-average}) to the optimal control force (\ref{optimal-u}),
  this means that we construct the control for the slow variable by using the
  averaged value function $U_0$ in (\ref{oc-reduced}) and leave the fast variable uncontrolled. 
  Notice that control (\ref{sub-control-average}) has also been suggested in
  \cite{ip-kostas1} for more general dynamics with a general subsolution of
  the HJB equation.
\begin{remark}
Another variant of a suboptimal control would be 
 \begin{align}
\hat{u}^{0}_{s} = \left(\widetilde{\alpha}^T(x^{u}_s)\nabla_x U_0(x^{u}_s), 0\right),
\label{sub-control-average-2}
\end{align}
where the $x$-component is the optimal control of the averaged system
(\ref{oc-reduced})--(\ref{dynamics-averaged-controlled}). The advantage of
using (\ref{sub-control-average-2}) rather than (\ref{sub-control-average}) is
that the fast variables do not need to be explicitly known or observable in
order to control  the system. In the following we will assume that $\alpha_{1}$ is independent of $y$, in which case (\ref{sub-control-average}) and (\ref{sub-control-average-2}) coincide (see Assumption \ref{assumption-3}). 
\end{remark}

\subsection{Main result}
\label{sub-sec-main-1}

\noindent
Our main assumptions are as follows. 

\begin{assumption}
  $f, g, h, \alpha_1, \alpha_2$ are $C^2$ functions, with derivatives that are uniformly bounded by a 
  constant $C>0$. $\alpha_1, \alpha_2$ and $h$ are bounded. Furthermore, there
  exist constant $C_1>0$, such that 
  \begin{align*}
     \zeta^T \alpha_2(x,y)\alpha_2(x,y)^T \zeta \ge C_1 |\zeta|^2\,,
  \end{align*}
 $\forall x \in \mathbb{R}^k, \zeta,y \in \mathbb{R}^l$.
    \label{assumption-1}
\end{assumption}
\begin{assumption} $\exists \lambda > 0$, such that  $\forall x \in
  \mathbb{R}^k, y_1, y_2 \in \mathbb{R}^l$, we have 
  \begin{align}
    \langle g(x, y_1) - g(x, y_2), y_1 - y_2\rangle + \frac{3}{\beta}\|\alpha_2(x,y_1) -
    \alpha_2(x,y_2)\|^2 \le -\lambda |y_1 - y_2|^2, \quad 
  \end{align}
  where $\|\cdot\|$ denotes the Frobenius norm.
    \label{assumption-2}
\end{assumption}
\begin{assumption}
  $\alpha_1$ and $h$ do not depend on $y$.
  \label{assumption-3}
\end{assumption}

\begin{remark}
  \begin{enumerate}
    \item
      Assumption~\ref{assumption-1} implies the coefficients are Lipschitz
      functions. In particular, it holds that 
      $|f(x,y)| \le C(1 + |x| + |y|)$, $\forall x \in \mathbb{R}^k,
      y\in\mathbb{R}^l$ (similarly for the other coefficients).
    \item
      For $\widetilde{f}$ given by (\ref{averaging-coeff}), 
      Lemma~\ref{lemma-stationary} in
      Appendix~\ref{app-2} implies that $\widetilde{f}$ is Lipschitz
      continuous. Unlike in \cite{liu2010}, we do not assume that $f$ is bounded.
    \item
  Assumption~\ref{assumption-2} guarantees that the fast dynamics is
  exponentially mixing. As we study the asymptotic solution of (\ref{averaging-dynamics}) as $\epsilon \rightarrow 0$ at fixed noise intensity, the inverse temperature $\beta$ can be absorbed into the coefficients $\alpha_1$, $\alpha_2$ and $h$. In
  Section~\ref{sec-proof}, we will therefore assume $\beta = 1$, in which case
  Assumption~\ref{assumption-2} implies that 
  \begin{align}\label{mixing}
    \langle \nabla_{y}g\, \xi, \xi \rangle  + 3\|\nabla_{y}\alpha_2\, \xi\|^2 \le -\lambda
    |\xi|^2, \quad \forall y, \xi \in \mathbb{R}^l, x \in \mathbb{R}^k\,,
  \end{align}
  where $\nabla_y\alpha_2 \xi$ is an $l \times m_2$ matrix with components 
  \begin{align}
    \big(\nabla_y\alpha_2 \xi\big)_{ij} = \sum_{r=1}^l
    \frac{\partial (\alpha_2)_{ij}}{\partial y_r} \xi_r\,,\qquad 1 \le i \le
    l\,, \quad 1 \le j \le m_2\,.
  \end{align}
    Combining this with Assumption~\ref{assumption-1}, we have 
  \begin{align}
      & \langle g(x,y), y\rangle + \frac{3}{2}\|\alpha_2(x,y)\|^2 \notag \\
    \le &
    \langle g(x,y) - g(x,0), y\rangle + \langle g(x,0), y\rangle +
    3\|\alpha_2(x,y) - \alpha_2(x,0)\|^2 + 3\|\alpha_2(x,0)\|^2\notag \\
    \le & -\frac{\lambda}{2}|y|^2 + C(|x|^2 + 1)\,,\qquad \forall x \in
    \mathbb{R}^k, y \in \mathbb{R}^l\,. \label{mixing-1}
  \end{align}
  The constant $3$ in (\ref{mixing}) is not optimal, but it will simplify matters later on.
  \end{enumerate}
  \label{rmk-1}
\end{remark}

Now we are ready to state our main result, whose proof will be given in
Section~\ref{sec-proof}.
\begin{theorem}
  Suppose Assumptions \ref{assumption-1}--\ref{assumption-3} hold, and consider
  the importance sampling method for computing (\ref{exp-I}) with dynamics
  (\ref{averaging-dynamics}) and control $\hat{u}^0$ as given by
  (\ref{sub-control-average}). Then, for $\epsilon \ll 1$, the relative error
  (\ref{relative-error}) of the importance sampling 
  estimator satisfies
  \begin{align}
    \mbox{RE}_{\hat{u}^0}(I) \le C\epsilon^{\frac{1}{8}}, \notag
  \end{align}
  where the constant $C > 0$ is independent of $\epsilon$.
  \label{main_thm}
\end{theorem}

\subsection{Formal expansion by asymptotic analysis}
\label{sub-sec-main-2}

\noindent
The proof of Theorem~\ref{main_thm} in Section~\ref{sec-proof} is relatively long and technical, which is why we shall give a 
formal derivation of (\ref{sub-control-average}) first. The idea is to identify the suboptimal control $\hat{u}^{0}$ 
as the leading term of the optimal control using formal asymptotic expansions \cite{asymptotic_analysis,book_PS08}. 
To this end, let $U^{\epsilon}$ denote the solution of (\ref{hjb}), for which we seek an asymptotic expansion in 
powers of $\epsilon$. Further let $\phi^{\epsilon}(t,x,y) = \exp(-\beta U^{\epsilon})$. From the dual relation
(\ref{dual-relation}), we know that $\phi^{\epsilon}$ is the expectation (\ref{exp-I}) we want to
compute. By the Feynman-Kac formula, we have
\begin{align}
\begin{split}
&\frac{\partial \phi^{\epsilon}}{\partial t} + \cL
\phi^{\epsilon} - \beta h\phi^{\epsilon} = 0\,, \quad 0 \le t \le T
\\
&\phi^{\epsilon}(T,x,y) = 1\,,
\end{split}
\label{linear-compact-hjb}
\end{align}
where $\cL =\epsilon^{-1} \cL_0 + \cL_1$ is the infinitesimal generator of
process (\ref{averaging-dynamics}), with 
\begin{align}
  \begin{split}
   \cL_0 & = g \cdot \nabla_y + \frac{1}{2\beta} \alpha_2\alpha_2^T\colon \nabla_y^2\\
\cL_1 & = f \cdot \nabla_x + \frac{1}{2\beta}\alpha_1\alpha_1^T\colon\nabla^2_x\,.
\end{split}
\label{generator}
\end{align}

Now consider the expansion $\phi^{\epsilon} = \phi_0 +
\epsilon \phi_1 + \ldots$ of $\phi^{\epsilon}$ in powers of $\epsilon$. Plugging it into (\ref{linear-compact-hjb}) and
comparing different powers of $\epsilon$, we obtain : 
\begin{align}
&\frac{\partial \phi_0}{\partial t} + \cL_0 \phi_1 + \cL_1 \phi_0 - \beta h\phi_0 = 0, \label{coeff-1} \\
& \cL_0\phi_0 = 0\,. \label{coeff-2} 
\end{align}
By the assumption that the fast dynamics (\ref{frozen-fast-dynamics}) are ergodic for every $x \in \mathbb{R}^k$ with 
unique invariant density $\rho_x(y)$, it follows that $\rho_x(y) > 0$ is the unique solution to the linear equation 
$\cL_0^* \rho_x = 0$ with $\int_{\mathbb{R}^l} \rho_x(y) dy = 1$. Here
$\cL_0^*$ is the adjoint operator of $\cL_0$ with respect to the standard 
scalar product in the space $L^{2}(\mathbb{R}^l)$. Hence we can conclude from (\ref{coeff-2}) that $\phi_0
= \phi_0(t,x)$ is independent of $y$. Integrating both sides of (\ref{coeff-1}) against $\rho_x(y)$, we obtain a closed equation for $\phi_0$: 
\begin{align}
\frac{\partial \phi_0}{\partial t} + \widetilde{\cL} \phi_0 - \beta \widetilde{h}\phi_0 = 0 
\label{averaged-fk}
\end{align}
with
\begin{align}
\begin{split}
&\widetilde{\cL} = \widetilde{f}(x) \cdot \nabla_x +
\frac{\widetilde{\alpha}(x)\widetilde{\alpha}(x)^T}{2\beta}:\nabla^2_x\,,
\end{split}
\label{reduced-coeff}
\end{align}
and $\widetilde{h}, \widetilde{f}, \widetilde{\alpha}$ as given by  
(\ref{averaging-coeff}) and (\ref{averaging-h}).

Notice that $\widetilde{\cL}$ is the infinitesimal generator of the averaged 
dynamics (\ref{dynamics-averaged}). Again by the Feynman-Kac formula, the solution to (\ref{averaged-fk}) is 
recognized as the conditional expectation  
\begin{align}
\phi_0(t,x) = \mathbf{E} \Big[\exp \Big(-\beta\int_t^T
\widetilde{h}(\widetilde{x}_s)\,
ds\Big)~\Big|~\widetilde{x}_t = x\Big]\,
\label{expectation-averaged}
\end{align}
of the averaged path functional over all realizations of the averaged dynamics (\ref{dynamics-averaged}) starting at 
$\widetilde{x}_{t}=x$.  Recalling $U^\epsilon = -\beta^{-1}\ln \phi^\epsilon$, it follows that $U^{\epsilon}$ has the expansion 
\begin{align}
  U^\epsilon = -\beta^{-1} \ln (\phi_0 + \epsilon \phi_1 + o(\epsilon)) =
 -\beta^{-1} \ln \phi_0
 - \beta^{-1} \frac{\phi_1}{\phi_0}\epsilon + o(\epsilon).
 \label{u-expansion}
\end{align}
Combining (\ref{u-expansion}) with (\ref{expectation-averaged})
and the dual relation (\ref{dual-relation}), we conclude that $U_0$ in
(\ref{oc-reduced}) satisfies $U_0 = -\beta^{-1}\ln \phi_0$ and is the leading
term of $U^\epsilon$ in expansion (\ref{u-expansion}). Finding the corresponding expression for the optimal control is now straightforward: Setting $\hat{u}_{s} =
(\hat{u}_{s,1}, \hat{u}_{s,2})\in \mathbb{R}^{m_1} \times \mathbb{R}^{m_2}$, the relation (\ref{optimal-u}) between the
optimal feedback control and the value function yields  
\begin{align}
\begin{split}
  &\hat{u}_{s,1} =
  \alpha_1^T\nabla_x U_0 + \mathcal{O}(\epsilon) = -\beta^{-1}
  \frac{\alpha_1^T\nabla_x \phi_0}{\phi_0} + \mathcal{O}(\epsilon) , \\
&\hat{u}_{s,2} = \frac{\alpha_2^T}{\sqrt{\epsilon}} \nabla_y U^\epsilon =
  \mathcal{O}(\epsilon^{\frac{1}{2}})\,,
\end{split}
\label{lead-control}
\end{align}
where all functions are evaluated at $(s,x_{s}^{\hat{u}},y_{s}^{\hat{u}})$. 

The last equation shows that (\ref{sub-control-average}) appears to be the leading term of the 
optimal control force as $\epsilon \rightarrow 0$. Reiterating the argument given in
Section~\ref{sec-setup}, we expect (\ref{sub-control-average}) to be a 
reasonably good approximation of the exact control force that gives rise to sufficiently accurate importance sampling estimators of 
(\ref{exp-I}) in the asymptotic regime $\epsilon \ll 1$.

As for the corresponding numerical algorithm, our derivations suggest that one possible strategy for finding 
good control forces for importance sampling is to first compute $U_0$ from
(\ref{oc-reduced}) or (\ref{expectation-averaged}), which corresponds to a
low-dimensional stochastic optimal control problem, and then to construct the control
force as in (\ref{sub-control-average}) to perform importance sampling. The numerical strategy 
will be discussed in Section~\ref{sec-examples}, along with some details regarding the numerical implementation.

\begin{remark}
  A closely related variant of the slow-fast dynamics
  (\ref{averaging-dynamics}) is homogenization problems that exhibit more than two time scales~\cite{book_PS08}.
  Although a rigorous treatment of multiscale diffusions with three or more time scales is beyond the scope of this 
  work, we stress that the formal asymptotic argument carries over directly. 
  See \cite{ip-dupuis-multiscale,ip-kostas1,ip-kostas3} for large deviations
  and importance sampling studies of related dynamics.
  \label{rmk-homogenization}
\end{remark}

\section{Numerical example}
\label{sec-examples}

In this section, we study a numerical example and discuss some algorithmic issues 
related to the calculation of the suboptimal control force
(\ref{sub-control-average}) as proposed in Section~\ref{sec-main}. 
The dynamics we considered here is described by the two-dimensional SDE
\begin{align}
  \begin{split}
    dx_s &= -\frac{\partial V(x_s,y_s)}{\partial x} ds + \beta^{-1/2} dw^1_s \\
  dy_s &= -\frac{1}{\epsilon}\frac{\partial V(x_s,y_s)}{\partial y} ds +
    \beta^{-1/2} \frac{1}{\sqrt{\epsilon}} dw^2_s\,, 
\end{split}
\label{dynamics-ex1}
\end{align}
where $(x_s, y_s) \in \mathbb{R}^2$, $w_s = (w_s^1, w_s^2)$ is a two-dimensional
Wiener process and $\beta, \epsilon > 0$. The potential $V(x,y) = V_1(x) + V_2(x,y)$ is defined as
\begin{align}
  \begin{split}
  V_1(x) =& \frac{1}{2}\big(1 - \eta(x) - \eta(-x)\big) \cos\Big(\frac{4\pi
  x}{5}\Big) + 3 \eta(x) (x-1)^2 + 3\eta(-x) (x+1)^2, \\
  V_2(x,y) =& \frac{1}{2} (x - y)^2 \,,
  \end{split}
  \label{pot-def}
\end{align}
with $\eta(x) = e^{-\frac{1}{x}}$ if $x>0$, and $\eta(x) = 0$ otherwise.
The function $V_1(x)$ is a smooth bistable potential that has two ``wells'' centered around $x=-1$ and
$x=1$. 
As in (\ref{exp-I}), we aim at computing the expectation
\begin{align}
  I = \mathbf{E}\!\left[\exp\left(-\beta\int_0^T h(x_s) ds\right)~\bigg|~ x_0 =
  -1,\,
  y_0 = 0 \right],
  \label{target-ex1}
\end{align}
where 
\begin{align}
h(x) = 
\eta\Big(\frac{x+2}{w}\Big) \eta\Big(\frac{4-x}{w}\Big) (x-1)^2 + 10\Big[2 -
\eta\Big(\frac{x+2}{w}\Big) -\eta\Big(\frac{4-x}{w}\Big)\Big] \,,
\label{fun-h-def}
\end{align}
with parameter $w=0.02$.
The graphs of the functions $\eta, V_1$ and $h$ are shown in Figure~\ref{fig-ex1-data}.
Notice that the auxiliary function $\eta$ is introduced in (\ref{pot-def}) and
(\ref{fun-h-def}) in order to guarantee that
Assumption~\ref{assumption-1}-\ref{assumption-3} of Theorem~\ref{main_thm} in Section~\ref{sec-main} are satisfied. More discussions on these
assumptions can be found in the section of Introduction and Conclusions.

Using the specific form of potential $V$, we can explicitly compute the invariant measure of the fast dynamics $y_s$ in (\ref{dynamics-ex1}), which for each fixed $x \in \mathbb{R}$ has the Lebesgue density 
\begin{align}
  \rho_x(y) \propto e^{-\beta(x-y)^2}\,.
\end{align}
Recalling the discussion in Section~\ref{sec-main}, especially (\ref{dynamics-averaged}) and (\ref{averaging-coeff}), 
we conclude that the averaged dynamics is a one-dimensional diffusion in a double well
potential 
\begin{align}
  d\widetilde{x}_s &= -V'_1(\widetilde{x}_s) ds + \beta^{-1/2} dw_s\,,
  \label{ex1-reduced}
\end{align} 
where the potential $V_1$ is given in (\ref{pot-def}) and $w_s$ is a one-dimensional Wiener process.

Before we proceed, we shall briefly discuss the potential difficulties to
compute (\ref{target-ex1}) with the standard Monte Carlo method, which is
mainly due to the inherent metastability of the system, even for moderate
values of $\beta$. To this end, notice that, in the path space, the exponential integrand in (\ref{target-ex1})
is peaked around trajectories which spend a large portion of time at the minimum
of $h$, which is located around $x = 1$ (Figure~\ref{fig-ex1-h}). 
But in order to get close to the state $x = 1$, trajectories starting from $x_0 = -1$ need to cross the energy barrier
$\Delta V_1(\approx V_1(0) - V_1(-1))$ of $V_1$ (Figure~\ref{fig-ex1-pot}). 
The probability of these barrier-crossing trajectories is roughly of order $\exp(-\beta
\Delta V_{1})$ when $\beta \Delta V_1$ is large. Combining these facts, we 
expect that the rare barrier crossing events play an important role for
computing (\ref{target-ex1}). And standard Monte Carlo method will be inefficient in such a situation due to
insufficient sampling of these rare events (cf.~the discussion in Section~\ref{sec-intro}).\\

\textbf{Computation of the suboptimal estimator based on the averaged equation. }
Now let us consider the method outlined in Subsection \ref{sub-sec-main-1}. 
In accordance with (\ref{averaged-fk}), the conditional expectation $\phi_0$ solves the linear backward evolution equation
\begin{align}
  \begin{split}
  &\frac{\partial \phi_0}{\partial t} + \widetilde{\mathcal{L}} \phi_0 - \beta
    \widetilde{h} \phi_0 = 0 \\
  &\phi_0(T,x) = 1,
\end{split}
\label{averaged-fk-ex1}
\end{align}
with 
\begin{align}
  \widetilde{\cL} = - V_1' \frac{\partial}{\partial x} +  \frac{1}{2\beta}
  \frac{\partial^2}{\partial^2 x}, \quad
\widetilde{h}(x) = h(x)\,.
\end{align}
The equation for $\phi_{0}$ is one-dimensional (in space), and can be solved
by standard grid-based method. For instance, using Rothe's method, we can first discretize (\ref{averaged-fk-ex1}) in time, which yields 
\begin{align}
  \Big(\frac{1}{\Delta t} - \widetilde{\mathcal{L}}\Big) \phi_0^{j} = \Big(\frac{1}{\Delta t} -
  \beta h\Big) \phi_0^{j + 1}\,,\quad j = 0, 1,\cdots, m - 1
  \label{discrete-fk-ex1}
\end{align}
where $\phi_0^{j}$ denotes the approximation of $\phi_0$ at time $t_j =
j\Delta t$, $j = 0, 1, \cdots, m$ with time step size $\Delta t = T/m$. Equation 
(\ref{discrete-fk-ex1}) is then further discretized in space using the  
structure-preserving finite volume method described in \cite{Latorre2010}.
Starting from $\phi_0^m\equiv 1$, we can obtain all $\phi_0^j$ for $j=m-1,
m-2, \cdots, 1$ by solving (\ref{discrete-fk-ex1}) backwardly.

After obtaining $\phi_0$, we can compute the feedback control force
(\ref{sub-control-average}) as 
\begin{align}
  \hat{u}^0_{s} = \left(-\beta^{-1}\frac{\partial_x
  \phi_0(s,x^u_s)}{\phi_0(s,x^u_s)}, 0\right),
  \label{control-ex1}
\end{align}
when system's state is at $(x_s^u, y_s^u)$ at time $s$.
Plugging the last expression into (\ref{dynamics-ex1}) then yields the
controlled dynamics (also see (\ref{dynamics-2}))
\begin{align}
  \begin{split}
    dx^u_s &= -\frac{\partial V(x^u_s,y^u_s)}{\partial x} ds 
    +\beta^{-1} \frac{\partial_x \phi_0(s,x^u_s)}{\phi_0(s,x^u_s)} ds + \beta^{-1/2} dw^1_s \\
  dy^u_s &= -\frac{1}{\epsilon}\frac{\partial V(x^u_s,y^u_s)}{\partial y} ds + \beta^{-1/2} \frac{1}{\sqrt{\epsilon}} dw^2_s, 
\end{split}
\label{dynamics-ex1-changed}
\end{align}
which will be employed to sample (\ref{target-ex1}) using the reweighted estimator (\ref{ip_mc}).

\begin{figure}
\centering
\begin{tabular}{ccc}
  \subfigure[]{\includegraphics[width=4.5cm]{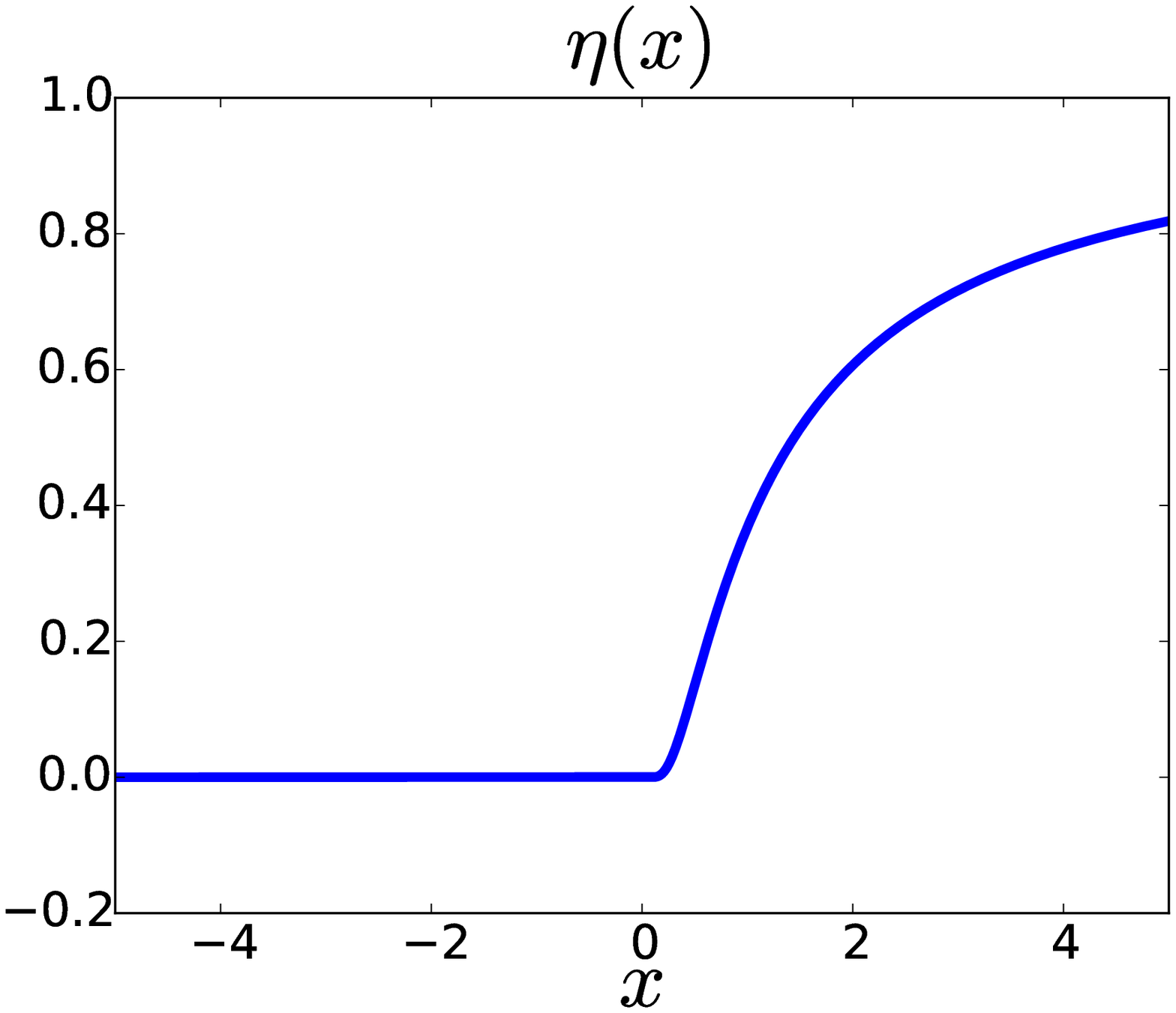}\label{fig-ex1-eta}}
  \subfigure[]{\includegraphics[width=4.5cm]{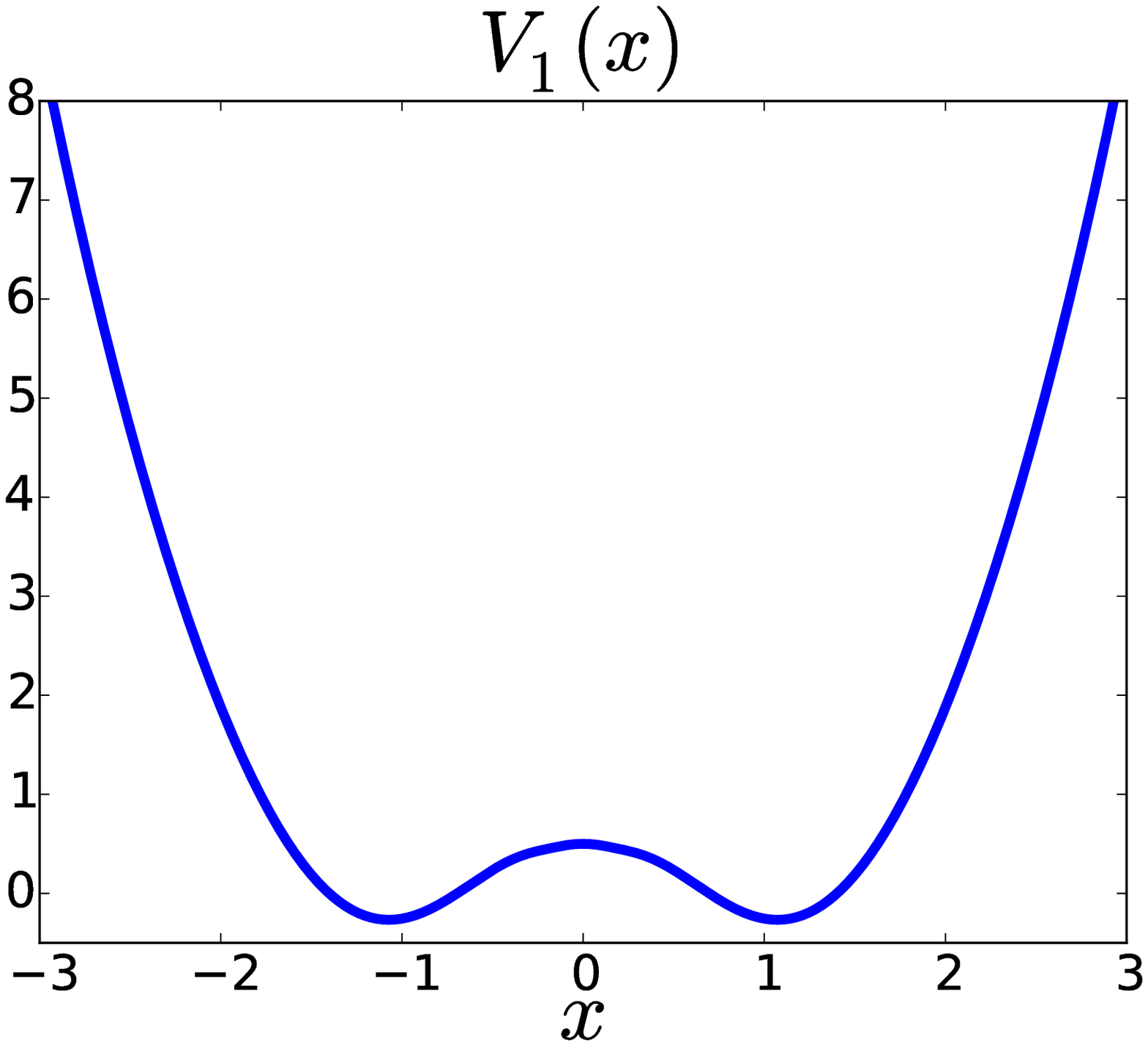}\label{fig-ex1-pot}}
  &
  \subfigure[]{\includegraphics[width=4.5cm]{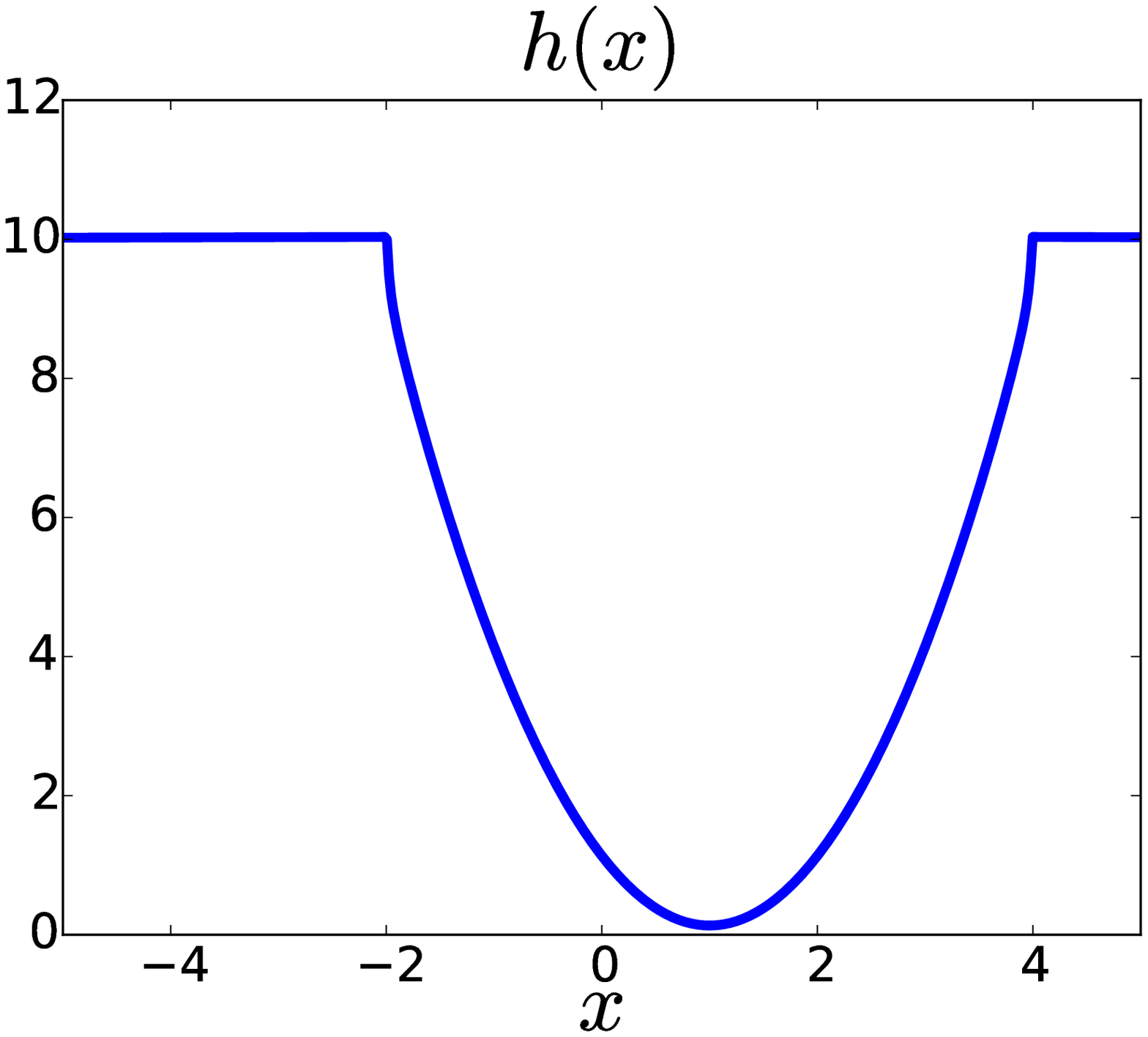}\label{fig-ex1-h}}
\end{tabular}
\caption{(a) Function $\eta(x)$ used to define potential $V_1$. (b) Double
  well potential $V_1(x)$. (c) Function $h$ in (\ref{target-ex1}).
\label{fig-ex1-data}
}
\end{figure}

\textbf{Numerical results. }
Now we turn to the numerical results.
Table \ref{tab-1} shows the numerical results of the Monte Carlo method with the above 
importance sampling strategy, i.e. (\ref{dynamics-ex1-changed}), which should be compared to Table \ref{tab-3} that shows the result of standard 
Monte Carlo method. For both the weighted and unweighted estimates, the sample size was
set to $N = 10^4$ trajectories of length $T=1$ with time step $\Delta t\le
10^{-7}$ that is chosen small enough to remove discretization bias. The
control (\ref{control-ex1}) was obtained by computing $\phi_0$ from
(\ref{discrete-fk-ex1}) on a grid of size $n_x$. For comparison, we have
computed a reference importance sampling Monte-Carlo solution (``exact'' mean value) based on
$N=10^5$ independent realizations that is displayed in Table \ref{tab-1} in
the column with label ``$I$''. 
The performance of the Monte Carlo methods can be evaluated based on the
variance (\ref{variance}) and the relative error (\ref{relative-error}). In
our numerical study, they are estimated from the sampled trajectories as 
\begin{align}
  \begin{split}
  \mathrm{Var}_{u} I =& \frac{1}{N}\sum_{i=1}^{N}
  \Big[\Big(\exp\Big(-\beta\int_0^T h(x^{u,i}_s) \,ds\Big) (Z^{u,i}_t)^{-1}
  \Big) - I_N\Big]^2\,, \\
  \mbox{RE}_u(I) =& \frac{\sqrt{\mathrm{Var}_{u} I}}{I_N}, 
\end{split}
  \label{sv-re}
\end{align}
where $x_s^{u,i}$ is the $i$-th trajectories, $1\le i \le N$, $I_N$ is the
estimator (\ref{ip_mc}) of $I$, and $u$ denotes the control force. See
Section~\ref{sec-setup} for details.
Furthermore, in order to illustrate the actual effect of the control force, we monitor the barrier crossing events with $x_s \ge 0$ for some $0 <
s \le T=1$ and let $R_c$ record the ratio of trajectories which cross the barrier among all the trajectories. 

In Table~\ref{tab-1}, for different values of $\beta$, we can see that 
the relative error of the importance sampling estimator becomes smaller as $\epsilon$ decreases from $0.1$ to $0.001$. 
This indicates that the importance sampling estimator performs better and
better when $\epsilon$ deceases and therefore is accordance with the conclusion of
Theorem~\ref{main_thm} in Section~\ref{sec-main}.

It is also worth making a comparison of both the importance sampling estimator and the standard Monte Carlo estimator. 
For the importance sampling estimator (Table~\ref{tab-1}), we observe that both the mean values
and the variances, estimated with $N=10^4$ trajectories, are stable after we ran several times and are close to the results estimated with $N=10^5$
trajectories, which we take as the ``exact'' mean value. For the standard Monte Carlo method (Table~\ref{tab-3}),
at $\beta = 1$, while it gives acceptable mean values, the sample variances
(and the relative errors) are larger compared to the importance sampling estimator. 
For $\beta=5,\, 8$, the results of standard Monte Carlo method drift 
away from the ``exact'' mean values and show a significant bias. These results indicate that the standard Monte Carlo method is inefficient or useless in this situation.

The above results can be better understood if we record the
barrier-crossing events during time $[0,1]$. These events are
related to the metastability of the system and become rare for $\beta=5$
and $\beta=8$.  In the ``$R_c$'' column of Table~\ref{tab-3}, we see that
very few trajectories can cross the energy barrier when $\beta = 5$, 
and it becomes even rarer when $\beta$ is further increased to $\beta = 8$, at which no barrier-crossing
trajectories are sampled with $N=10^4$ trajectories. 
This observation reveals the fact that 
crossing the energy barrier is a rare event (in the uncontrolled system) due to system's metastability at moderate
temperature. And it also explains why the estimations of the mean values are
largely underestimated by the standard Monte Carlo method (compare
Table~\ref{tab-1} and Table~\ref{tab-3}). On the other hand, as shown in
``$R_c$'' column of Table~\ref{tab-1}, the barrier-crossing events are much
better sampled by the importance sampling estimator. 
Figure \ref{fig-2} shows the control force (\ref{control-ex1}) as a
function of $x$ and time $s$ for various values of $\beta$. We clearly 
observe that the control acts against the energy barrier (blue region) and
assists the slow variable $x_s$ of the system to transit from $x=-1$ to $x=1$. 

We conclude this section with a couple of comments on numerical issues. 
\begin{remark}
  \begin{enumerate}
    \item
      It is necessary to solve the averaged equation (\ref{oc-reduced}) for
      $U_0$, or equivalently (\ref{averaged-fk}) for $\phi_0$, in order to compute control (\ref{sub-control-average}). 
      Solving $\phi_0$ from (\ref{averaged-fk}) may be relatively easy because the equation is linear. 
      Furthermore, since equation (\ref{averaged-fk}) doesn't involve the small
      parameter $\epsilon$ any more, it can be solved on a coarser grid and
      the numerical computation is not expensive.
    \item
      In our example, the probability density $\rho_x(y)$ can be solved
      analytically and used to obtain averaged dynamics
      (\ref{dynamics-averaged}) or  (\ref{ex1-reduced}). In general, the coefficients
      (\ref{averaging-coeff}) of the averaged dynamics
      (\ref{dynamics-averaged}) could be numerically computed from the time
      integration of the fast subsystem (\ref{frozen-fast-dynamics}). See
      Chapter~$10$\,-$11$ of \cite{book_PS08} and also \cite{analysis_SDE} for more details.
    \item
      In principle, the method described above for solving linear PDE (\ref{averaged-fk-ex1}) 
      is computationally applicable when the dimension $k$ of system's slow variables $x$ is smaller or equal to $3$. 
In certain cases, however, the slow dynamics may still be higher dimensional,
and alternatives to the direct numerical discretization are needed. We refer to
the Conclusions for further discussions of this issue.
  \end{enumerate}
\end{remark}

\begin{table}[h]
  \centering
  \caption{Numerical results for importance sampling Monte Carlo method with $T=1.0$.
  Columns $I$ and $I_N$ are the mean values computed with $N=10^5$ (``exact'')
  and $N=10^4$ trajectories, respectively. Columns $\mathrm{Var}_u I, \mbox{RE}_u(I)$ 
  display the variance and the relative error defined in 
  (\ref{variance}) and (\ref{relative-error}) estimated from trajectories as in (\ref{sv-re}). Column $R_{c}$ shows 
  the ratio of the trajectories that have crossed the potential barrier.\label{tab-1}}
  \begin{tabular}{|c|c|c|c|c|c|c|c|c|}
    \hline
    $\beta$ & $\epsilon$ & $n_x$ & $\Delta t$ &$I$ & $I_N$
    &$\mathrm{Var}_u I$ & $\mbox{RE}_u(I)$ & $R_{c}$\\
    \hline
    \multirow{3}{*}{$1.0$} & $0.1$ & \multirow{3}{*}{$2000$} & $1.0 \times
					      10^{-7}$ &$3.52 \times 10^{-2}$& $3.54\times 10^{-2}$ & $1.5 \times 10^{-4}$ & $0.35$ & $6.5 \times 10^{-1}$  \\
						& $0.01$ & &$1.0 \times
			10^{-7}$  &$3.12 \times 10^{-2}$& $3.12\times 10^{-2}$ & $1.5 \times 10^{-5}$ & $0.12$ & $6.3 \times 10^{-1}$ \\
					       &$0.001$ & &$1.0 \times
			       10^{-8}$ &$3.09 \times 10^{-2}$ & $3.09 \times 10^{-2}$ & $1.5 \times 10^{-6}$ & $0.04$ & $6.2 \times 10^{-1}$  \\
    \hline
    \multirow{3}{*}{$5.0$} & $0.1$ & \multirow{3}{*}{$5000$} & $1.0 \times
					       10^{-7}$&$3.82 \times 10^{-8}$ &$3.81 \times 10^{-8}$ & $3.5 \times 10^{-15}$ &$1.55$ & $8.1 \times 10^{-1}$  \\
				  & $0.01$ & & $1.0 \times
		   10^{-7}$&$1.60\times 10^{-8}$ & $1.62\times 10^{-8}$ & $4.9 \times 10^{-17}$ & $0.43$ & $7.6 \times 10^{-1}$   \\
					 &$0.001$ &&$1.0 \times 10^{-8}$&$1.47
			   \times 10^{-8}$& $1.47 \times 10^{-8}$ & $3.7 \times 10^{-18}$ & $0.13$ & $7.6 \times 10^{-1}$  \\
    \hline
    \multirow{3}{*}{$8.0$} & $0.1$ & \multirow{3}{*}{$8000$} & $1.0 \times 10^{-7}$&$1.59\times 10^{-12}$&$1.47 \times 10^{-12}$ & $1.1 \times 10^{-23}$& $2.26$ & $8.9 \times 10^{-1}$ \\ 
						 & $0.01$ & &$5.0 \times
				  10^{-8}$&$3.68 \times 10^{-13}$ & $3.68 \times 10^{-13}$ & $4.9 \times 10^{-26} $&$0.60$ & $8.7 \times 10^{-1}$  \\
	      & $0.001$ &&$1.0 \times 10^{-8}$&$3.18 \times 10^{-13}$ &
	      $3.18\times 10^{-13}$ & $3.2 \times 10^{-27} $&$0.18$ & $8.7 \times 10^{-1}$  \\
    \hline
  \end{tabular}
\end{table}

\begin{table}[h]
  \centering
  \caption{Numerical results for standard Monte Carlo method ($u=0$). The
  labels have the same meaning as in Table \ref{tab-1}. \label{tab-3}}
  \begin{tabular}{|c|c|c|c|c|c|c|c|}
    \hline
    $\beta$ & $\epsilon$ & $\Delta t$ & $I_N$ & $\mathrm{Var}_u I$ & $\mbox{RE}_u(I)$ & $R_{c}$\\
    \hline
    \multirow{3}{*}{$1.0$} & $0.1$ & $1.0 \times 10^{-7}$ & $3.58 \times
				 10^{-2}$& $4.3 \times 10^{-3}$ & $1.83$ & $1.9 \times 10^{-1}$ \\
		    &  $0.01$ & $1.0 \times 10^{-7}$& $3.27\times 10^{-2}$ &
    $3.9 \times 10^{-3}$ & $1.91$ & $1.8 \times 10^{-1}$ \\
      & $0.001$ & $1.0 \times 10^{-8}$&  $3.14\times 10^{-2}$ & $3.4 \times 10^{-3}$ & $1.86$ & $1.8 \times 10^{-1}$\\
    \hline
    \multirow{3}{*}{$5.0$} & $0.1$ & $1.0 \times 10^{-7}$& $2.27\times
				10^{-8}$ & $6.3 \times 10^{-13}$ & $34.97$ & $3.0 \times 10^{-4}$\\
		     & $0.01$ & $1.0 \times 10^{-7}$ & $2.98\times 10^{-9}$&
		     $6.4 \times 10^{-16}$ & $8.49$ & $0$\\
				   & $0.001$ & $1.0 \times 10^{-8}$& $3.61 \times 10^{-9}$& $6.8 \times 10^{-15}$ & $22.84$ & $1.0 \times 10^{-4}$ \\
    \hline
    \multirow{3}{*}{$8.0$} & $0.1$ & $1.0 \times 10^{-7}$& $3.68 \times
				10^{-14}$ & $1.1 \times 10^{-24}$ & $28.50$ & $0$\\
						& $0.01$ & $5.0 \times
					10^{-8}$& $1.87 \times 10^{-14}$ &$3.8
		 \times 10^{-25}$ & $32.96$ & $0$\\
      & $0.001$ &$1.0 \times 10^{-8}$ & $2.01 \times 10^{-14}$&$4.4 \times
      10^{-25}$ & $33.00$ & $0$ \\
    \hline
  \end{tabular}
\end{table}

\begin{figure}
\centering
\includegraphics[width=12cm]{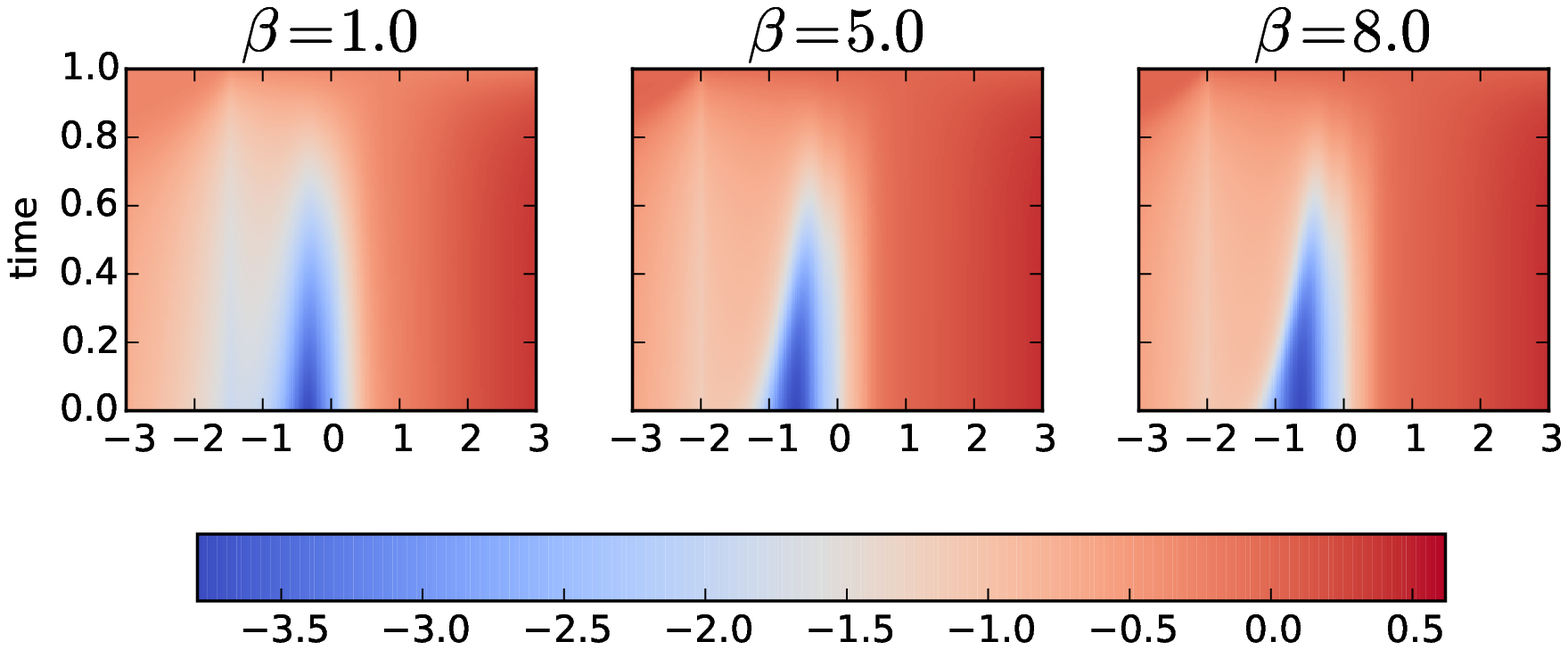}
\vspace{-1cm}
\caption{$x$-component of control force $\hat{u}^0_s$ defined in
  (\ref{control-ex1}) for different $\beta$ as a function of $x$ and $s$. 
\label{fig-2}}
\end{figure}

\section{Proof of the main result}
\label{sec-proof}

In this section, we prove our main result, Theorem~\ref{main_thm} in Section~\ref{sub-sec-main-1}.
Since the parameter $\beta$ is fixed, it can be absorbed into coefficients  $\alpha_1$ and $\alpha_2,
h$, and we can assume $\beta = 1$ without loss of generality.
Also recall that $\|\cdot\|$ denotes the Frobenius norm of matrices and
$|\cdot|$ is the Euclidean norm of vectors or the absolute value of a scalar. 

Our analysis is based on the solution $\phi^\epsilon$ of the linear backward evolution equation 
(\ref{linear-compact-hjb}) and the solution $\phi_0$ of (\ref{averaged-fk}) where,
by the Feynman-Kac formula, both $\phi^\epsilon$ and $\phi_0$ can be expressed in terms of conditional 
expectations like (\ref{expectation-averaged}).\\

\textbf{Idea of the proof. }
Under Assumption~\ref{assumption-1}, it is well known that both
$\phi^\epsilon$ and $\phi_0$ are $C^1$ functions
\cite{da2002second,cerrai2001second,friedman1964partial} and that, 
using the probabilistic representation (\ref{expectation-averaged}), their derivatives have explicit 
expressions in terms of conditional expectations : 
\begin{align}
  \begin{split}
    \partial_{x_i} \phi^\epsilon =&- \bE^{x,y}\Big[
   e^{-\int_t^T h(x_s) ds}\int_t^T \nabla_x
 h(x_s) \cdot x_{s,x_i}  \,ds\Big],\quad 1 \le i \le k \\
   \partial_{y_i} \phi^\epsilon =&- \bE^{x,y}\Big[
   e^{-\int_t^T h(x_s) ds}\int_t^T \nabla_x
 h(x_s) \cdot x_{s,y_i}  \,ds\Big], \quad 1 \le i \le l\\
 \partial_{x_i} \phi_0 =&- \bE^{x}\Big[
 e^{-\int_t^T h(\widetilde{x}_s) ds}\int_t^T \nabla_x
 h(\widetilde{x}_s) \cdot \widetilde{x}_{s,x_i}  \,ds\Big],\quad 1 \le i \le
 k\,. 
  \end{split}
  \label{derivative-formula}
\end{align}
That is, the derivatives can be put inside the expectation, see
Section 1.3 of \cite{cerrai2001second} and Section~2.7-2.8 of \cite{krylow_controlled1980}. Here, we have
used Assumption~\ref{assumption-3} that the running cost $h$ depends only on $x$, and that 
the dynamics $x_s, y_s$ and $\widetilde{x}_s$ satisfy (\ref{averaging-dynamics}) and
(\ref{dynamics-averaged}). Moreover, we have introduced the shorthand $\bE^{x,y}$ to denote the expectation
conditioned on $x_t=x,y_t=y$ and similarly for $\bE^{x}$. \\ 

The processes $x_{s,x_i} \in \mathbb{R}^k, y_{s,x_i} \in \mathbb{R}^l$ in
(\ref{derivative-formula}) describe the partial derivatives of processes $x_s$ and $y_s$
with respect to the initial conditions and satisfy the equations
\begin{align}
  \begin{split}
  &dx_{s,x_i} = (\nabla_{x} f\, x_{s,x_i} + \nabla_{y} f\, y_{s, x_i}) ds + (\nabla_{x} \alpha_1\, x_{s,x_i} + \nabla_{y} \alpha_1\, y_{s,x_i}) dw_s^1 \\
  &dy_{s,x_i} = \frac{1}{\epsilon} (\nabla_{x} g\, x_{s,x_i} + \nabla_{y} g\, y_{s, x_i}) ds +
  \frac{1}{\sqrt{\epsilon}} (\nabla_{x} \alpha_2\, x_{s,x_i} + \nabla_{y}
  \alpha_2\, y_{s,x_i}) dw_s^2\,,
\end{split}
\quad 1 \le i \le k
\label{1st-variation-x}
\end{align}
with $x^{j}_{t,x_i} = \delta_{ij}, 1 \le j \le k$, $y_{t,x_i} = 0 \in
\mathbb{R}^l$. 
Here $\nabla_x \alpha_1 x_{s,x_i}$ denotes the $k \times m_1$ matrix
whose components are 
\begin{align}
  (\nabla_x \alpha_1 x_{s,x_i})_{j_1j_2} = \sum_{r=1}^k\frac{\partial
  (\alpha_1)_{j_1j_2}}{\partial
  x_r} x^{r}_{s,x_i}\,, \qquad 1 \le j_1\le k\,, \quad 1 \le j_2 \le m_1\,.
  \label{notation-explain}
\end{align}
The other terms in (\ref{1st-variation-x}) are analogously defined.
Similarly, the processes $x_{s,y_i} \in \mathbb{R}^k$ and $y_{s,y_i} \in
\mathbb{R}^l$ satisfy
\begin{align}
  \begin{split}
  &dx_{s,y_i} = (\nabla_{x} f\, x_{s,y_i} + \nabla_{y} f\, y_{s, y_i}) ds + (\nabla_{x} \alpha_1\,
    x_{s,y_i} + \nabla_{y} \alpha_1\, y_{s,y_i}) dw_s^1 \\
  &dy_{s,y_i} = \frac{1}{\epsilon} (\nabla_{x} g \,x_{s,y_i} + \nabla_{y} g \,y_{s, y_i}) ds +
    \frac{1}{\sqrt{\epsilon}} (\nabla_{x} \alpha_2\, x_{s,y_i} + \nabla_{y}
    \alpha_2\, y_{s,y_i}) dw_s^2 \,,
\end{split}
\quad 1 \le i \le l
\label{1st-variation-y}
\end{align}
with $x_{t,y_i} = 0 \in \mathbb{R}^k$, $y^j_{t,y_i} = \delta_{ij} \in
\mathbb{R}^l, 1 \le j \le l$
(Notice that the above equations also hold when the coefficient $\alpha_1$ depends on both
$x,y$, so terms involving $\nabla_{y}\alpha_1$ are kept there). 
The above formulas (\ref{derivative-formula})--(\ref{1st-variation-y}) allow us to compare the 
dynamics $x_s, y_s$, $\widetilde{x}_s$, the controlled dynamics and the resulting importance sampling
estimators. For simplicity, we
consider the dynamics on $[0,T]$ that entails similar estimates for the case $s\in [t,T]$. 
We therefore suppose that the initial values of $x_s, \widetilde{x}_s$ are
$x_0 \in \mathbb{R}^k$ and the initial value of $y_s$ is $y_0 \in
\mathbb{R}^l$. The notation $\mathbf{E}$ below will always refer to the expectation
conditioned on these initial values.\\

To prove Theorem \ref{main_thm}, we will adapt some estimates used in
\cite{liu2010}. See also
\cite{da1996ergodicity,cerrai2001second,liu2014_nnsa,Givon_jump_diffusion2007}
for similar techniques. We follow \cite{liu2010} and define a partition of the interval $[0,T]$ by $[0, \Delta]$,
$[\Delta,2 \Delta]$, $\cdots$, $[(M-1)\Delta, M\Delta]$ with $\Delta =T/M$, $M > 0$, and consider the auxiliary process 
\begin{align}
  \begin{split}
  d\hat{x}_s &= f(x_{j\Delta},\hat{y}_s) ds + \alpha_1(x_s) dw_s^1 \\
  d\hat{y}_s &= \frac{1}{\epsilon} g(x_{j\Delta},\hat{y}_s) ds +
  \frac{1}{\sqrt{\epsilon}}\alpha_2(x_{j\Delta},\hat{y}_s) dw_s^2 
\end{split}
\label{auxiliary}
\end{align}
for $s\in[j\Delta, (j+1)\Delta)$, $0\le j \le (M-1)$, with the continuity
  condition 
  \begin{align*}
    \hat{x}_{(j+1) \Delta} = \lim_{s \rightarrow (j+1)\Delta^-}
    \hat{x}_s,\quad 
    \hat{y}_{(j+1) \Delta} = \lim_{s \rightarrow (j+1)\Delta^-}
    \hat{y}_s\,,
  \end{align*}
  and initial conditions $\hat{x}_0 = x_0$, $\hat{y}_0 = y_0$.
Without loss of generality, we can suppose that  $\Delta \le 1$. This auxiliary process
will serve as a bridge between (\ref{averaging-dynamics}) and
(\ref{dynamics-averaged}). In contrast to \cite{liu2010} and due to the fact that we consider controlled dynamics, 
estimates for $4$th-order moments as well as for the processes (\ref{1st-variation-x})
and (\ref{1st-variation-y}) will be needed in order to prove the theorem.

Before entering the details of the various estimates,
we first summarize our main technical results, the proofs of which will be given in the
following subsections. 

For the derivative processes satisfying (\ref{1st-variation-x}) and
(\ref{1st-variation-y}), we have (see Theorem~\ref{y-derivative-bound-2} and
Lemma~\ref{x-derivative-bound} below): 
\begin{theorem}
  Let Assumptions \ref{assumption-1}--\ref{assumption-3} hold.
  Then $\exists C>0$, independent of $\epsilon$, $x_0$ and $y_0$, such that 
  \begin{align*}
    & \max_{0\le s \le T} \bE|x_{s,x_i}|^2 \le C, \qquad  \max_{0\le s \le T}
    \bE|y_{s,x_i}|^2 \le C, \quad \quad 1 \le i \le k. \\
    & \max_{0\le s \le T} \bE|x_{s,y_i}|^2 \le C\epsilon^2, \qquad  \bE|y_{t,y_i}|^2 \le
    e^{-\frac{\lambda t}{\epsilon}} + C\epsilon^2, \quad t \in [0, T]\,, \quad
    1 \le i \le l.
  \end{align*}
  \label{main-result-1}
\end{theorem}

For the approximation results, we have (see Theorem~\ref{thm-average-limit-xs}
and Theorem~\ref{thm-average-limit-xs_dx} below): 
\begin{theorem}
  Let Assumptions \ref{assumption-1}--\ref{assumption-3} hold.
  Then $\exists C>0$, independent of $\epsilon$
  and can be chosen uniformly for $x_0$, $y_0$ which are contained in some
  bounded domain of $\mathbb{R}^k \times \mathbb{R}^l$, such that 
    \begin{align}
      \max_{0 \le s \le T} \bE|x_s - \widetilde{x}_s|^4 \le C \epsilon^{\frac{1}{2}}\,.  \notag
    \end{align}
  \label{main-result-2}
\end{theorem}

\begin{theorem}
  Let Assumptions \ref{assumption-1}--\ref{assumption-3} hold.
  Then $\exists C>0$, independent of $\epsilon$ 
  and can be chosen uniformly for $x_0$, $y_0$ which are contained in some
  bounded domain of $\mathbb{R}^k \times \mathbb{R}^l$ , such that 
    \begin{align}
      \max_{0 \le s \le T} \bE|x_{s,x_i} - \widetilde{x}_{s,x_i}|^2 \le C \epsilon^{\frac{1}{4}}\,. \notag
    \end{align}
  \label{main-result-3}
\end{theorem}

From these results that will be proved in the remainder of this section, we then obtain: 
\begin{theorem}
  Let Assumptions \ref{assumption-1}--\ref{assumption-3} hold.
  Then $\exists C>0$, independent of $\epsilon$ 
  and can be chosen uniformly for $x$, $y$ which are contained in some bounded domain
  of $\mathbb{R}^k \times \mathbb{R}^l$, such that 
  \begin{enumerate}
    \item
    $|\nabla_y \phi^\epsilon| \le C\epsilon, \quad |\nabla_x \phi^\epsilon -
    \nabla_x \phi_0| \le C\epsilon^{\frac{1}{8}}$.
\item For $U^\epsilon = -\ln \phi^\epsilon$, $U_0= - \ln \phi_0$, we have  
  \begin{align}
    |\nabla_y U^\epsilon| \le C\epsilon, \qquad |\nabla_x U^\epsilon -
    \nabla_x U_0| \le C\epsilon^{\frac{1}{8}}\,.
  \end{align}
  \end{enumerate}
  \label{main-result-4}
\end{theorem}
\begin{proof}
  We use the representation formulas (\ref{derivative-formula}).
  For $\nabla_y \phi^\epsilon$, using Assumption~\ref{assumption-1} and
  Theorem~\ref{main-result-1}, we have 
  \begin{align*}
    |\partial_{y_i} \phi^\epsilon| \le & \bE\Big(e^{-\int_t^T h(x_s)
  ds}\int_t^T |\nabla_x h(x_s)| |x_{s,y_i}| ds \Big) \notag \\
  \le & C
  \bE\int_t^T |x_{s,y_i}| ds \le  C 
  \int_t^T (\bE|x_{s,y_i}|^2)^{\frac{1}{2}} ds
\le C\epsilon \notag\,.
  \end{align*}
  To compare $\nabla_x \phi^\epsilon$ with $\nabla_x \phi_0$, we compute that 
  \begin{align*}
    & |\partial_{x_i} \phi^\epsilon - \partial_{x_i}\phi_0| \\
    \le &
     \Big|\bE\Big[
  e^{-\int_t^T h(x_s) ds}\Big(\int_t^T \big(\nabla_x h(x_s) \cdot x_{s,x_i}-
\nabla_x h(\widetilde{x}_s) \cdot \widetilde{x}_{s,x_i}\big)
\,ds\Big)\Big]\Big| \\
&+  \Big|\bE\Big[
\Big(e^{-\int_t^T h(x_s) ds} - e^{-\int_t^T h(\widetilde{x}_s)
ds}\Big)
\Big(\int_t^T
\nabla_x h(\widetilde{x}_s) \cdot \widetilde{x}_{s,x_i}  \,ds\Big) \Big]\Big| \\
=& I_1 + I_2 \,.
\end{align*}
For $I_1$, using Assumption~\ref{assumption-1}, Theorem~\ref{main-result-2}
and Theorem~\ref{main-result-3}, it follows that 
\begin{align*}
  I_1 \le& \Big|\bE
  \Big(\int_t^T \big(\nabla_x h(x_s) \cdot x_{s,x_i}-
\nabla_x h(\widetilde{x}_s) \cdot \widetilde{x}_{s,x_i}\big)
\,ds\Big)\Big| \\
=& \Big|\bE
\Big(\int_t^T \big[\big(\nabla_x h(x_s) - \nabla_x h(\widetilde{x}_s)\big) \cdot x_{s,x_i}+
\nabla_x h(\widetilde{x}_s) \cdot (x_{s,x_i} - \widetilde{x}_{s,x_i}) \big]
\,ds\Big)\Big| \\
  \le & C \bE\Big[
   \int_t^T 
   \Big(|x_s - \widetilde{x}_s||x_{s,x_i}| + |x_{s,x_i}-
   \widetilde{x}_{s,x_i}|\Big)  \,ds\Big] \\
   \le & 
C 
\int_t^T \Big[\big(\bE|x_s - \widetilde{x}_s|^2\big)^{\frac{1}{2}}
  \big(\bE|x_{s,x_i}|^2\big)^{\frac{1}{2}}
+ \big(\bE|x_{s,x_i}- \widetilde{x}_{s,x_i}|^2\big)^{\frac{1}{2}}\Big]  \,ds
\le C\epsilon^{\frac{1}{8}}\,.
  \end{align*}
For $I_2$, we have
\begin{align*}
  I_2 \le&  \Big[\bE
  \Big(e^{-\int_t^T h(x_s) ds} - e^{-\int_t^T h(\widetilde{x}_s)
ds}\Big)^2\Big]^{\frac{1}{2}} \Big[\bE\Big(\int_t^T
\nabla_x h(\widetilde{x}_s) \cdot \widetilde{x}_{s,x_i}  \,ds\Big)^2\Big]^{\frac{1}{2}}\\
\le & C \Big\{\bE \Big[\int_0^1 e^{-\int_t^T (1-r) h(x_s) + rh(\widetilde{x}_s) ds} \Big(\int_t^T
|h(\widetilde{x}_s) - h(x_s)| ds\Big) dr\Big]^2\Big\}^{\frac{1}{2}}\Big(\bE\int_t^T
 |\widetilde{x}_{s,x_i}|^2  \,ds\Big)^{\frac{1}{2}}\\
 \le & C\Big(\bE\int_t^T |\widetilde{x}_s - x_s|^2
 ds\Big)^{\frac{1}{2}} \le C\epsilon^{\frac{1}{8}} \,,
 \end{align*}
which then entails the estimates for the derivatives of $\phi^\epsilon$.
Meanwhile, using a similar argument,  
\begin{align*}
  |\phi^\epsilon - \phi_0| =&
 \Big|\bE\Big(e^{-\int_t^T h(x_s) ds} - e^{-\int_t^T h(\widetilde{x}_s) ds}\Big)\Big| \\
\le & \bE \Big[\int_0^1 e^{-\int_t^T (1-r) h(x_s) + rh(\widetilde{x}_s) ds} \Big(\int_t^T
|h(\widetilde{x}_s) - h(x_s)| ds\Big) dr\Big] \\
\le & C\bE \Big(\int_t^T |h(\widetilde{x}_s) - h(x_s)| ds\Big) \\
\le & C \int_t^T \big(\bE|\widetilde{x}_s - x_s|^4 \big)^\frac{1}{4} ds\le C
\epsilon^{\frac{1}{8}}\,.
\end{align*}

Since $h$ is bounded by Assumption~\ref{assumption-1}, we have that 
$e^{-C(T-t)} \le \phi^\epsilon \le e^{C(T-t)}$ is uniformly bounded (and
bounded away from zero) for all $\epsilon>0$.
The conclusion concerning $|\nabla_y U^\epsilon|$ and $|\nabla_x U^\epsilon -
\nabla_x U_0|$ follows directly from the above estimates.
\qed
\end{proof}

Recall from Section~\ref{sec-setup} and Subsection~\ref{sub-sec-main-1} that
$\hat{u}$ denotes the optimal control as given by (\ref{optimal-u}) and that the
control $\hat{u}^0$ 
defined in (\ref{sub-control-average})  is a candidate for the suboptimal control which is used for estimating (\ref{exp-I}) with nearly optimal variance. Theorem~\ref{main_thm} that is entailed by the above
results expresses this fact, and we restate it for the readers' convenience: 
\begin{theorem}
  Let Assumptions \ref{assumption-1}--\ref{assumption-3}
  hold, and consider the importance sampling method for computing (\ref{exp-I}) under
  the dynamics (\ref{averaging-dynamics}). When the control $\hat{u}^0$ as given
  in (\ref{sub-control-average}) is used 
  to perform the importance sampling, the relative error
  (\ref{relative-error}) of the Monte Carlo
  estimator satisfies
  \begin{align}
    \mbox{RE}_{\hat{u}^0}(I) \le C\epsilon^{\frac{1}{8}} \notag
  \end{align}
  for $\epsilon \ll 1$ where $C > 0$ is a constant independent of $\epsilon$.
  \label{main_thm_repeat}
\end{theorem}
\begin{proof}
  In the following we will regard the optimal control $\hat{u}$ and
  control $\hat{u}^0$ as functions of $t,x$ and $y$. 
Using (\ref{optimal-u}) and (\ref{sub-control-average}), 
we see that Theorem~\ref{main-result-4} implies that $|\hat{u}_s-\hat{u}_s^0| \le C \epsilon^{\frac{1}{8}}$ uniformly on
  $[0,T] \times D$ where $D$ is any bounded domain of $\mathbb{R}^k \times
  \mathbb{R}^l$ and constant $C$ depends on domain $D$. Furthermore, both of them are uniformly bounded on $[0,T] \times \mathbb{R}^k \times
  \mathbb{R}^l$ from the boundedness of $\phi^\epsilon, \alpha_1, \alpha_2$ and formula (\ref{derivative-formula}).
  
  Now call $\tilde{\bar{x}}^u_s,\tilde{\bar{y}}^u_s$ the controlled dynamics 
  of (\ref{averaging-dynamics}) corresponding to the control $\tilde{\bar{u}}_s = 2 \hat{u}_s
  - \hat{u}^0_s$. Specifically, using (\ref{optimal-u}) and
  (\ref{sub-control-average}) again, we have (for $\beta = 1$ and assume
  Assumption~\ref{assumption-3})
  \begin{align}
  \begin{split}
    d\tilde{\bar{x}}^u_s &= f(\tilde{\bar{x}}^u_s,\tilde{\bar{y}}^u_s) ds -
    \alpha_1(\tilde{\bar{x}}^u_s)\alpha_1^T(\tilde{\bar{x}}^u_s)
    \big(2\nabla_x U^\epsilon(\tilde{\bar{x}}^u_s,\tilde{\bar{y}}^u_s) -
    \nabla_x U_0(\tilde{\bar{x}}^u_s)\big)\,ds + \alpha_1(\tilde{\bar{x}}^u_s) dw_s^1\\
    d\tilde{\bar{y}}^u_s &= \frac{1}{\epsilon} g(\tilde{\bar{x}}^u_s,\tilde{\bar{y}}^u_s) ds
    -\frac{2}{\epsilon}
    \alpha_2(\tilde{\bar{x}}^u_s,
    \tilde{\bar{y}}^u_s)\alpha^T_2(\tilde{\bar{x}}^u_s,\tilde{\bar{y}}^u_s)
    \nabla_y U^\epsilon(\tilde{\bar{x}}^u_s,\tilde{\bar{y}}^u_s)\,ds
    +
    \frac{1}{\sqrt{\epsilon}}\alpha_2(\tilde{\bar{x}}^u_s,\tilde{\bar{y}}^u_s)
    dw_s^2 \,,
\end{split}
\label{controlled-averaging-dynamics}
\end{align}
and control $\tilde{\bar{u}}_s$ is bounded on $[0,T] \times
\mathbb{R}^k\times \mathbb{R}^l$ uniformly for $\epsilon$.
This especially implies that Lemma~\ref{lemma-4th-stability} and Lemma~\ref{lemma-occupation-time} in 
Subsection~\ref{sub-sec-stability} also hold for dynamics
$\tilde{\bar{x}}^u_s, \tilde{\bar{y}}^u_s$ (see
Remark~\ref{rmk-occupation-time}).

  Let $R > 0$ and for $y \in \mathbb{R}^l$, we define $\chi_R(y) = 1$, if $|y| \le R$, and $\chi_R(y) = 0$
otherwise. Similarly, for $x \in \mathbb{R}^k, y\in \mathbb{R}^l$, we define
$\chi_R(x,y) = 1$, if both $|x|, |y| \le R$, and 
$\chi_R(x,y) = 0$ otherwise. Then applying the uniform approximation
$|\hat{u}_s-\hat{u}_s^0| \le C_R \epsilon^{\frac{1}{8}}$ on bounded domain
defined by $\chi_R(x,y)$ and using the boundedness of both controls, we can recast (\ref{likelihood-exp}) as 
\begin{align}
  & \mathbf{\tilde{\bar{E}}} \Big[\exp\Big(\int_t^T |\hat{u}_s - \hat{u}^0_s|^2
\chi_R(\tilde{\bar{x}}^u_s, \tilde{\bar{y}}^u_s) ds + \int_t^T |\hat{u}_s
- \hat{u}^0_s|^2
\big(1 - \chi_R(\tilde{\bar{x}}^u_s, \tilde{\bar{y}}^u_s)\big) ds \Big)\Big] \notag \\
\le & e^{C_R (T - t) \epsilon^{\frac{1}{4}}}
\mathbf{\tilde{\bar{E}}} \Big[\exp\Big(\int_t^T |\hat{u}_s
- \hat{u}^0_s|^2
\big(1 - \chi_R(\tilde{\bar{x}}^u_s, \tilde{\bar{y}}^u_s)\big) ds \Big)\Big] \notag \\
\le & e^{C_R (T - t) \epsilon^{\frac{1}{4}}}
\mathbf{\tilde{\bar{E}}} \Big[\exp\Big(C\int_t^T \big(1 -
\chi_R(\tilde{\bar{x}}^u_s, \tilde{\bar{y}}^u_s)\big) ds \Big)\Big] \notag \\
\le &  e^{C_R (T - t) \epsilon^{\frac{1}{4}}}
\Big[e^{C\delta} + e^{C T}
\mathbf{P} \Big(\int_t^T \big(1 - \chi_R(\tilde{\bar{x}}^u_s,
\tilde{\bar{y}}^u_s)\big) ds
\ge \delta\Big)
\Big] \label{bound-tmp}
\end{align}
where $\delta > 0$ and $C_R$ is a constant that depends on $R>0$. 
In the last inequality we have split the expectation according to the event
$\big\{\int_t^T \big(1 - \chi_R(\tilde{\bar{x}}^u_s, \tilde{\bar{y}}^u_s)\big) ds
\ge \delta\big\}$ and its complement.
Therefore, applying the conclusion of Lemma~\ref{lemma-occupation-time} to
processes $\tilde{\bar{x}}^u_s, \tilde{\bar{y}}^u_s$, we
can bound the above quantity (\ref{bound-tmp}) by
\begin{align}
e^{C_R (T - t) \epsilon^{\frac{1}{4}}}
\Big[e^{C\delta} + e^{C T}
  \frac{CT(1 + |x|^4 + |y|^4)}{\delta R^4} \Big] \notag \,.
\end{align}
Now we can first choose a small $\delta$ and then a large $R$ such that  
\begin{align}
\mathbf{\tilde{\bar{E}}} \Big[\exp\Big(\int_t^T |\hat{u}_s -
\hat{u}^0_s|^2 ds\Big)\Big]
\le 2e^{C (T - t) \epsilon^{\frac{1}{4}}} \notag
\end{align}
where the constant $C>0$ is independent of $\epsilon$. Combining this with
(\ref{variance}) and  (\ref{relative-error}), (\ref{zt-ratio-exp}), (\ref{likelihood-exp}), we conclude that 
\begin{align}
  \mbox{RE}_{\hat{u}^0}(I) \le C\epsilon^{\frac{1}{8}} \notag
\end{align}
whenever $\epsilon$ is sufficiently small.
\qed
\end{proof}

\subsection{Estimates for processes $x_{s,y_i}$ and $y_{s,y_i}$}

\noindent
We first consider the processes $x_{s,y_i}$ and $y_{s,y_i}$ in (\ref{1st-variation-y}), since the arguments are simpler and
largely unrelated to the rest of the proof. In the following and throughout
this section, we denote by $C$ a generic constant that is independent of
$\epsilon$ and whose value may change from line to line. 
Also recall H\"older and Young's inequalities : Given two random variables $X, Y$, and $p, q > 0$ with $\frac{1}{p} + \frac{1}{q} = 1$, it holds that 
\begin{align}
\bE |XY| \le \big(\bE|X|^p\big)^{\frac{1}{p}}\big(\bE|Y|^q\big)^{\frac{1}{q}} \le 
\frac{\bE|X|^p}{p} + \frac{\bE|Y|^q}{q}\,. 
\label{hodler-young-ineq}
\end{align}

\begin{lemma}
  Under Assumptions \ref{assumption-1}--\ref{assumption-2}, there exists $C>0$,
  independent of $\epsilon$, $x_0$ and $y_0$, such that 
  \begin{align}
    \max_{0\le s \le T} \bE|x_{s,y_i}|^2 \le C\epsilon, \quad  \bE|y_{t,y_i}|^2 \le
    e^{-\frac{\lambda t}{\epsilon}} + C\epsilon, \quad t \in [0, T], \quad 1 \le i \le l.
  \end{align}
  \label{y-derivative-bound}
\end{lemma}
\begin{proof}
Recall the notation in (\ref{notation-explain}) 
and apply Ito's formula to $|x_{s,y_i}|^2$ and $|y_{s,y_i}|^2$. 
After taking expectations, equation (\ref{1st-variation-y}) yields   
\begin{align}
  \begin{split}
  &d \bE|x_{s,y_i}|^2 = 2\bE\langle \nabla_{x} f\,x_{s,y_i}, x_{s,y_i}\rangle ds +
    2\bE\langle \nabla_{y} f\,y_{s,y_i}, x_{s,y_i}\rangle  ds + \bE\|\nabla_{x}
    \alpha_1\, x_{s,y_i} + \nabla_{y} \alpha_1 \,y_{s,y_i}\|^2
  ds \\
  &d \bE|y_{s,y_i}|^2 = \frac{2}{\epsilon} \bE\langle \nabla_{x} g\,x_{s,y_i},
  y_{s,y_i}\rangle ds + \frac{2}{\epsilon} \bE\langle \nabla_{y} g\,y_{s,y_i},
  y_{s,y_i}\rangle ds + \frac{1}{\epsilon} \bE\|\nabla_{x} \alpha_2\,x_{s,y_i} + \nabla_{y}
  \alpha_2\, y_{s,y_i}\|^2
  ds \,,
\end{split}
\label{square-exp-y}
\end{align}
where $\|\cdot\|$ denotes the Frobenius norm of a matrix.
Then, using the Cauchy-Schwarz inequality, Lipschitz continuity of the
coefficients (Assumption~\ref{assumption-1}) and inequality (\ref{mixing}) in Remark~\ref{rmk-1}, it follows that  
\begin{align}
  \begin{split}
  &\frac{d\bE|x_{s,y_i}|^2}{ds} \le C \big(\bE|x_{s,y_i}|^2 +
    \bE|y_{s,y_i}|^2\big) \\
  &\frac{d \bE|y_{s,y_i}|^2}{ds} \le -\frac{\lambda}{\epsilon}
    \bE|y_{s,y_i}|^2 + \frac{C}{\epsilon} \bE|x_{s,y_i}|^2
\end{split}
\label{y-derivative-inq}
\end{align}
with $\bE|x_{0,y_i}|^2 = 0$, $\bE|y_{0,y_i}|^2 = 1$. The conclusion then follows
from Claim~\ref{claim-2} in Appendix~\ref{app-1}.
\qed
\end{proof}

The above result can be improved if we additionally impose Assumption~\ref{assumption-3} and if we
treat the initial layer near $t=0$ more carefully.
\begin{theorem}
  Let Assumptions \ref{assumption-1}--\ref{assumption-3} hold.
  Then $\exists C>0$, independent of $\epsilon$, $x_0$ and $y_0$, such that 
  \begin{align}
    \max_{0\le s \le T} \bE|x_{s,y_i}|^2 \le C\epsilon^2, \qquad  \bE|y_{t,y_i}|^2 \le
    e^{-\frac{\lambda t}{\epsilon}} + C\epsilon^2, \quad t \in [0, T]\,, \quad
    1 \le i \le l\,. \notag
  \end{align}
  \label{y-derivative-bound-2}
\end{theorem}

\begin{proof}
  Applying Ito's formula in the same way as in Lemma~\ref{y-derivative-bound}
  and noticing that the coefficient $\alpha_1$ is independent of $y$, we can obtain 
\begin{align}
  \begin{split}
  &d \bE|x_{s,y_i}|^2 = 2\bE\langle \nabla_{x} f\,x_{s,y_i}, x_{s,y_i}\rangle ds +
    2\bE\langle \nabla_{y} f\,y_{s,y_i}, x_{s,y_i}\rangle  ds + \bE\|\nabla_{x}
    \alpha_1\, x_{s,y_i}\|^2
  ds \\
  &d \bE|y_{s,y_i}|^2 = \frac{2}{\epsilon} \bE\langle \nabla_{x} g\,x_{s,y_i},
  y_{s,y_i}\rangle ds + \frac{2}{\epsilon} \bE\langle \nabla_{y} g\,y_{s,y_i},
  y_{s,y_i}\rangle ds + \frac{1}{\epsilon} \bE\|\nabla_{x} \alpha_2\,x_{s,y_i} + \nabla_{y}
  \alpha_2\, y_{s,y_i}\|^2
  ds \,.
\end{split}
\label{square-exp-y-assump3}
\end{align}

Now set $t_1 = -\frac{2\epsilon\ln \epsilon}{\lambda}$ 
and introduce the function $\gamma \colon [0,T] \rightarrow [0, 1]$ by
\begin{align}
  \gamma(t) =\left\{
  \begin{array}{cl}
    1- \frac{t}{t_1} &\qquad 0 \le t \le t_1 	\\
    0 &\qquad  t_1 < t \le T
  \end{array}
  \right.
\end{align}
Then using the Cauchy-Schwarz inequality and the Lipschitz condition in
Assumption~\ref{assumption-1}, we have 
\begin{align}
  &  \bE\langle \nabla_{y} f\, y_{s,y_i}, x_{s,y_i}\rangle  \le
  C\Big(\epsilon^{-\gamma(s)}
  \frac{\bE|x_{s,y_i}|^2}{2}+ \epsilon^{\gamma(s)} \frac{\bE|y_{s,y_i}|^2}{2}\Big) \notag \\
  &\bE\langle \nabla_{y} g\, x_{s,y_i}, y_{s,y_i}\rangle  \le \frac{C^2}{\lambda}
  \frac{\bE|x_{s,y_i}|^2}{2}+ \lambda \frac{\bE|y_{s,y_i}|^2}{2}\,.  \notag
\end{align}
Substituting them into (\ref{square-exp-y-assump3}) and applying inequality
(\ref{mixing}) in Remark~\ref{rmk-1}, we find
\begin{align}
  & \frac{d\bE|x_{s,y_i}|^2}{ds} \le C(1+\epsilon^{-\gamma(s)})
  \bE|x_{s,y_i}|^2 + C\epsilon^{\gamma(s)} \bE|y_{s,y_i}|^2 \notag \\
  &\frac{d \bE|y_{s,y_i}|^2}{ds} \le -\frac{\lambda}{\epsilon} \bE|y_{s,y_i}|^2 + \frac{C}{\epsilon}
\bE|x_{s,y_i}|^2\,, \notag
\end{align}
with $\bE|x_{0,y_i}|^2 = 0$, $\bE|y_{0,y_i}|^2 = 1$. The conclusion follows from
Claim~\ref{claim-3} in Appendix~\ref{app-1}. 
\qed
\end{proof}

\subsection{Stability estimates}
\label{sub-sec-stability}

\noindent
We start with some basic facts related to the stability of the dynamics
(\ref{averaging-dynamics}), (\ref{dynamics-averaged}), (\ref{1st-variation-x}) and (\ref{auxiliary}).
Bear in mind that $\beta = 1$ throughout this section. For processes $x_s, y_s$ satisfying (\ref{averaging-dynamics}), we have:

\begin{lemma}
  Under Assumption \ref{assumption-1}, \ref{assumption-2}, there exists $C>0$,
  independent of $\epsilon$, $x_0$ and $y_0$, such that 
  \begin{align}
      \max_{0\le s \le T} \bE|x_s|^4 \le  C \big(|x_0|^4 + |y_0|^4 + 1\big), \qquad 
   \max_{0\le s \le T} \bE|y_s|^4 \le  C\big(|y_0|^4 + |x_0|^4 + 1\big).
  \end{align}
  \label{lemma-4th-stability}
\end{lemma}
\begin{proof}
  Applying Ito's formula to $|x_s|^4$ and taking expectation, we can obtain
  \begin{align}
    \frac{d\bE|x_s|^4}{ds} 
    = & 4\bE\Big(|x_s|^2\langle f(x_s, y_s), x_s\rangle\Big) + 2
    \bE\Big(|x_s|^2\|\alpha_1(x_s, y_s)\|^2\Big) + 4\bE\Big(|\alpha^T_1(x_s,
    y_s) x_s|^2\Big)\notag \\
    \le& 4\bE\Big(|x_s|^2\langle f(x_s, y_s), x_s\rangle\Big) + 6
    \bE\Big(|x_s|^2\|\alpha_1(x_s, y_s)\|^2\Big) \notag \,,
  \end{align}
  and similarly for $|y_s|^4$, 
  \begin{align}
    \frac{d\bE|y_s|^4}{ds} \le& \frac{4}{\epsilon} \bE\Big(|y_s|^2\langle g(x_s, y_s), y_s\rangle\Big) + 
    \frac{6}{\epsilon} \bE\Big(|y_s|^2\|\alpha_2(x_s, y_s)\|^2\Big). \notag 
  \end{align}
  By Assumption \ref{assumption-1}, $f$ is Lipschitz and $\alpha_1$ is bounded. 
  We also know from Remark~\ref{rmk-1} that $|f(x_s, y_s)| \le C(1 + |x_s| +
  |y_s|)$ and inequality (\ref{mixing-1}) holds.
  Together with Young's inequality, we obtain  
  \begin{align}
    \frac{d\bE|x_s|^4}{ds} \le& C \Big(\bE|x_s|^4 + \bE|y_s|^4 + 1\Big) \notag \\
    \frac{d\bE|y_s|^4}{ds} \le& -\frac{\lambda}{\epsilon} \bE|y_s|^4 +
    \frac{C}{\epsilon} \Big(\bE|x_s|^4 + 1\Big) \notag \,.
  \end{align}
  An argument similar to the one in Claim~\ref{claim-2} of Appendix~\ref{app-1} provides us with the desired estimates.
\qed
\end{proof}
\begin{remark}
  Reiterating the above argument, we can prove that the solutions of (\ref{auxiliary}) and (\ref{dynamics-averaged}) satisfy 
  \begin{align}
    \max_{0\le s \le T} \bE|\hat{x}_s|^4 \le  C \big(|x_0|^4 + |y_0|^4 + 1\big), \qquad 
    \max_{0\le s \le T} \bE|\hat{y}_s|^4 \le  C \big(|y_0|^4 + |x_0|^4 +
    1\big)\,,
  \end{align}
  and 
  \begin{align}
    \max_{0\le s \le T} \bE|\widetilde{x}_s|^4 \le  C \big(|x_0|^4 + 1\big)\,,
  \end{align}
  since $\widetilde{f}$ is Lipschitz as well (Remark~\ref{rmk-1}).
  \label{rmk-2}
\end{remark}

The above results entail estimates for the supremum of the solution $x_{s}$ of SDE
(\ref{averaging-dynamics}), as well as for the occupation time of $y_s$ on finite time intervals:  
\begin{lemma}
  Letting Assumptions \ref{assumption-1}--\ref{assumption-2} hold, there exists $C>0$,
  independent of $\epsilon$, $x_0$ and $y_0$, such that
  \begin{align}
\mathbf{E}\big(\sup_{0 \le s \le T} |x_s|^4\big) \le C\big(1 + |x_0|^4 + |y_0|^4\big) \notag\,.
  \end{align}
  Moreover, for all $\delta, R > 0$, it holds  
\begin{align}
  \mathbf{P} \Big(\int_0^T \big(1 - \chi_R(y_s)\big) ds \ge \delta \Big) \le
  \frac{C\big(1 + |x_0|^4 + |y_0|^4\big)}{\delta
  R^4}\,, \notag \\
  \mathbf{P} \Big(\int_0^T \big(1 - \chi_R(x_s, y_s)\big) ds \ge \delta \Big) \le
  \frac{C\big(1 + |x_0|^4 + |y_0|^4\big)}{\delta R^4}\,, \notag
\end{align}
where the characteristic functions are defined in the proof of Theorem~\ref{main_thm_repeat}.
\label{lemma-occupation-time}
\end{lemma}
\begin{proof}
  The proof is standard. Since $f$ is Lipschitz, H\"older's inequality entails
  \begin{align}
    |x_s|^4 \le & C \Big(|x_0|^4 + \Big|\int_0^s f(x_r,y_r) dr\Big|^4 + \Big|\int_0^s
    \alpha_1(x_r,y_r) dw_r^1\Big|^4\Big)\notag \\
    \le &  C \Big(|x_0|^4 + s^3 \int_0^s |f(x_r,y_r)|^4 dr + \Big|\int_0^s
    \alpha_1(x_r,y_r) dw_r^1\Big|^4\Big) \notag\\
    \le &  C \Big(|x_0|^4 + T^3 \int_0^T \big(|x_r|^4 + |y_r|^4 + 1\big) dr + \Big|\int_0^s
    \alpha_1(x_r,y_r) dw_r^1\Big|^4\Big)\,. \notag
  \end{align}
  Taking first the supremum and then the expected value on both sides, we find 
  \begin{align}
    \mathbf{E}\big(\sup_{0 \le s \le T} |x_s|^4\big) \le & C\Big[|x_0|^4 +
    T^3\mathbf{E}\int_0^T
    \big(|x_r|^4 + |y_r|^4 + 1\big) dr + \mathbf{E}\Big(\sup_{0 \le s \le T}
  \Big(\int_0^s \alpha_1(x_r,y_r) dw_r^1\Big)^4\Big)\Big]\,.  \notag
  \end{align}
  The first integral in the last equation can be bounded using Lemma~\ref{lemma-4th-stability}, whereas the second one is bounded by the maximal
  martingale inequality \cite{Karatzas1991}. Hence 
  \begin{align}
\mathbf{E}\big(\sup_{0 \le s \le T} |x_s|^4\big) \le & C\big(|x_0|^4 + |y_0|^4
    + 1\big) + C\Big(
    \mathbf{E}\int_0^T |\alpha_1(x_r,y_r)|^2 dr\Big)^{2} \notag
  \end{align}
   and the boundedness of $\alpha_1$ entails 
  \begin{align}
\mathbf{E}(\sup_{0 \le s \le T} |x_s|^4) \le C\big(1 + |x_0|^4 +
|y_0|^4\big) \notag \,.
  \end{align}
  
  As for the second part of the assertion, notice that for all $\delta>0$ and $R > 0$ it holds:   
  \begin{align*}
R^4 \mathbf{E} \Big[\int_0^T \big(1 - \chi_R(y_s)\big) ds\Big] & \le 
\mathbf{E} \Big[\int_0^T |y_s|^4 \big(1- \chi_R(y_s)\big) ds\Big]\\
&  \le 
  \mathbf{E} \Big(\int_0^T |y_s|^4 ds\Big) \le C\big(1 + |x_0|^4 +
  |y_0|^4\big)\,.
\end{align*}
Thus, by Chebyshev's inequality,  
\begin{align}
  \mathbf{P} \Big(\int_0^T \big(1 - \chi_R(y_s)\big) ds \ge \delta \Big) \le
  \frac{C\big(1 + |x_0|^4 + |y_0|^4\big)}{\delta
  R^4}\,.
  \notag
\end{align}
The second inequality follows in the same fashion. 
\qed
\end{proof}
\begin{remark}
  Based on the result of Theorem~\ref{main-result-4}, we can prove that the
  same conclusions of Lemma~\ref{lemma-4th-stability} and
  Lemma~\ref{lemma-occupation-time} hold for
  processes~(\ref{controlled-averaging-dynamics}) as well. See the discussions in the
  proof of Theorem~\ref{main_thm_repeat}.
  \label{rmk-occupation-time}
\end{remark}

We proceed our analysis by inspecting (\ref{1st-variation-x}) for the
processes $x_{s,x_i}, y_{s, x_i}$, for which we seek the analogue of  
the inequality (\ref{y-derivative-inq}). In this case the initial values satisfy $\bE|x_{0,x_i}|^2 =
1$, $\bE|y_{0,x_i}|^2 = 0$ and by
similar argument as in the proof of Lemma~\ref{y-derivative-bound}, we find:
\begin{lemma}
  Under Assumptions \ref{assumption-1}--\ref{assumption-2}, there exists $C>0$,
  independent of $\epsilon$, $x_0$ and $y_0$, such that 
  \begin{align}
    \max_{0\le s \le T} \bE|x_{s,x_i}|^2 \le C, \quad  \max_{0\le s \le T} \bE|y_{s,x_i}|^2 \le C, \quad \quad 1 \le i \le k.
  \end{align}
  \label{x-derivative-bound}
\end{lemma}

Upper bounds on $4$th moments can be obtained in the same manner: 
\begin{lemma}
  Under Assumptions \ref{assumption-1}--\ref{assumption-2}, there exists $C>0$,
  independent of $\epsilon$, $x_0$ and $y_0$, such that 
  \begin{align}
    \max_{0\le s \le T} \bE|x_{s,x_i}|^4 \le C, \quad  \max_{0\le s \le T} \bE|y_{s,x_i}|^4 \le C, \quad 1 \le i \le k.
  \end{align}
  \label{d-xi-4}
\end{lemma}
\begin{proof}
  The proof is similar to Lemma~\ref{lemma-4th-stability}. Using Ito's formula, we obtain
\begin{align}
  \begin{split}
    d \bE|x_{s,x_i}|^4 =& 4\bE\Big(|x_{s,x_i}|^2\langle \nabla_{x} f\,x_{s,x_i} + \nabla_{y} f\,y_{s,x_i},
    x_{s,x_i}\rangle\Big) ds + 2 \bE\Big(|x_{s,x_i}|^2\|\nabla_{x} \alpha_1\,
    x_{s,x_i} + \nabla_{y} \alpha_1 \,y_{s,x_i}\|^2\Big)
  ds \\
  &+ 4 \bE\Big(|(\nabla_{x} \alpha_1\, x_{s,x_i} + \nabla_{y} \alpha_1
  \,y_{s,x_i})^Tx_{s,x_i}|^2\Big) ds \\
  \le & 4\bE\Big(|x_{s,x_i}|^2\langle \nabla_{x} f\,x_{s,x_i} + \nabla_{y} f\,y_{s,x_i},
    x_{s,x_i}\rangle\Big) ds + 6 \bE\Big(|x_{s,x_i}|^2\|\nabla_{x} \alpha_1\,
    x_{s,x_i} + \nabla_{y} \alpha_1 \,y_{s,x_i}\|^2\Big)
  ds \\
  d \bE|y_{s,x_i}|^4 =& \frac{4}{\epsilon} \bE\Big(|y_{s,x_i}|^2\langle \nabla_{x} g\,x_{s,x_i} + \nabla_{y} g\,y_{s,x_i},
  y_{s,x_i}\rangle\Big) ds + \frac{2}{\epsilon} \bE\Big(|y_{s,x_i}|^2\|\nabla_{x}
  \alpha_2\, x_{s,x_i} + \nabla_{y} \alpha_2 \,y_{s,x_i}\|^2\Big)
  ds \\
  &+ \frac{4}{\epsilon} \bE\Big(|(\nabla_{x} \alpha_2\, x_{s,x_i} + \nabla_{y} \alpha_2
  \,y_{s,x_i})^Ty_{s,x_i}|^2\Big) ds\, \\
  \le & \frac{4}{\epsilon} \bE\Big(|y_{s,x_i}|^2\langle \nabla_{x} g\,x_{s,x_i} + \nabla_{y} g\,y_{s,x_i},
  y_{s,x_i}\rangle\Big) ds + \frac{6}{\epsilon} \bE\Big(|y_{s,x_i}|^2\|\nabla_{x}
  \alpha_2\, x_{s,x_i} + \nabla_{y} \alpha_2 \,y_{s,x_i}\|^2\Big) ds \,.
\end{split}
\end{align}
Lipschitz conditions on the coefficients in Assumption \ref{assumption-1},
Assumption~\ref{assumption-2}, especially inequality (\ref{mixing}) in
Remark~\ref{rmk-1} as well as Young's inequality now readily imply that 
\begin{align*}
  \frac{d\bE|x_{s,x_i}|^4}{ds} &\le C\big(\bE|x_{s,x_i}|^4 +
  \bE|y_{s,x_i}|^4\big) \\
	 \frac{d\bE|y_{s,x_i}|^4}{ds} &\le -\frac{2\lambda}{\epsilon} \bE|y_{s,x_i}|^4 +
  \frac{C}{\epsilon} \bE|x_{s,x_i}|^4\,,
\end{align*}
with $\bE|y_{0,x_i}|^4 = 0$, $\bE|x_{0,x_i}|^4 = 1$. The assertion then follows by the same argument as in the proof of Claim~\ref{claim-2} in Appendix~\ref{app-1}.
\qed
\end{proof}

We also have the following simple bounds for processes $x_s$ and $x_{s,x_i}$. 
\begin{lemma}
  Let $\Delta \le 1$, $s \in [j\Delta, (j + 1) \Delta), \; 0 \le j \le M - 1$.
 Further let Assumptions~\ref{assumption-1}--\ref{assumption-2} hold. 
 \begin{enumerate} 
   \item
For the process $x_s$ satisfying (\ref{averaging-dynamics}), it holds 
\begin{align}
\bE|x_s - x_{j\Delta}|^4 \le C(s - j\Delta)^2, 
\end{align}
where the constant $C>0$ is independent of $\epsilon, \Delta$ and can be chosen
uniformly for $x_0$ and $y_0$ which are contained in some bounded domain
of $\mathbb{R}^k \times \mathbb{R}^l$.
The same bound is satisfied by processes $\widetilde{x}_s, \hat{x}_s$.
   \item
     For process $x_{s, x_i}$ in (\ref{1st-variation-x}),  we have
\begin{align}
  \bE|x_{s,x_i}-x_{j\Delta, x_i}|^{4} \le C (s - j\Delta)^2 \le C\Delta^2,
\end{align}
with a constant $C>0$ that is independent of $\epsilon, x_0, y_0$. The same
inequality holds if $x_{s,x_{i}}$ is replaced by the processes $\hat{x}_{s,
x_i}$ and $\widetilde{x}_{s, x_i}$.  
\end{enumerate}
\label{x-diff-s-k}
\end{lemma}
\begin{proof}
   For the first part of the conclusion, using that $f$ is
   Lipschitz and therefore $|f(x_r, y_r)|\le C(1 + |x_r| + |y_r|)$ (Remark~\ref{rmk-1}),
   $\alpha_1$ is bounded (Assumption~\ref{assumption-1}), as well as Lemma~\ref{lemma-4th-stability}, we can conclude that 
  \begin{align}
    \bE|x_s - x_{j\Delta}|^4 =& \bE\Big[\int_{j\Delta}^s f(x_r,y_r) dr + \int_{j\Delta}^s
    \alpha_1(x_r,y_r) dw^1_r\Big]^4 \notag \\
    \le & C \bE\Big[\int_{j\Delta}^s \big(1 + |x_r| + |y_r|\big) dr\Big]^4 +
    C\bE\Big[\int_{j\Delta}^s \alpha_1(x_r,y_r) dw^1_r\Big]^4 \notag \\
    \le & C\big(|x_0|^4 + |y_0|^4 + 1\big) (s - j\Delta)^4 + C(s - j\Delta)^2 \notag \\
    \le & C(s - j\Delta)^2\,, \notag 
  \end{align}
  where, in the last inequality, we have used the fact that $\Delta \le 1$. It
  is clear that a common constant $C$ can be chosen for $x_0$, $y_0$ which are
  contained in some bounded domain.

  The second part of the conclusion can be obtained in a similar way by using
  the Lipschitz continuity of the coefficients together with Lemma~\ref{d-xi-4}.
\qed
\end{proof}

\subsection{Approximation by the auxiliary process}

\noindent
In this subsection, we study the approximations of   
the original dynamics (\ref{averaging-dynamics}) by the auxiliary discrete process
(\ref{auxiliary}) and the averaged dynamics (\ref{dynamics-averaged}).
First of all, we have 
\begin{lemma}
  Suppose that Assumptions \ref{assumption-1}--\ref{assumption-3} are met. 
For processes $x_s, y_s$ satisfying (\ref{averaging-dynamics}) and the auxiliary processes $\hat{x}_s$, $\hat{y}_s$ defined in (\ref{auxiliary}),
  we have 
\begin{align}
  \max_{0 \le s \le T} \bE|y_s - \hat{y}_s|^4 \le C \Delta^2\,, \qquad \max_{0
  \le s \le T} \bE|x_s - \hat{x}_s|^4 \le C
  \Delta^2\,,
\end{align}
  where the constant $C>0$ is independent of $\epsilon, \Delta$ and can be chosen
  uniformly for $x_0$, $y_0$ which are contained on some bounded domain
of $\mathbb{R}^k \times \mathbb{R}^l$.
\label{auxiliary-error}
\end{lemma}
\begin{proof}
  Let $j = \floor[\big]{\frac{s}{\Delta}}$, which is the largest integer smaller or
  equal to $\frac{s}{\Delta}$. Applying Ito's formula and using the Lipschitz condition
  for the coefficients $g$, $\alpha_2$ in Assumptions~\ref{assumption-1},
  the inequality in Assumption~\ref{assumption-2}, the conclusion of
  Lemma~\ref{x-diff-s-k}, as well as inequality (\ref{hodler-young-ineq}), we
  can estimate 
\begin{align}
  &\frac{d\bE|y_s - \hat{y}_s|^4}{ds} \notag \\
  =& \frac{4}{\epsilon} \bE\Big(|y_s - \hat{y}_s|^2\langle y_s -
\hat{y}_s, g(x_s, y_s) - g(x_{j\Delta}, \hat{y}_s)\rangle\Big) +
  \frac{2}{\epsilon}\bE\Big(|y_s - \hat{y}_s|^2 \|\alpha_2(x_s, y_s) -
  \alpha_2(x_{j\Delta}, \hat{y}_s)\|^2 \Big) \notag \\
  & + \frac{4}{\epsilon} \bE\Big(\Big|\big(\alpha_2(x_s, y_s) - \alpha_2(x_{j\Delta},
  \hat{y}_s)\big)^T(y_s - \hat{y}_s)\Big|^2\Big) \notag \\
  \le &\frac{4}{\epsilon} \bE\Big(|y_s - \hat{y}_s|^2\langle y_s -
\hat{y}_s, g(x_s, y_s) - g(x_{j\Delta}, \hat{y}_s)\rangle\Big) +
  \frac{6}{\epsilon}\bE\Big(|y_s - \hat{y}_s|^2 \|\alpha_2(x_s, y_s) -
  \alpha_2(x_{j\Delta}, \hat{y}_s)\|^2 \Big) \notag \\
  \le & \frac{4}{\epsilon}\bE\Big[|y_s - \hat{y}_s|^2 \Big(\langle y_s -
  \hat{y}_s, g(x_s, y_s) - g(x_s, \hat{y}_s)\rangle + 3 \|\alpha_2(x_s, y_s) -
  \alpha_2(x_s, \hat{y}_s)\|^2\Big)\Big] \notag \\
   &+ \frac{4}{\epsilon}\bE\Big[|y_s - \hat{y}_s|^2 \Big(\langle y_s -
  \hat{y}_s, g(x_s, \hat{y}_s) - g(x_{j\Delta}, \hat{y}_s)\rangle + 3
  \|\alpha_2(x_s, \hat{y}_s) -
\alpha_2(x_{j\Delta}, \hat{y}_s)\|^2\Big)\Big] \notag \\
  \le & 
-\frac{4\lambda}{\epsilon} \bE|y_s - \hat{y}_s|^4
+ \frac{C}{\epsilon} \bE\Big(|y_s - \hat{y}_s|^3|x_s - x_{j\Delta}|\Big)
  + \frac{C}{\epsilon} \bE\Big(|y_s - \hat{y}_s|^2|x_s - x_{j\Delta}|^2\Big) \notag \\
  \le & -\frac{2\lambda}{\epsilon} \bE|y_s - \hat{y}_s|^4 + \frac{C}{\epsilon}
  \bE|x_s - x_{j\Delta}|^4 \notag \\
  \le & -\frac{2\lambda}{\epsilon} \bE|y_s - \hat{y}_s|^4 + \frac{C}{\epsilon} \Delta^2 \notag
\end{align}
which, by Gronwall's inequality, yields the first inequality. For the second
inequality, applying Ito's formula, taking Assumption~\ref{assumption-1}, Lemma~\ref{x-diff-s-k} and the
above estimate into account, we obtain
\begin{align}
  \frac{d\bE|\hat{x}_s - x_s|^4}{ds} 
   =& 4 \bE\Big(|\hat{x}_s - x_s|^2 \langle f(x_{j\Delta}, \hat{y}_s) -
  f(x_s, y_s), \hat{x}_s - x_s\rangle\Big) \notag \\
  \le & C \bE\Big[|\hat{x}_s - x_s|^3 \Big(|x_{j\Delta} - x_s| + |\hat{y}_s -
y_s|\Big)\Big] \notag \\
  \le & C \Big(\bE|\hat{x}_s - x_s|^4 + \bE|x_{j\Delta} - x_s|^4 +
  \bE|\hat{y}_s - y_s|^4\Big) \notag  \\
  \le  & C \bE|\hat{x}_s - x_s|^4 + C \Delta^2\,, \notag  
\end{align}
and the conclusion follows again by applying Gronwall's inequality.  \qed
\end{proof}

The following elementary estimate will be useful.
\begin{clm}
  Define $F(x) = |x|^2x$, $x \in \mathbb{R}^k$. We have $|F(x) -
  F(x')| \le \frac{3}{2}\big(|x|^2+|x'|^2\big)|x-x'|$, $\forall x, x' \in
  \mathbb{R}^k$.
  \label{claim-4}
\end{clm}
\begin{proof}
  We have 
  \begin{align}
    & |F(x) - F(x')| \notag \\
    =& \Big|\int_0^1 \frac{d}{dt} F\big((1-t)x + tx'\big) dt\Big| \notag \\
      =& \Big|\int_0^1 \Big[2\langle (1-t)x + tx',
  x'-x\rangle\big((1-t)x + tx'\big) + |(1-t)x + tx'|^2(x'-x)\Big] dt\Big| \notag \\
  \le& 3\int_0^1 |(1-t)x + tx'|^2 |x'-x| dt \le
  \frac{3}{2}\big(|x|^2+|x'|^2\big)|x-x'|\,.
  \notag 
  \end{align}
\qed
\end{proof}

As the next step, we show that the averaged process $\widetilde{x}_s$ in
(\ref{dynamics-averaged}) can be approximated by the time-discrete process
(\ref{auxiliary}) as well.
\begin{lemma}
  Under Assumptions \ref{assumption-1}--\ref{assumption-3}, we have
\begin{align}
\max_{0 \le s \le T} \bE|\hat{x}_s - \widetilde{x}_s|^4 \le C\left(\frac{\epsilon}{\lambda\Delta} +
  \Delta\right) e^{C\left(1+\frac{\epsilon}{\lambda\Delta}\right)T}\,.
\end{align}
where the constant $C>0$ is independent of $\epsilon, \Delta$ and can be chosen
uniformly for $x_0$, $y_0$ which are contained in some bounded domain
of $\mathbb{R}^k \times \mathbb{R}^l$.
Especially, for $\Delta=\epsilon^{\frac{1}{2}}$, we have $\max\limits_{0 \le s \le T} \bE|\hat{x}_s -
\widetilde{x}_s|^4 \le C\epsilon^{\frac{1}{2}}$.
\label{lemma-hatx-tildex}
\end{lemma}
\begin{proof}
  We apply Ito's formula to $|\hat{x}_s - \widetilde{x}_s|^4$ and take
  expectations similarly as before. Using the function $F$ defined in
  Claim~\ref{claim-4}, we can estimate 
\begin{align}
   & \bE|\hat{x}_s - \widetilde{x}_s|^4 \notag \\
\le& 4 \int_0^s \bE \Big(|\hat{x}_r - \widetilde{x}_r|^2\big\langle\hat{x}_r - \widetilde{x}_r,
  f(x_{\floor{\frac{r}{\Delta}}\Delta}, \hat{y}_r) - \widetilde{f}(\widetilde{x}_r)
 \rangle\Big) dr 
   + 6 \int_0^s \bE\Big(|\hat{x}_r - \widetilde{x}_r|^2|\alpha_1(x_r) - \alpha_1(\widetilde{x}_r)|^2\Big) dr \notag\\
 = & 
   4 \int_0^s \bE \Big(\big\langle F(\hat{x}_{\floor{\frac{r}{\Delta}}\Delta} -
   \widetilde{x}_{\floor{\frac{r}{\Delta}}\Delta}),
   f(x_{\floor{\frac{r}{\Delta}}\Delta}, \hat{y}_r) -
   \widetilde{f}(x_{\floor{\frac{r}{\Delta}}\Delta}) \big\rangle\Big) dr \notag\\
 & +4 \int_0^s \bE\Big(\big\langle F(\hat{x}_r - \widetilde{x}_r) -
   F(\hat{x}_{\floor{\frac{r}{\Delta}}\Delta} - \widetilde{x}_{\floor{\frac{r}{\Delta}}\Delta}),
   f(x_{\floor{\frac{r}{\Delta}}\Delta}, \hat{y}_r) -
   \widetilde{f}(x_{\floor{\frac{r}{\Delta}}\Delta}) \big\rangle
   \Big) dr \notag\\
 & + 4 \int_0^s \bE\Big(\big\langle F(\hat{x}_r - \widetilde{x}_r),
   \widetilde{f}(x_{\floor{\frac{r}{\Delta}}\Delta}) -
   \widetilde{f}(\widetilde{x}_r)\big\rangle \Big) dr \notag\\
  & + 6 \int_0^s \bE\Big(|\hat{x}_r - \widetilde{x}_r|^2|\alpha_1(x_r) -
 \alpha_1(\widetilde{x}_r)|^2\Big) dr \notag\\
 = & I_1 + I_2 + I_3 + I_4\,.\notag
\end{align}
We estimate the above four terms in the sum separately.
For $I_1$,  
we have 
\begin{align}
  |I_1| \le & 4\sum_{j=0}^{\floor{s/\Delta}} \int_{j\Delta}^{[(j+1)\Delta]\wedge s} \bE
  \Big(|\hat{x}_{j\Delta} -
  \widetilde{x}_{j\Delta}|^3 |\bE_{j\Delta}f(x_{j\Delta},
  \hat{y}_r) - \widetilde{f}(x_{j\Delta})|\Big) dr \notag \\
  \le & C\sum_{j=0}^{\floor{s/\Delta}} \int_{j\Delta}^{[(j+1)\Delta]\wedge s} \bE
  \Big(|\hat{x}_{j\Delta} -
  \widetilde{x}_{j\Delta}|^3 (|x_{j\Delta}| + |\hat{y}_{j\Delta}|
  + 1)\Big) e^{-\frac{\lambda(r - j\Delta)}{\epsilon}} dr \notag \\
  \le & \frac{\epsilon C}{\lambda} \bE \Big[\Big( \sum_{j=0}^{\floor{s/\Delta}} 
  |\hat{x}_{j\Delta} - \widetilde{x}_{j\Delta}|^4\Big)^\frac{3}{4}
  \Big(\sum_{j=0}^{\floor{s/\Delta}}  \big(|x_{j\Delta}| + |\hat{y}_{j\Delta}|
+ 1\big)^4\Big)^{\frac{1}{4}}\Big] \notag\\
\le & \frac{\epsilon C}{\lambda}  \Big(\bE\sum_{j=0}^{\floor{s/\Delta}} 
  |\hat{x}_{j\Delta} - \widetilde{x}_{j\Delta}|^4\Big)^\frac{3}{4}
  \Big(\bE\sum_{j=0}^{\floor{s/\Delta}}  (|x_{j\Delta}| + |\hat{y}_{j\Delta}|
+ 1)^4\Big)^{\frac{1}{4}} \notag\\
\le & \frac{\epsilon C}{\lambda\Delta}  \Big(\bE\sum_{j=0}^{\floor{s/\Delta}} 
  |\hat{x}_{j\Delta} - \widetilde{x}_{j\Delta}|^4 \Delta + 
  \bE\sum_{j=0}^{\floor{s/\Delta}}  (|x_{j\Delta}| + |\hat{y}_{j\Delta}|
+ 1)^4 \Delta \Big) \notag\\
\le & \frac{\epsilon C}{\lambda\Delta} \bE\int_0^s |\hat{x}_r - \widetilde{x}_r|^4 dr
+ \frac{\epsilon C}{\lambda\Delta} \bE\int_0^s
\Big||\hat{x}_{\floor{\frac{r}{\Delta}}\Delta} - \widetilde{x}_{\floor{\frac{r}{\Delta}}\Delta}|^4 - |\hat{x}_r - \widetilde{x}_r|^4\Big| dr \notag\\
& + \frac{\epsilon C}{\lambda\Delta}  \bE\sum_{j=0}^{\floor{s/\Delta}}  \Big(|x_{j\Delta}| + |\hat{y}_{j\Delta}|
+ 1\Big)^4 \Delta\,.  \notag
\end{align}
In the first inequality above, $\bE_{j\Delta}$ denotes the expectation
conditioned on $\hat{y}_{s}$ at time $s=j\Delta$. We have used
Lemma~\ref{exp-converge} in Appendix~\ref{app-2} to derive the second
inequality. H\"older inequality and Young's inequality (\ref{hodler-young-ineq}) were also used.
Therefore, by Lemma~\ref{lemma-4th-stability} and Remark~\ref{rmk-2}, the last
inequality implies 
\begin{align}
  |I_1| \le 
  \frac{\epsilon C}{\lambda\Delta} \bE\int_0^s |\hat{x}_r - \widetilde{x}_r|^4 dr
+  \frac{Cs\epsilon }{\lambda\Delta} .
\notag
\end{align}
For $I_2$, since functions $f, \widetilde{f}$ are Lipschitz, we have 
\begin{align*}
  |f(x_{\floor{\frac{r}{\Delta}}\Delta}, \hat{y}_r)| \le& \,C\big(1 +
  |x_{\floor{\frac{r}{\Delta}}\Delta}| + |\hat{y}_r|\big)\,,  \\
|\widetilde{f}(x_{\floor{\frac{r}{\Delta}}\Delta})| \le& \,C\big(1 +
  |x_{\floor{\frac{r}{\Delta}}\Delta}|\big)\,.
\end{align*}
Then using Claim~\ref{claim-4}, Lemma~\ref{lemma-4th-stability} and
Lemma~\ref{x-diff-s-k}, as well as H\"older and Young's inequalities
(\ref{hodler-young-ineq}),
we can estimate 
\begin{align*}
  |I_2| \le & C\bE \int_0^s \Big(|\hat{x}_r - \widetilde{x}_r|^2 +
  |\hat{x}_{\floor{\frac{r}{\Delta}}\Delta} -
  \widetilde{x}_{\floor{\frac{r}{\Delta}}\Delta}|^2\Big)\\
  & \times \Big|(\hat{x}_r -
  \hat{x}_{\floor{\frac{r}{\Delta}}\Delta}) - (\widetilde{x}_r -
  \widetilde{x}_{\floor{\frac{r}{\Delta}}\Delta})\Big| \Big(1 +
  |x_{\floor{\frac{r}{\Delta}}\Delta}| + |\hat{y}_r|\Big) dr \\
  \le & C\bE\int_0^s \Big(|\hat{x}_r - \widetilde{x}_r|^2 + |(\hat{x}_r -
  \hat{x}_{\floor{\frac{r}{\Delta}}\Delta})  - (\widetilde{x}_r -
  \widetilde{x}_{\floor{\frac{r}{\Delta}}\Delta})|^2\Big)\\
  & \times \Big|(\hat{x}_r -  \hat{x}_{\floor{\frac{r}{\Delta}}\Delta})  - (\widetilde{x}_r -
  \widetilde{x}_{\floor{\frac{r}{\Delta}}\Delta})\Big|
   \Big(1 + |x_{\floor{\frac{r}{\Delta}}\Delta}| + |\hat{y}_r|\Big) dr  \\
  \le & C \bE\int_0^s |\hat{x}_r - \widetilde{x}_r|^4 dr  + C
  \bE\int_0^s\Big|(\hat{x}_r - \hat{x}_{\floor{\frac{r}{\Delta}}\Delta})  - (\widetilde{x}_r -
  \widetilde{x}_{\floor{\frac{r}{\Delta}}\Delta})\Big|^3 \Big(1 +
  |x_{\floor{\frac{r}{\Delta}}\Delta}| + |\hat{y}_r|\Big) dr  \\
  & + C \bE\int_0^s\Big|(\hat{x}_r - \hat{x}_{\floor{\frac{r}{\Delta}}\Delta})  - (\widetilde{x}_r -
  \widetilde{x}_{\floor{\frac{r}{\Delta}}\Delta})\Big|^2 \Big(1 +
  |x_{\floor{\frac{r}{\Delta}}\Delta}| + |\hat{y}_r|\Big)^2 dr  \\
  \le & C \bE\int_0^s |\hat{x}_r - \widetilde{x}_r|^4 dr \\
  &  + C
  \int_0^s\Big[\bE\big|(\hat{x}_r - \hat{x}_{\floor{\frac{r}{\Delta}}\Delta})  - (\widetilde{x}_r -
  \widetilde{x}_{\floor{\frac{r}{\Delta}}\Delta})\big|^4\Big]^{\frac{3}{4}}
  \Big[\bE\big(1 + |x_{\floor{\frac{r}{\Delta}}\Delta}| + |\hat{y}_r|\big)^4\Big]^{\frac{1}{4}} dr  \\
  & + C \int_0^s\Big[\bE\big|(\hat{x}_r - \hat{x}_{\floor{\frac{r}{\Delta}}\Delta})  - (\widetilde{x}_r -
  \widetilde{x}_{\floor{\frac{r}{\Delta}}\Delta})\big|^4\Big]^{\frac{1}{2}}
  \Big[\bE\big(1 + |x_{\floor{\frac{r}{\Delta}}\Delta}| + |\hat{y}_r|\big)^4\Big]^{\frac{1}{2}} dr  \\
  \le & C \bE\int_0^s |\hat{x}_r - \widetilde{x}_r|^4 dr  + C s(\Delta +
  \Delta^{\frac{3}{2}}) \,.
\end{align*}
For $I_3$, since function $\widetilde{f}$ is Lipschitz, we have 
\begin{align}
  |I_3|\le & C\bE\int_0^s |\hat{x}_r -
  \widetilde{x}_r|^3|x_{\floor{\frac{r}{\Delta}}\Delta} - \widetilde{x}_r| dr \notag \\
  = & C\bE\int_0^s |\hat{x}_r -
  \widetilde{x}_r|^3 \Big|(x_{\floor{\frac{r}{\Delta}}\Delta} - x_r) + (x_r -
  \hat{x}_r) + (\hat{x}_r - \widetilde{x}_r)\Big| dr \notag \\
  \le & C\bE\int_0^s |\hat{x}_r -
  \widetilde{x}_r|^4 dr + C\bE\int_0^s |\hat{x}_r -
  \widetilde{x}_r|^3|x_{\floor{\frac{r}{\Delta}}\Delta} - x_r| dr + C\bE\int_0^s |\hat{x}_r 
  - \widetilde{x}_r|^3|x_r - \hat{x}_r| dr \notag\\
  \le & C\bE\int_0^s |\hat{x}_r -
  \widetilde{x}_r|^4 dr + C\bE\int_0^s |x_{\floor{\frac{r}{\Delta}}\Delta} -
  x_r|^4 dr + C\bE\int_0^s |x_r - \hat{x}_r|^4 dr \notag\\
  \le&  C\bE\int_0^s |\hat{x}_r - \widetilde{x}_r|^4 dr + Cs \Delta^2\,,\notag
\end{align}
where Lemma~\ref{x-diff-s-k}, Lemma~\ref{auxiliary-error} and Young's
inequality have been used. 

Finally, using that coefficient $\alpha_1$ is Lipschitz and Lemma~\ref{auxiliary-error}, we obtain the following bound for $I_{4}$: 
\begin{align}
|I_4| \le & C\bE\int_0^s |\hat{x}_r -
  \widetilde{x}_r|^2|x_r - \widetilde{x}_r|^2 dr \notag\\
  =  & C\bE\int_0^s |\hat{x}_r - \widetilde{x}_r|^2|(x_r - \hat{x}_r) +
  (\hat{x}_r - \widetilde{x}_r)|^2 dr \notag\\
  \le  & C\bE\int_0^s |\hat{x}_r - \widetilde{x}_r|^4 dr + C\bE\int_0^s |\hat{x}_r -
  \widetilde{x}_r|^2|x_r - \hat{x}_r|^2 dr  \notag\\
  \le  & C\bE\int_0^s |\hat{x}_r - \widetilde{x}_r|^4 dr + C\bE\int_0^s |x_r - \hat{x}_r|^4 dr  \notag\\
  \le &  C\bE\int_0^s |\hat{x}_r - \widetilde{x}_r|^4 dr + Cs\Delta^2\,. \notag 
\end{align}
Combining the above estimates, we obtain the bound (assuming $\Delta \le 1$)
\begin{align}
  \bE|\hat{x}_s - \widetilde{x}_s|^4 \le C\left(1 + \frac{\epsilon}{\lambda\Delta}\right)
  \bE\int_0^s |\hat{x}_r - \widetilde{x}_r|^4 dr + Cs
\left(\frac{\epsilon}{\lambda\Delta} + \Delta\right)\,,
\end{align}
and Gronwall's inequality yields the assertion 
\begin{align}
  \bE|\hat{x}_s - \widetilde{x}_s|^4 \le C\left(\frac{\epsilon}{\lambda\Delta} +
  \Delta\right)
  e^{C\left(1+\frac{\epsilon}{\lambda\Delta}\right)s}\,.
\end{align}
\qed
\end{proof}

 Summarizing Lemma~\ref{auxiliary-error} and Lemma~\ref{lemma-hatx-tildex}, we
 have proved the following estimate for the 4th moments of processes $x_s$
 and $\widetilde{x}_s$ (see \cite{liu2010} for stronger result about the $2$nd moments):

\begin{theorem}
  Suppose that Assumption \ref{assumption-1}--\ref{assumption-3} hold.
  Then there exists $C>0$, independent of $\epsilon$ and can be chosen
  uniformly for $x_0$, $y_0$ which are contained in some bounded domain
of $\mathbb{R}^k \times \mathbb{R}^l$, such that 
    \begin{align}
      \max_{0 \le s \le T} \bE|x_s - \widetilde{x}_s|^4 \le C \epsilon^{\frac{1}{2}} \,. \notag
    \end{align}
  \label{thm-average-limit-xs}
\end{theorem}

As the next step, we consider derivatives of the auxiliary processes (\ref{auxiliary})
\begin{align}
  \begin{split}
    &d\hat{x}_{s,x_i} = \big(\nabla_{x} f\, x_{j\Delta,x_i} + \nabla_{y} f\,
    \hat{y}_{s, x_i}\big) ds + \big(\nabla_{x} \alpha_1\, x_{s,x_i}\big) dw_s^1 \\
    &d\hat{y}_{s,x_i} = \frac{1}{\epsilon} \big(\nabla_{x} g\, x_{j\Delta,x_i}
    + \nabla_{y} g\, \hat{y}_{s, x_i}\big) ds + \frac{1}{\sqrt{\epsilon}}
    \big(\nabla_{x} \alpha_2\, x_{j\Delta,x_i} + \nabla_{y} \alpha_2\,
    \hat{y}_{s,x_i}\big) dw_s^2 \,,
\end{split}
\quad 1 \le i \le k
\label{auxiliary-derivative}
\end{align}
where $j = \floor{\frac{s}{\Delta}}$ and we have assumed that Assumption~\ref{assumption-3} holds.
The following lemma shows that (\ref{auxiliary-derivative}) is an approximation of 
(\ref{1st-variation-x}).
\begin{lemma}
  Under Assumptions \ref{assumption-1}--\ref{assumption-3}, there exists
  $C>0$, independent of $\epsilon, \Delta$ 
and can be chosen uniformly for $x_0$, $y_0$ which are contained in some bounded domain of $\mathbb{R}^k \times \mathbb{R}^l$, such that 
\begin{align}
  \bE|y_{s, x_i} - \hat{y}_{s, x_i} |^2 \le C \Delta\,, \qquad 
  \bE|x_{s, x_i} - \hat{x}_{s, x_i} |^2 \le C \Delta\,.
\end{align}
\label{average-auxiliary-xy}
\end{lemma}
\begin{proof}
  Let $j = \floor{\frac{s}{\Delta}}$. Applying Ito's formula to $|y_{s,x_i} -
  \hat{y}_{s,x_i}|^2$ and taking expectation, we obtain
\begin{align}
  &\frac{d\bE|y_{s,x_i} - \hat{y}_{s,x_i}|^2}{ds} \notag \\
  =& \frac{2}{\epsilon} \bE\big\langle
  \nabla_{x} g(x_s, y_s) x_{s,x_i} - \nabla_{x} g(x_{j\Delta}, \hat{y}_s)
  x_{j\Delta,x_i}, y_{s,x_i} - \hat{y}_{s,x_i}\big\rangle \notag \\ 
  &+ \frac{2}{\epsilon} \bE\big\langle
  \nabla_{y} g(x_s, y_s) y_{s,x_i} - \nabla_{y} g(x_{j\Delta}, \hat{y}_s)
  \hat{y}_{s,x_i}, y_{s,x_i} - \hat{y}_{s,x_i}\big\rangle \notag \\ 
  &+ \frac{1}{\epsilon} \bE\Big(\big\|\nabla_{x} \alpha_2(x_s, y_s)\, x_{s,x_i} + \nabla_{y}
  \alpha_2(x_s, y_s)\, y_{s,x_i} - \nabla_{x} \alpha_2(x_{j\Delta}, \hat{y}_s)\,
  x_{j\Delta,x_i} - \nabla_{y}\alpha_2(x_{j\Delta}, \hat{y}_s)
  \hat{y}_{s,x_i}\big\|^2\Big)\,.
  \notag
\end{align}
We estimate each terms using H\"older and Young's inequality (\ref{hodler-young-ineq}).
For the first term,  
\begin{align}
  &\bE\big\langle
  \nabla_{x} g(x_s, y_s) x_{s,x_i} - \nabla_{x} g(x_{j\Delta}, \hat{y}_s)
  x_{j\Delta,x_i}, y_{s,x_i} - \hat{y}_{s,x_i}\big\rangle \notag \\ 
 =& \bE\big\langle
  \big(\nabla_{x} g(x_s, y_s) - \nabla_{x} g(x_{j\Delta}, \hat{y}_s) \big) x_{s,x_i} + \nabla_{x} g(x_{j\Delta}, \hat{y}_s)
  \big(x_{s,x_i} -  x_{j\Delta,x_i}\big), y_{s,x_i} - \hat{y}_{s,x_i}\big\rangle \notag \\ 
  \le &
 \frac{4}{\lambda} \bE \big|\big(\nabla_{x} g(x_s, y_s) - \nabla_{x} g(x_{j\Delta}, \hat{y}_s)
  \big) x_{s,x_i}\big|^2  
  +  \frac{4}{\lambda}\bE \big|\nabla_{x} g(x_{j\Delta}, \hat{y}_s) \big(x_{s,x_i} - x_{j\Delta,x_i}\big)\big|^2
  + \frac{\lambda}{4} \bE|y_{s,x_i} - \hat{y}_{s,x_i}|^2 \notag \\
   \le& C\Big[\big(\bE |x_{s,x_i}|^4\big)^{1/2} \big(\bE|x_s - x_{j\Delta}|^4 + \bE|y_s -
  \hat{y}_s|^4\big)^{1/2} +
\bE|x_{s,x_i} -  x_{j\Delta,x_i}|^2\Big] + \frac{\lambda}{4} \bE|y_{s,x_i} - \hat{y}_{s,x_i}|^2\,.\notag
\end{align}
In a similar way, we find the second term : 
\begin{align}
  & \bE\big\langle
  \nabla_{y} g(x_s, y_s) y_{s,x_i} - \nabla_{y} g(x_{j\Delta}, \hat{y}_s)
   \hat{y}_{s,x_i}, y_{s,x_i} - \hat{y}_{s,x_i}\big\rangle \notag \\ 
  =& \bE\big\langle 
  \big(\nabla_{y} g(x_s, y_s) - \nabla_{y} g(x_{j\Delta}, \hat{y}_s)\big)y_{s,x_i}
  + \nabla_{y} g(x_{j\Delta}, \hat{y}_s) (y_{s,x_i}-\hat{y}_{s,x_i}),
  y_{s,x_i} - \hat{y}_{s,x_i}\big\rangle \notag\\ 
  \le& 
 C\Big[\big(\bE |y_{s,x_i}|^4\big)^{1/2} \big(\bE|x_s - x_{j\Delta}|^4 + \bE|y_s -
 \hat{y}_s|^4\big)^{1/2}\Big] + \frac{\lambda}{4}\bE|y_{s,x_i} -
 \hat{y}_{s,x_i}|^2\notag \\
 & + \bE\big\langle \nabla_{y} g(x_{j\Delta}, \hat{y}_s)
 (y_{s,x_i}-\hat{y}_{s,x_i}), y_{s,x_i} - \hat{y}_{s,x_i}\big\rangle \,. \notag
\end{align}
For the third term, 
\begin{align*}
  &\bE\Big(\big\|(\nabla_{x} \alpha_2(x_s, y_s)\, x_{s,x_i} + \nabla_{y}
  \alpha_2(x_s, y_s)\, y_{s,x_i} - \nabla_{x} \alpha_2(x_{j\Delta}, \hat{y}_s)\,
  x_{j\Delta,x_i} - \nabla_{y}\alpha_2(x_{j\Delta}, \hat{y}_s)\,
  \hat{y}_{s,x_i})\big\|^2\Big)  \\
  \le & 4\bE\Big(\big\|\big(\nabla_{x} \alpha_2(x_s, y_s) - \nabla_{x} \alpha_2(x_{j\Delta},
  \hat{y}_s)\big)\, x_{s,x_i}\big\|^2\Big) 
   + 4\bE\Big(\big\|\nabla_{x} \alpha_2(x_{j\Delta},
  \hat{y}_s) \big(x_{s,x_i}- x_{j\Delta,x_i}\big)\big\|^2\Big) \\
  & + 4\bE\Big(\big\| \big(\nabla_{y}
  \alpha_2(x_s, y_s)\, - \nabla_{y}\alpha_2(x_{j\Delta}, \hat{y}_s)\big) \, y_{s,x_i} \big\|^2\Big)   
   + 4\bE\Big(\big\| \nabla_{y}\alpha_2(x_{j\Delta}, \hat{y}_s)
  \big(y_{s,x_i} - \hat{y}_{s,x_i}\big)\big\|^2\Big)   \\
  \le & C\bE\Big|\big(|x_s - x_{j\Delta}| + |y_s - \hat{y}_s|\big)\, x_{s,x_i}\Big|^2 
   + C\bE\big|x_{s,x_i}- x_{j\Delta,x_i}\big|^2 \\
  & +  C\bE\Big|\big(|x_s - x_{j\Delta}| + |y_s - \hat{y}_s|\big)\, y_{s,x_i}\Big|^2 
   + 4\bE\Big(\big\| \nabla_{y}\alpha_2(x_{j\Delta}, \hat{y}_s)
  \big(y_{s,x_i} - \hat{y}_{s,x_i}\big)\big\|^2\Big)   \\
  \le &C\Big[\Big(\big(\bE |y_{s,x_i}|^4\big)^{1/2} + \big(\bE
  |x_{s,x_i}|^4\big)^{1/2}\Big)
  \Big(\bE|x_s - x_{j\Delta}|^4 + \bE|y_s -
\hat{y}_s|^4\Big)^{1/2} + \bE|x_{s,x_i} -  x_{j\Delta,x_i}|^2\Big]   \\
& + 4\bE\| \nabla_{y}
\alpha_2(x_{j\Delta}, \hat{y}_s)\, (y_{s,x_i} - \hat{y}_{s,x_i})\|^2\,.  
\end{align*}
Now combining the above estimates and applying Lemma~\ref{d-xi-4},
Lemma~\ref{x-diff-s-k}, Lemma~\ref{auxiliary-error} as well as inequality
(\ref{mixing}) in Assumption~\ref{assumption-2}, we conclude that 
\begin{align}
  &\frac{d\bE|y_{s,x_i} - \hat{y}_{s,x_i}|^2}{ds} \le - \frac{\lambda}{\epsilon}
  \bE|y_{s,x_i} - \hat{y}_{s,x_i}|^2 + \frac{C\Delta}{\epsilon}\,, \notag 
\end{align}
and the first part of the assertion follows from Gronwall's inequality. In the
same way, we can compute that 
\begin{align}
  &\frac{d\bE|x_{s,x_i} - \hat{x}_{s,x_i}|^2}{ds} \notag\\
  =& 2 \bE\big\langle
  \nabla_{x} f(x_s, y_s) x_{s,x_i} - \nabla_{x} f(x_{j\Delta}, \hat{y}_s)
  x_{k\Delta,x_i}, x_{s,x_i} - \hat{x}_{s,x_i}\big\rangle \notag\\ 
  &+ 2 \bE\big\langle
  \nabla_{y} f(x_s, y_s) y_{s,x_i} - \nabla_{y} f(x_{j\Delta}, \hat{y}_s)
  \hat{y}_{s,x_i}, x_{s,x_i} - \hat{x}_{s,x_i}\big\rangle \notag\\
  =& 2 \bE\big\langle
  \big(\nabla_{x} f(x_s, y_s) - \nabla_{x} f(x_{j\Delta}, \hat{y}_s)\big)
  x_{s, x_i}, x_{s,x_i} - \hat{x}_{s,x_i}\big\rangle 
  + 2 \bE\big\langle
  \nabla_{x} f(x_{j\Delta}, \hat{y}_s)
  \big(x_{s, x_i} - x_{k\Delta,x_i}\big), x_{s,x_i} - \hat{x}_{s,x_i}\big\rangle \notag\\ 
  &+ 2 \bE\big\langle
  \big(\nabla_{y} f(x_s, y_s) - \nabla_{y} f(x_{j\Delta}, \hat{y}_s)\big) y_{s,x_i}, x_{s,x_i} - \hat{x}_{s,x_i}\big\rangle \notag\\
  &+ 2 \bE\big\langle
\nabla_{y} f(x_{j\Delta}, \hat{y}_s)
   \big(y_{s,x_i} - \hat{y}_{s,x_i}\big), x_{s,x_i} - \hat{x}_{s,x_i}\big\rangle \notag\\
  \le & \bE
  \big|\big(\nabla_{x} f(x_s, y_s) - \nabla_{x} f(x_{j\Delta}, \hat{y}_s)\big)
  x_{s, x_i}\big|^2 + \bE|x_{s,x_i} - \hat{x}_{s,x_i}|^2  
  + C \bE\big|\big\langle
  x_{s, x_i} - x_{k\Delta,x_i}, x_{s,x_i} - \hat{x}_{s,x_i}\big\rangle\big| \notag\\ 
  &+ \bE\big|
  \big(\nabla_{y} f(x_s, y_s) - \nabla_{y} f(x_{j\Delta}, \hat{y}_s)\big)
  y_{s,x_i}\big|^2 + \bE|x_{s,x_i} - \hat{x}_{s,x_i}\big|^2\notag\\
  &+ C \bE\big|\big\langle
   y_{s,x_i} - \hat{y}_{s,x_i}, x_{s,x_i} - \hat{x}_{s,x_i}\big\rangle\big| \notag\\
   \le & 
  C\Big[\bE \big|\big(|x_s - x_{j\Delta}| + |y_s - \hat{y}_s|\big) x_{s, x_i}\big|^2   
  + \bE \big|\big(|x_s - x_{j\Delta}| + |y_s - \hat{y}_s|\big) y_{s, x_i}\big|^2  
   + \bE|x_{s,x_i} -  x_{j\Delta,x_i}|^2 \notag \\
   &+ \bE|y_{s,x_i} - \hat{y}_{s,x_i}|^2  +  \bE|x_{s,x_i} - \hat{x}_{s,x_i}|^2\Big] \notag\\
  \le &C\Big[\big((\bE |y_{s,x_i}|^4)^{1/2} + (\bE |x_{s,x_i}|^4)^{1/2}\big)
  \big(\bE|x_s - x_{j\Delta}|^4 + \bE|y_s -
\hat{y}_s|^4\big)^{1/2} + \bE|x_{s,x_i} -  x_{j\Delta,x_i}|^2 + \bE|y_{s,x_i} -
\hat{y}_{s,x_i}|^2\Big]  \notag \\
&+ C \bE|x_{s,x_i} - \hat{x}_{s,x_i}|^2 \notag\\
\le & C\Delta + C \bE|x_{s,x_i} - \hat{x}_{s,x_i}|^2\,, \notag
\end{align}
where Lemma~\ref{d-xi-4}, Lemma~\ref{x-diff-s-k}, Lemma~\ref{auxiliary-error},
as well as the first part of conclusion have been used to obtain the last
inequality. Now Gronwall's inequality implies the second part of the assertion.
\qed
\end{proof}

We continue our study by comparing the processes $\hat{x}_{s,x_i}$ with $\widetilde{x}_{s,x_i}$, where 
\begin{align}
d\widetilde{x}_{s,x_i} &= \nabla_x \widetilde{f}(\widetilde{x}_s)
  \widetilde{x}_{s, x_i} ds + \nabla_x
  \alpha_1(\widetilde{x}_s)\widetilde{x}_{s, x_i} dw^1_s \, .
\end{align}
Recalling (\ref{averaging-coeff}), we can write 
\begin{align}
  \begin{split}
  \widetilde{f}(\widetilde{x}_s) &= \bE^{\xi}\big[f(\widetilde{x}_s,
  \xi^{\widetilde{x}_s}_{t})\big]\,, \\
\nabla_x\widetilde{f}(\widetilde{x}_s) \widetilde{x}_{s, x_i} &= \bE^{\xi}\big[\nabla_{x}f(\widetilde{x}_s,
\xi^{\widetilde{x}_s}_{t}) + \nabla_{y}f(\widetilde{x}_s, \xi^{\widetilde{x}_s}_{t})
\xi^{\widetilde{x}_s}_{t,x}\big]\widetilde{x}_{s,x_i}, 
\end{split}
\label{chain-rule-wilde-f}
\end{align}
where $\xi^x_t$ is the stationary process defined in
Appendix~\ref{app-2}, $\xi^x_{t,x}$ is the derivative process of $\xi^x_t$ with
respect to $x \in \mathbb{R}^k$, and $\bE^{\xi}$ denotes the expectation \wrt the stationary
process.  We have  
\begin{lemma}
  Let $\Delta = \epsilon^{\frac{1}{2}}$ and Assumptions
  \ref{assumption-1}--\ref{assumption-3} be satisfied. Then there exists
  $C>0$, independent of $\epsilon$ and can be chosen uniformly for $x_0$,
  $y_0$ which are contained in some bounded domain of $\mathbb{R}^k \times \mathbb{R}^l$, such that 
\[
\max_{0\le s \le T} \bE|\hat{x}_{s,x_i} - \widetilde{x}_{s,x_i}|^2 \le
C\epsilon^{\frac{1}{4}}\,, \quad 1 \le i \le k\,.
\]
\label{auxiliary-limit-dx}
\end{lemma}
\begin{proof}
  Let $j = \floor{\frac{r}{\Delta}}$. By Ito's formula and equality
  (\ref{chain-rule-wilde-f}), we have 
\begin{align}
  & \bE|\hat{x}_{s,x_i} - \widetilde{x}_{s,x_i}|^2 \notag \\
  =& 
  2\int_{0}^s \bE\big\langle \nabla_{x} f(x_{j\Delta},
  \hat{y}_r)\, x_{j\Delta,x_i} + \nabla_{y}
  f(x_{j\Delta}, \hat{y}_r)\, \hat{y}_{r, x_i}  - \nabla_{x}\widetilde{f}(\widetilde{x}_r)
  \widetilde{x}_{r, x_i}, \hat{x}_{r,x_i} - \widetilde{x}_{r,x_i}\big\rangle\, dr \notag \\
  &+ \int_{0}^s \bE\big\|\nabla_{x} \alpha_1(x_r)\, x_{r,x_i} -
  \nabla_{x}\alpha_1(\widetilde{x}_r)\widetilde{x}_{r, x_i}) \big\|^2 \,dr \notag \\
  = &
  2\int_{0}^s \bE\big\langle \nabla_{x} f(x_{j\Delta},
  \hat{y}_r)\, x_{j\Delta,x_i} 
  - \bE^{\xi}\big(\nabla_{x}f(\widetilde{x}_r, \xi^{\widetilde{x}_r}_{t})\big) \widetilde{x}_{r,x_i}
  , \hat{x}_{r,x_i} - \widetilde{x}_{r,x_i}\big\rangle dr \notag \\ 
  &+ 2\int_{0}^s \bE\big\langle \nabla_{y} f(x_{j\Delta}, \hat{y}_r)\, \hat{y}_{r, x_i} - 
  \bE^{\xi}\big(\nabla_{y}f(\widetilde{x}_r, \xi^{\widetilde{x}_r}_{t})
 \xi^{\widetilde{x}_r}_{t,x}\big)\widetilde{x}_{r,x_i},  \hat{x}_{r,x_i} -
 \widetilde{x}_{r,x_i}\big\rangle dr \notag \\ 
  &+ \int_{0}^s \bE\big\|\nabla_x \alpha_1(x_r)\, x_{r,x_i} -
  \nabla_x\alpha_1(\widetilde{x}_r)\widetilde{x}_{r, x_i}) \big\|^2 dr \notag \\
  =&I_1 + I_2 + I_3\,. \notag
\end{align}
Using the notations in Appendix~\ref{app-2}, we can identify the process
$\hat{y}_r$ with $\xi^{x_{j\Delta}}_{j\Delta, r}$  and process
$\hat{y}_{r,x_i}$ with $ \xi^{x_{j\Delta}}_{j\Delta, r, x} x_{j\Delta, x_i}$. Then, the term $I_1$ on the
right hand side above can be recast as   
\begin{align}
  &\int_{0}^s \bE\big\langle \nabla_{x} f(x_{j\Delta}, \hat{y}_r)\, x_{j\Delta, x_i} - 
  \bE^{\xi}\big(\nabla_{x}f(\widetilde{x}_r, \xi^{\widetilde{x}_r}_{t})
  \big)\widetilde{x}_{r,x_i},  \hat{x}_{r,x_i} - \widetilde{x}_{r,x_i}\big\rangle \,dr \notag \\ 
 =& \int_{0}^s \bE\big\langle \nabla_{x} f(x_{j\Delta}, \xi^{x_{j\Delta}}_{j\Delta, r})x_{j\Delta, x_i}- 
 \bE^{\xi}\big(\nabla_{x}f(x_{j\Delta}, \xi^{x_{j\Delta}}_{t})
 \big)x_{j\Delta,x_i}, \hat{x}_{r,x_i} - \widetilde{x}_{r,x_i}\big\rangle dr \notag \\ 
 & +  \int_{0}^s \bE\big\langle \bE^{\xi}\big(\nabla_{x}f(\hat{x}_{r},
 \xi^{\hat{x}_{r}}_{t})\big)  
 (\hat{x}_{r,x_i} - \widetilde{x}_{r,x_i}), \hat{x}_{r,x_i} -
 \widetilde{x}_{r,x_i}\big\rangle dr \notag \\ 
 &+  \int_{0}^s \bE\big\langle \big[\bE^{\xi}\big(\nabla_{x}f(\hat{x}_{r},
 \xi^{\hat{x}_{r}}_{t}) \big) - \bE^{\xi}\big(\nabla_{x}f(\widetilde{x}_r, \xi^{\widetilde{x}_r}_{t})
 \big)\big]\widetilde{x}_{r,x_i},  \hat{x}_{r,x_i} - \widetilde{x}_{r,x_i}\big\rangle dr \notag \\ 
 & + \int_{0}^s \bE\big\langle
 \bE^{\xi}\big(\nabla_{x}f(x_{j\Delta}, \xi^{x_{j\Delta}}_{t})
 \big)x_{j\Delta,x_i} - \bE^{\xi}\big(\nabla_{x}f(\hat{x}_{r},
 \xi^{\hat{x}_{r}}_{t}) \big)\hat{x}_{r,x_i}  , \hat{x}_{r,x_i} -
 \widetilde{x}_{r,x_i}\big\rangle dr \notag \\ 
 =& I_{1,1} + I_{1,2} + I_{1,3} + I_{1,4} \,.\notag
\end{align}
For $I_{1,1}$, using Lemma~\ref{exp-converge} in Appendix~\ref{app-2}
 and Lemma~\ref{x-diff-s-k}, we have 
\begin{align}
  |I_{1,1}| \le & 
  \Big|\int_{0}^s \bE\big\langle \nabla_{x} f(x_{j\Delta}, \xi^{x_{j\Delta}}_{j\Delta, r})x_{j\Delta, x_i}- 
 \bE^{\xi}\big(\nabla_{x}f(x_{j\Delta}, \xi^{x_{j\Delta}}_{t})
 \big)x_{j\Delta,x_i},
 \hat{x}_{j\Delta,x_i} - \widetilde{x}_{j\Delta,x_i}\big\rangle dr\Big| \notag \\ 
 &+ 
  \Big|\int_{0}^s \bE\big\langle \nabla_{x} f(x_{j\Delta}, \xi^{x_{j\Delta}}_{j\Delta, r})x_{j\Delta, x_i}- 
 \bE^{\xi}\big(\nabla_{x}f(x_{j\Delta}, \xi^{x_{j\Delta}}_{t})
 \big)x_{j\Delta,x_i},
 \hat{x}_{r,x_i} - \hat{x}_{j\Delta,x_i} \big\rangle dr\Big| \notag \\ 
 &+ \Big|\int_{0}^s \bE\big\langle \nabla_{x} f(x_{j\Delta}, \xi^{x_{j\Delta}}_{j\Delta, r})x_{j\Delta, x_i}- 
 \bE^{\xi}\big(\nabla_{x}f(x_{j\Delta}, \xi^{x_{j\Delta}}_{t})
 \big)x_{j\Delta,x_i},
 \widetilde{x}_{r,x_i} - \widetilde{x}_{j\Delta,x_i} \big\rangle dr\Big| \notag \\ 
 \le & 
 C\sum_{j=0}^{\floor{s / \Delta}} \int_{j\Delta}^{[(j+1)\Delta]\wedge s}
 \bE\Big( \big(1+|x_{j\Delta}| + |\hat{y}_{j\Delta}|\big)
 \big|x_{j\Delta, x_i}\big|\big|\hat{x}_{j\Delta,x_i} -
 \widetilde{x}_{j\Delta,x_i}\big|\Big) e^{-\frac{\lambda(r-j\Delta)}{\epsilon}} dr +
 Cs \Delta^{\frac{1}{2}} \notag \\
 \le & \frac{C\epsilon}{\lambda} \sum_{j=0}^{\floor{s / \Delta}} 
 \Big[
 \bE\big(1+|x_{j\Delta}| + |\hat{y}_{j\Delta}|\big)^4 + \bE\big|x_{j\Delta,
 x_i}\big|^4  + \bE\big|\hat{x}_{j\Delta,x_i} -
 \widetilde{x}_{j\Delta,x_i}\big|^2\Big] +
 Cs \Delta^{\frac{1}{2}} \notag \\
 \le & \frac{C\epsilon}{\lambda} \sum_{j=0}^{\floor{s / \Delta}} \bE|\hat{x}_{j\Delta,x_i} -
 \widetilde{x}_{j\Delta,x_i}|^2 + Cs(\Delta^{\frac{1}{2}} +
 \frac{\epsilon}{\Delta})
 \le \frac{C\epsilon}{\lambda \Delta} \int_{0}^s \bE|\hat{x}_{r,x_i} -
 \widetilde{x}_{r,x_i}|^2 dr + Cs(\Delta^{\frac{1}{2}} +
 \frac{\epsilon}{\Delta})\,, \notag 
\end{align}
where the $4$th order estimates in Lemma~\ref{lemma-4th-stability},
Lemma~\ref{d-xi-4},
 as well as Remark~\ref{rmk-2} are used in the last
two inequalities.
For $I_{1,2}$, since function $f$ is Lipschitz, it follows that 
\begin{align}
  |I_{1,2}| \le 
  C\int_{0}^s \bE|\hat{x}_{r,x_i} - \widetilde{x}_{r,x_i}|^2 dr\,. \notag 
\end{align}
For $I_{1,3}$, Lemma~\ref{lemma-stationary} implies that  
\begin{align}
   \Big|\bE^{\xi}\big(\nabla_{x}f(\hat{x}_{r}, \xi^{\hat{x}_{r}}_{t})\big) - \bE^{\xi}\big(\nabla_{x}f(\widetilde{x}_r, \xi^{\widetilde{x}_r}_{t}) \big)\Big| 
\le  
  C\bE^{\xi}\big(|\hat{x}_{r}-\widetilde{x}_r| +  |\xi^{\hat{x}_{r}}_{t} - \xi^{\widetilde{x}_r}_{t}|\big) 
\le & C|\hat{x}_{r} - \widetilde{x}_{r}|\,, \notag
\end{align}
and therefore using inequality (\ref{hodler-young-ineq}), 
\begin{align}
  |I_{1,3}|
  \le & C\int_{0}^s \bE\big(|\hat{x}_{r} - \widetilde{x}_{r}|\,|\widetilde{x}_{r,x_i}|\,|\hat{x}_{r,x_i} -
  \widetilde{x}_{r,x_i}|\big)\,dr \notag \\ 
  \le & C \int_{0}^s \bE|\hat{x}_{r,x_i} - \widetilde{x}_{r,x_i}|^2 dr +  
  C\int_{0}^s \big(\bE|\hat{x}_{r} -
  \widetilde{x}_{r}|^4\big)^{\frac{1}{2}}
  \big(\bE|\widetilde{x}_{r,x_i}|^4\big)^{\frac{1}{2}} dr \notag \\ 
  \le & C \int_{0}^s \bE|\hat{x}_{r,x_i} - \widetilde{x}_{r,x_i}|^2 dr + 
  C\int_{0}^s \big(\bE|\hat{x}_{r} -
  \widetilde{x}_{r}|^4\big)^{\frac{1}{2}} dr\,. \notag 
  \notag
\end{align}
The remaining term $I_{1,4}$ can be estimated in pretty much the same way as $I_{1,2}$ and $I_{1,3}$: 
\begin{align}
  |I_{1,4}| 
  \le & C\int_{0}^s \bE\Big(|x_{j\Delta ,x_i} - \hat{x}_{r,x_i}| |\hat{x}_{r,x_i} - \widetilde{x}_{r,x_i}|\Big)dr 
   + C\int_{0}^s \bE\Big(|x_{j\Delta} - \hat{x}_{r}|\,|\hat{x}_{r,x_i}|\,|\hat{x}_{r,x_i} -
  \widetilde{x}_{r,x_i}|\Big)dr \notag \\ 
  \le & C \int_{0}^s \bE|\hat{x}_{r,x_i} - \widetilde{x}_{r,x_i}|^2 dr +
C \int_{0}^s \bE|x_{j\Delta ,x_i} - \hat{x}_{r,x_i}|^2 dr + 
C \int_{0}^s 
 \bE\Big(|x_{j\Delta} - \hat{x}_{r}|^2\,|\hat{x}_{r,x_i}|^2\Big) dr \notag \\
  \le & C \int_{0}^s \bE|\hat{x}_{r,x_i} - \widetilde{x}_{r,x_i}|^2 dr +
C \int_{0}^s \bE|x_{j\Delta ,x_i} - \hat{x}_{r,x_i}|^2 dr + 
C \int_{0}^s \big(\bE|x_{j\Delta} - \hat{x}_{r}|^4\big)^{\frac{1}{2}} dr \notag \\
  \le & C \int_{0}^s \bE|\hat{x}_{r,x_i} - \widetilde{x}_{r,x_i}|^2 dr +
C \int_{0}^s \bE|x_{j\Delta ,x_i} - x_{r,x_i}|^2 dr + 
C \int_{0}^s \bE|x_{r,x_i} - \hat{x}_{r,x_i}|^2 dr \notag \\
& + C \int_{0}^s \big(\bE|x_{j\Delta} - x_{r}|^4\big)^{\frac{1}{2}} dr 
+ C \int_{0}^s \big(\bE|x_{r}-\hat{x}_r|^4\big)^{\frac{1}{2}} dr \notag \\
  \le & C \int_{0}^s \bE|\hat{x}_{r,x_i} - \widetilde{x}_{r,x_i}|^2 dr + Cs\Delta\,, \notag
\end{align}
where the last inequality follows from Lemma~\ref{x-diff-s-k}, Lemma~\ref{auxiliary-error}
and Lemma~\ref{average-auxiliary-xy}. 

We proceed with $I_2$. Similarly as $I_1$, we have 
\begin{align}
  &\int_{0}^s \bE\langle \nabla_{y} f(x_{j\Delta}, \hat{y}_r)\, \hat{y}_{r, x_i} - 
  \bE^{\xi}\big(\nabla_{y}f(\widetilde{x}_r, \xi^{\widetilde{x}_r}_{t})
 \xi^{\widetilde{x}_r}_{t,x}\big)\widetilde{x}_{r,x_i},  \hat{x}_{r,x_i} - \widetilde{x}_{r,x_i}\rangle dr \notag \\ 
 =& \int_{0}^s \bE\langle \nabla_{y} f(x_{j\Delta}, \xi^{x_{j\Delta}}_{j\Delta, r})\xi^{x_{j\Delta}}_{j\Delta, r, x}x_{j\Delta, x_i}- 
 \bE^{\xi}\big(\nabla_{y}f(x_{j\Delta}, \xi^{x_{j\Delta}}_{t})
 \xi^{x_{j\Delta}}_{t,x}\big)x_{j\Delta,x_i}, \hat{x}_{r,x_i} - \widetilde{x}_{r,x_i}\rangle dr \notag \\ 
 & +  \int_{0}^s \bE\langle \bE^{\xi}\big(\nabla_{y}f(\hat{x}_{r},
 \xi^{\hat{x}_{r}}_{t}) \xi^{\hat{x}_{r}}_{t,x}\big)  
 (\hat{x}_{r,x_i} - \widetilde{x}_{r,x_i}), \hat{x}_{r,x_i} - \widetilde{x}_{r,x_i}\rangle dr \notag \\ 
 &+  \int_{0}^s \bE\langle \big[\bE^{\xi}\big(\nabla_{y}f(\hat{x}_{r},
 \xi^{\hat{x}_{r}}_{t}) \xi^{\hat{x}_{r}}_{t,x}\big) - \bE^{\xi}\big(\nabla_{y}f(\widetilde{x}_r, \xi^{\widetilde{x}_r}_{t})
 \xi^{\widetilde{x}_r}_{t,x}\big)\big]\widetilde{x}_{r,x_i},  \hat{x}_{r,x_i} - \widetilde{x}_{r,x_i}\rangle dr \notag \\ 
 & + \int_{0}^s \bE\langle
 \bE^{\xi}\big(\nabla_{y}f(x_{j\Delta}, \xi^{x_{j\Delta}}_{t})
 \xi^{x_{j\Delta}}_{t,x}\big)x_{j\Delta,x_i} - \bE^{\xi}\big(\nabla_{y}f(\hat{x}_{r},
 \xi^{\hat{x}_{r}}_{t}) \xi^{\hat{x}_{r}}_{t,x}\big)\hat{x}_{r,x_i}  , \hat{x}_{r,x_i} -
 \widetilde{x}_{r,x_i}\rangle dr \notag \\ 
 =& I_{2,1} + I_{2,2} + I_{2,3} + I_{2,4}\,.
\end{align}
Using Lemma~\ref{exp-converge} and Lemma~\ref{lemma-stationary}, we can
estimate the above four terms similarly as terms $I_{1,1}$ to $I_{1,4}$, and obtain 
\begin{align}
  I_{2,1} \le& \frac{C\epsilon}{\lambda \Delta} \int_{0}^s \bE|\hat{x}_{r,x_i} -
  \widetilde{x}_{r,x_i}|^2 dr + Cs(\Delta^{\frac{1}{2}} +
  \frac{\epsilon}{\Delta})\,, \notag
 \\
 I_{2,2} \le& C\int_{0}^s \bE|\hat{x}_{r,x_i} - \widetilde{x}_{r,x_i}|^2 dr\,,  \notag  \\
 I_{2,3} \le & C \int_{0}^s \bE|\hat{x}_{r,x_i} - \widetilde{x}_{r,x_i}|^2 dr + 
  C\int_{0}^s \big(\bE|\hat{x}_{r} -
  \widetilde{x}_{r}|^4\big)^{\frac{1}{2}} dr\,, \notag \\
 I_{2,4} \le & C \int_{0}^s \bE|\hat{x}_{r,x_i} - \widetilde{x}_{r,x_i}|^2 dr + Cs\Delta\,. \notag
\end{align}

For $I_3$, Lemma~\ref{d-xi-4}, Lemma~\ref{auxiliary-error} and the assumption that $\alpha_1$ is Lipschitz entail 
\begin{align}
  |I_3| \le & 
  3\int_{0}^s \bE\|(\nabla_{x} \alpha_1(x_r) -
  \nabla_{x}\alpha_1(\widetilde{x}_r)) x_{r, x_i} \|^2 dr 
  + 3\int_{0}^s \bE\|\nabla_{x} \alpha_1(\widetilde{x}_r)(x_{r,x_i} - \hat{x}_{r, x_i})
  \|^2 dr \notag \\
 & +
  3 \int_{0}^s \bE\|\nabla_{x} \alpha_1(\widetilde{x}_r)(\hat{x}_{r,x_i} -\widetilde{x}_{r, x_i}) \|^2 dr \notag \\
  \le & C\int_{0}^s \bE\big(|x_r - \widetilde{x}_r|^2|x_{r, x_i}|^2\big) dr 
  + C\int_{0}^s \bE|x_{r,x_i} - \hat{x}_{r, x_i}|^2 dr 
  + C\int_{0}^s \bE|\hat{x}_{r,x_i} -\widetilde{x}_{r, x_i}|^2 dr \notag \\
  \le & 
  C\int_{0}^s \bE|\hat{x}_{r,x_i} -\widetilde{x}_{r, x_i}|^2 dr +  
  C \int_{0}^s \big(\bE|x_r -\widetilde{x}_r|^4\big)^{\frac{1}{2}} \big(\bE|x_{r,
  x_i}|^4\big)^{\frac{1}{2}} dr + Cs\Delta \notag \\
  \le & 
  C\int_{0}^s \bE|\hat{x}_{r,x_i} -\widetilde{x}_{r, x_i}|^2 dr +  
  C \int_{0}^s \big(\bE|\hat{x}_r -\widetilde{x}_r|^4\big)^{\frac{1}{2}} dr + 
  Cs\Delta\,. \notag
\end{align}
Upon combining the bounds for $I_1$, $I_2$ and $I_3$, we conclude that 
\begin{align*}
  \bE|\hat{x}_{s,x_i} - \widetilde{x}_{s,x_i}|^2 
  \le &  
  C(1 + \frac{\epsilon}{\lambda\Delta}) \int_{0}^s \bE|\hat{x}_{r,x_i} -\widetilde{x}_{r, x_i}|^2 dr \\
  & +  
  C \int_{0}^s \big(\bE|\hat{x}_r -\widetilde{x}_r|^4\big)^{\frac{1}{2}} dr + 
  Cs(\Delta + \Delta^{\frac{1}{2}} + \frac{\epsilon}{\Delta})\,.
\end{align*}
Now letting $\Delta = \epsilon^{\frac{1}{2}}$ and using Lemma~\ref{lemma-hatx-tildex},  it follows that 
\begin{align}
  \bE|\hat{x}_{s,x_i} - \widetilde{x}_{s,x_i}|^2 
  \le 
  C\int_{0}^s \bE|\hat{x}_{r,x_i} -\widetilde{x}_{r, x_i}|^2 dr +  
  Cs\epsilon^{\frac{1}{4}} \notag
\end{align}
and Gronwall's inequality yields the conclusion.
\qed
\end{proof}

Combining Lemma~\ref{average-auxiliary-xy} and Lemma~\ref{auxiliary-limit-dx}, we have proved:

\begin{theorem}
  Suppose that Assumptions \ref{assumption-1}--\ref{assumption-3} hold.
  Then there exists $C>0$, independent of $\epsilon$ and can be chosen
  uniformly for $x_0$, $y_0$ which are contained in some bounded domain
  of $\mathbb{R}^k \times \mathbb{R}^l$, such that 
    \begin{align}
      \max_{1 \le s \le T} \bE|x_{s,x_i} - \widetilde{x}_{s,x_i}|^2 \le C \epsilon^{\frac{1}{4}}\,. \notag
    \end{align}
  \label{thm-average-limit-xs_dx}
\end{theorem}

\section{Conclusions}
\label{sec-discuss}
Importance sampling is a widely used variance reduction technique 
for the design of efficient Monte Carlo estimators. A crucial point in order to achieve substantial variance reduction is 
a clever (and careful) change of measure. In the diffusion process setting, this change of measure can be 
realized by adding a control force to the original system, where the optimal control that leads to a  \emph{zero-variance 
estimator} is related to a Hamilton-Jacobi-Bellman (HJB) equation that may not be easily solvable, e.g.~when the state space is high-dimensional.   

Our starting point is that even it may not be possible to compute the \emph{optimal} control, 
it is possible to approximate it in such a way that the resulting estimators remain efficient.
In the case of exponential type expectations and for multiscale diffusions with 
both slow and fast variables, the asymptotic optimality of the approximation 
based on a low-dimensional averaged equation has been proved and an upper bound
for the relative error of the importance sampling estimator has been obtained.
We expect our results to be helpful for the design of importance sampling methods
as well as for the study of multiscale diffusion processes.

%
%

There are many possible extensions related to the current work. 
For the theoretical aspects, our main result concerns the time scale separation limit ($\epsilon
\rightarrow 0$) for diffusion with
slow and fast variables and assumes the temperature $\beta$ is fixed. As a result, the constant 
in Theorem~\ref{main_thm} may depend on $\beta$.
It is interesting to consider asymptotics for both parameters
$\epsilon,\beta$ together.
Generalizing our results to dynamics with non-Lipschitz coefficients as well as to more
general types of dynamics is also important. 
For the numerical aspects, 
realistic systems in climate science, molecular dynamics may be
high-dimensional and even the averaged equation cannot be easily discretized 
and solved by usual grid-based methods. In more general situations, 
it may be impossible to separate systems' states into slow and fast ones
with an explicit time scale separation parameter. 
We leave these questions for future work and refer to \cite{ce_paper2014,HaEtalJCD2014} for some 
recent algorithmic and methodological developments in this regard. 

\section*{Acknowledgements}
The authors acknowledge financial support by the DFG Research Center {\sc Matheon}, the Einstein Center for Mathematics {\sc ECMath} and the DFG-CRC 1114 ``Scaling Cascades in Complex
Systems''. Special thanks also go to anonymous referees whose valuable
comments and criticism have helped to improve this paper.


\appendix
\section{Two useful inequalities}
\label{app-1}
\begin{clm}
  \label{claim-2}
  Consider functions $x_1(t), x_2(t)$ on $t \in [0,T]$ satisfying 
  \begin{align*}
    \dot{x}_1(t) \le a_{11}\,x_1(t) + a_{12}\,x_2(t) \\
    \dot{x}_2(t) \le \frac{a_{21}}{\epsilon} x_1(t) - \frac{a_{22}}{\epsilon} x_2(t) 
  \end{align*}
  with $x_1(0) = 0, x_2(0) = 1$, $a_{ij} > 0$, $1\le i, j \le 2$.
  Further assume that $x_1(t) \ge 0$ for all $t\in [0,T]$. Then there is a constant $C > 0$ depending on $a_{ij}$ and $T$, such that  
  \begin{align}
    \max_{0\le s \le T} x_1(s) \le C\epsilon\,, \qquad x_2(t) \le
    e^{-\frac{a_{22} t}{\epsilon}} + C\epsilon\,, \quad t \in [0, T].
  \end{align}
\end{clm}
\begin{proof}
   Applying Gronwall's inequality to the equation of $x_2$, we have 
  \begin{align}
    x_2(t) &\le e^{-\frac{a_{22}t}{\epsilon}} + \int_0^t
    e^{-\frac{a_{22}}{\epsilon}(t-s)} \frac{a_{21}}{\epsilon} x_1(s) ds \notag \\
    &\le e^{-\frac{a_{22}t}{\epsilon}} + \frac{a_{21}}{a_{22}} \max_{0 \le s
  \le t} x_1(s)\,. \label{x1-x2}
  \end{align}
  Applying Gronwall's inequality to $x_1$ and using 
  (\ref{x1-x2}), we find 
  \begin{align}
    x_1(t) \le a_{12} \int_0^t e^{a_{11}(t-s)} \Big[e^{-\frac{a_{22}s}{\epsilon}}
    + \frac{a_{21}}{a_{22}} \big(\max_{0 \le r \le s} x_1(r)\big)\Big] ds\,.
  \end{align}
  Since the right hand side in the last inequality is monotonically increasing
  (as a function of $t$), it follows that 
  \begin{align}
    \max_{0\le s \le t} x_1(s) &\le a_{12} \int_0^t e^{a_{11}(t-s)} \Big[e^{-\frac{a_{22}s}{\epsilon}}
    + \frac{a_{21}}{a_{22}} \big(\max_{0 \le r \le s} x_1(r)\big)\Big] ds \notag \\
    & \le \frac{a_{12}}{a_{22}} e^{a_{11}T}\epsilon +
    \frac{a_{12}a_{21}}{a_{22}} \int_0^t e^{a_{11}(t-s)} \big(\max_{0\le r \le s} x_1(r)\big) ds\,.
  \end{align}
  The first part of the assertion then follows by applying Gronwall's inequality in integral form 
  to $\max\limits_{0\le s \le t} x_1(s)$, while the second part is obtained using
  (\ref{x1-x2}). 
  \qed
\end{proof}

For $0 < \epsilon < 1$, we set $t_1 = -\frac{2\epsilon\ln \epsilon}{\lambda} > 0$ 
and introduce the function $\gamma\colon [0,T] \rightarrow [0, 1]$ by 
\begin{align}
  \gamma(t) =\left\{
  \begin{array}{cl}
    1- \frac{t}{t_1} & \qquad 0 \le t \le t_1 	\\
    0 & \qquad t_1 < t \le T\,.
  \end{array}
  \right.
  \label{eta}
\end{align}

\begin{clm}
  \label{claim-3}
  Consider functions $x_1(t), x_2(t)$ on $t \in [0, T]$ satisfying  
  \begin{align*}
    \dot{x}_1(t) &\le a_1(1 + \epsilon^{-\gamma(t)}) x_1(t) + a_2\epsilon^{\gamma(t)} x_2(t) \\
    \dot{x}_2(t) &\le \frac{a_3x_1(t)}{\epsilon} -\frac{\lambda x_2(t)}{\epsilon} \,,
  \end{align*}
  where $\gamma$ is given in (\ref{eta}), $a_i \ge 0, 1\le i \le 3$, and $x_1(0) = 0, x_2(0) = 1$.
  Further assume that $x_1(t) \ge 0$ on $t \in [0,T]$.
  Then there is a constant $C > 0$ independent of $\epsilon$, such that  
  \begin{align}
    \max_{0\le s \le T} x_1(s) \le C\epsilon^2, \qquad x_2(t) \le
    e^{-\frac{\lambda t}{\epsilon}} + C\epsilon^2\,, \quad t \in [0,T]\,.
\end{align}
\end{clm}
\begin{proof}
  As in Claim~\ref{claim-2}, we can obtain
  \begin{align}
    x_2(t) & \le e^{-\frac{\lambda t}{\epsilon}} +
    \frac{a_3}{\lambda} \max_{0 \le s \le t} x_1(s) \label{x2-x1-1}\\
    \max_{0\le s \le t} x_1(s) & \le a_2\int_0^t e^{a_1\int_s^t (1 +
    \epsilon^{-\gamma(r)}) dr} \epsilon^{\gamma(s)} \big[e^{-\frac{\lambda
    s}{\epsilon}} + \frac{a_3}{\lambda}\big(\max_{0 \le r \le s} x_1(r)\big)\big] ds \,.
  \end{align}
 Then, for $t < t_1$, the second inequality above implies 
  \begin{align}
    \max_{0\le s \le t} x_1(s) \le C \epsilon^2 + \frac{a_2a_3}{\lambda}
    \int_0^{t} e^{a_1\int_s^t (1 +
    \epsilon^{-\gamma(r)}) dr} \epsilon^{\gamma(s)} \big(\max_{0 \le r \le s} x_1(r)\big) ds \,.
  \end{align}
Using (\ref{x2-x1-1})  and Gronwall's inequality again, we conclude that 
\begin{align}
  \max_{0\le s \le t_1} x_1(s) \le C\epsilon^2, \quad x_2(t) \le
  e^{-\frac{\lambda t}{\epsilon}} + C\epsilon^2 , \qquad t \le t_1\,.
  \label{claim-2-3}
\end{align}
Repeating the above argument for $t\in[t_1, T]$, noticing that $x_1(t_1) \le C\epsilon^2$,
$x_2(t_1) \le C\epsilon^2$, $\gamma(t) \equiv 0, t \in [t_1, T]$, it follows that  
\begin{align}
  \max_{t_1\le s \le T} x_1(s) \le C\epsilon^2, \quad x_2(t) \le C\epsilon^2 ,
  \qquad t \in [t_1, T]\,.
  \label{claim-2-4}
\end{align}
The proof is completed by combining (\ref{claim-2-3}) and (\ref{claim-2-4}).
\qed
\end{proof}

\section{Properties of the stationary process}
\label{app-2}

For fixed $x \in \mathbb{R}^k$ and $\tau \in \mathbb{R}$, we introduce the process
\begin{align}
  d\xi^x_{\tau, s} &= \frac{1}{\epsilon} g(x,\xi^x_{\tau, s}) ds +
  \frac{1}{\sqrt{\epsilon}}\alpha_2(x,\xi^x_{\tau, s}) dw_s\,, \quad s\ge
\tau\,, \quad
  \xi^x_{\tau, \tau} = y
\end{align}
where $w_s$ is a standard Wiener process in $\mathbb{R}^{m_2}$.
In the following, we summarize some properties related to the above
process that we called \emph{the fast subsystem} in Section \ref{sec-main}. See also \cite{liu2010,da1996ergodicity} for additional results.
\begin{lemma}
  Under Assumptions~\ref{assumption-1}--\ref{assumption-2}, there exists a constant
  $C > 0$, independent of $\epsilon, x, y$, such that: 
  \begin{enumerate}
    \item
  $\bE|\xi^x_{\tau, s}|^4 \le e^{-\frac{\lambda (s-\tau)}{\epsilon}} |y|^4 +
  C\big(|x|^4 + 1\big)$.
    \item 
      For $\tau_1 \le \tau_2$, it holds 
      \[\bE|\xi^x_{\tau_2, s} - \xi^x_{\tau_1, s}|^4 \le
  C\big(|x|^4 + |y|^4+1\big)\,e^{-\frac{4\lambda(s-\tau_2)}{\epsilon}}\,,
  \quad s \ge \tau_2\,.
  \]
\item For $x, x' \in \mathbb{R}^k$ and $\tau_1 \le \tau_2$, 
  \begin{align}
    \bE|\xi^{x'}_{\tau_2, s} - \xi^{x}_{\tau_1, s}|^4\le
  e^{-\frac{2\lambda(s-\tau_2)}{\epsilon}} \big(|x|^4 + |y|^4 + 1\big) + C |x' -
  x|^4\,, \quad s \ge \tau_2\,.
  \notag
  \end{align}
  \end{enumerate}
  \label{lemma-stat-bound}
\end{lemma}
\begin{proof}
  \begin{enumerate}
    \item
  By Ito's formula,  we have 
  \begin{align}
    \frac{d\bE|\xi^x_{\tau, s}|^4}{ds} =& 
    \frac{1}{\epsilon}
    \bE\Big[|\xi^x_{\tau, s}|^2 \big(4\langle
      g(x, \xi^x_{\tau, s}), \xi^x_{\tau, s}\rangle +
      2\|\alpha_2(x,\xi^x_{\tau, s})\|^2\big) + 4|\alpha^T_2(x,\xi^x_{\tau,
  s})\xi^x_{\tau, s}|^2\Big] \notag \\
  \le& 
\frac{1}{\epsilon} \bE\Big[|\xi^x_{\tau, s}|^2 \big(4\langle
      g(x, \xi^x_{\tau, s}), \xi^x_{\tau, s}\rangle +
    6\|\alpha_2(x,\xi^x_{\tau, s})\|^2\big)\Big]\,.  \notag 
  \end{align}
  Applying inequality (\ref{mixing-1}) in Remark~\ref{rmk-1} and inequality
  (\ref{hodler-young-ineq}), we obtain
  \begin{align}
    \frac{d\bE|\xi^x_{\tau, s}|^4}{ds} 
    \le & -\frac{2\lambda}{\epsilon} \bE|\xi^x_{\tau, s}|^4 + \frac{C}{\epsilon}
    \bE\Big[|\xi^x_{\tau, s}|^2 (|x|^2 + 1)\Big] \notag \\
    \le & -\frac{\lambda}{\epsilon} \bE|\xi^x_{\tau, s}|^4 + \frac{C}{\epsilon}
    \Big(|x|^4 + 1\Big)\,,  \notag 
  \end{align}
  and the first statement follows from Gronwall's inequality.  
\item For the second statement, using Ito's formula and
  Assumption~\ref{assumption-2}, it follows
  \begin{align}
    & \frac{d\bE|\xi^x_{\tau_2, s} - \xi^x_{\tau_1, s}|^4}{ds} \notag \\
=& \frac{1}{\epsilon}
    \bE\Big[
      |\xi^x_{\tau_2, s} - \xi^x_{\tau_1, s}|^2 \big(4\langle
      g(x, \xi^x_{\tau_2, s}) - g(x, \xi^x_{\tau_1, s}), \xi^x_{\tau_2, s} - \xi^x_{\tau_1, s}
      \rangle \notag \\
      &+ 2\|\alpha_2(x,\xi^x_{\tau_2, s})-\alpha_2(x,\xi^x_{\tau_1, s})\|^2\big) 
      + 4\big|\big(\alpha_2(x,\xi^x_{\tau_2, s})- \alpha_2(x,\xi^x_{\tau_1,
    s})\big)^T\big(\xi^x_{\tau_2, s} - \xi^x_{\tau_1, s}\big)\big|^2\Big] \notag \\
\le& \frac{1}{\epsilon}
    \bE\Big[
      |\xi^x_{\tau_2, s} - \xi^x_{\tau_1, s}|^2 \big(4\langle
      g(x, \xi^x_{\tau_2, s}) - g(x, \xi^x_{\tau_1, s}), \xi^x_{\tau_2, s} - \xi^x_{\tau_1, s}
      \rangle + 6\|\alpha_2(x,\xi^x_{\tau_2, s})-\alpha_2(x,\xi^x_{\tau_1, s})\|^2\big) \Big] \notag \\
\le&
-\frac{4\lambda}{\epsilon} \bE|\xi^x_{\tau_2, s} - \xi^x_{\tau_1, s}|^4\,. \notag
  \end{align}
  Therefore, integrating and using the first statement above, we obtain
  \begin{align}
    \bE|\xi^x_{\tau_2, s} - \xi^x_{\tau_1, s}|^4 \le
  e^{-\frac{4\lambda(s-\tau_2)}{\epsilon}} \bE|\xi^x_{\tau_1, \tau_2} - y|^4 \le
C\big(1+|x|^4 + |y|^4\big)e^{-\frac{4\lambda(s-\tau_2)}{\epsilon}}\,. \notag
  \end{align}
\item For the third statement, in a similar way, applying Ito's formula, using
  Assumption~\ref{assumption-2}, as well as Lipschitz property of
   functions $g$ and $\alpha_2$, we have  
  \begin{align}
    &    \frac{d\bE|\xi^{x'}_{\tau_2, s} - \xi^{x}_{\tau_1, s}|^4}{ds} \notag \\
 =& \frac{1}{\epsilon}
    \bE\Big[
      |\xi^{x'}_{\tau_2, s} - \xi^x_{\tau_1, s}|^2 \big(4\langle
      g(x', \xi^{x'}_{\tau_2, s}) - g(x, \xi^x_{\tau_1, s}), \xi^{x'}_{\tau_2, s} - \xi^x_{\tau_1, s}
      \rangle \notag \\
      &+ 2\|\alpha_2(x',\xi^{x'}_{\tau_2, s})-\alpha_2(x,\xi^x_{\tau_1, s})\|^2\big) 
      + 4\big|\big(\alpha_2(x',\xi^{x'}_{\tau_2, s})- \alpha_2(x,\xi^x_{\tau_1,
    s})\big)^T\big(\xi^{x'}_{\tau_2, s} - \xi^x_{\tau_1, s}\big)\big|^2\Big] \notag \\
    \le & \frac{1}{\epsilon}
    \bE\Big[
      |\xi^{x'}_{\tau_2, s} - \xi^x_{\tau_1, s}|^2 \big(4\langle
      g(x', \xi^{x'}_{\tau_2, s}) - g(x, \xi^x_{\tau_1, s}), \xi^{x'}_{\tau_2, s} - \xi^x_{\tau_1, s}
      \rangle + 6\|\alpha_2(x',\xi^{x'}_{\tau_2, s})-\alpha_2(x,\xi^x_{\tau_1, s})\|^2\big) \Big] \notag \\
      \le & 
 \frac{1}{\epsilon}
    \bE\Big[
      |\xi^{x'}_{\tau_2, s} - \xi^x_{\tau_1, s}|^2 \big(4\langle
      g(x', \xi^{x'}_{\tau_2, s}) - g(x', \xi^x_{\tau_1, s}), \xi^{x'}_{\tau_2, s} - \xi^x_{\tau_1, s}
      \rangle + 12\|\alpha_2(x',\xi^{x'}_{\tau_2, s})-\alpha_2(x',\xi^x_{\tau_1, s})\|^2\big) \Big] \notag \\
      & + 
 \frac{1}{\epsilon}
    \bE\Big[
      |\xi^{x'}_{\tau_2, s} - \xi^x_{\tau_1, s}|^2 \big(4\langle
      g(x', \xi^{x}_{\tau_1, s}) - g(x, \xi^x_{\tau_1, s}), \xi^{x'}_{\tau_2, s} - \xi^x_{\tau_1, s}
      \rangle + 12\|\alpha_2(x',\xi^{x}_{\tau_1, s})-\alpha_2(x,\xi^x_{\tau_1, s})\|^2\big) \Big] \notag \\
    \le & 
    -\frac{4\lambda}{\epsilon} \bE|\xi^{x'}_{\tau_2, s} - \xi^{x}_{\tau_1,
    s}|^4 
   +  \frac{C}{\epsilon}
    \bE\big(
|\xi^{x'}_{\tau_2, s} - \xi^x_{\tau_1, s}|^3|x'-x|\big) + \frac{C}{\epsilon}
    \bE\big(|\xi^{x'}_{\tau_2, s} - \xi^x_{\tau_1, s}|^2|x'-x|^2\big) \notag \\
    \le & 
    -\frac{2\lambda}{\epsilon} \bE|\xi^{x'}_{\tau_2, s} - \xi^{x}_{\tau_1,
    s}|^4 + \frac{C}{\epsilon} |x' - x|^4\,, \notag 
  \end{align}
  where inequality (\ref{hodler-young-ineq}) is used to obtain the last inequality.
  Gronwall's inequality together with the first statement above then yield the assertion.
  \end{enumerate}
  \qed
\end{proof}

Now consider the derivative process 
\begin{align}
  d\xi^x_{\tau, s, x_i} &= \frac{1}{\epsilon} \Big(D_{x_i} g(x,\xi^x_{\tau, s}) +
  \nabla_{y} g(x,\xi^x_{\tau, s}) \xi^x_{\tau, s, x_i}\Big) ds +
  \frac{1}{\sqrt{\epsilon}}\Big(D_{x_i} \alpha_2(x,\xi^x_{\tau, s}) + \nabla_{y}
  \alpha_2(x,\xi^x_{\tau, s}) \xi^x_{\tau, s, x_i}\Big) dw_s\,, \notag
\end{align}
with $s\ge \tau\,, \xi^x_{\tau, \tau,x_i} = 0$, $1\le i \le k$. In the above,
we used $D_{x_i}$ to denote derivatives with respect to scalar $x_i \in
\mathbb{R}$ and
$\nabla_y$ to denote derivatives with respect to a vector $y \in \mathbb{R}^l$.
We summarize its properties in the following result. 
\begin{lemma}
  Under Assumptions~\ref{assumption-1}--\ref{assumption-2}, there exists a constant
  $C > 0$, independent of $\epsilon, x, y$, such that $\forall 1 \le i \le k$, 
  \begin{enumerate}
    \item
      For $x \in \mathbb{R}^k$, $s \ge \tau$, $\bE|\xi^x_{\tau, s, x_i}|^4 \le C$.
    \item
      For $\tau_1 \le \tau_2$, $x \in \mathbb{R}^k$, 
\begin{align}
  \bE|\xi^x_{\tau_2, s, x_i} - \xi^x_{\tau_1, s, x_i}|^2 \le C\big(1 +
  |x|^2+|y|^2\big) e^{-\frac{\lambda(s-\tau_2)}{\epsilon}}\,.
  \notag
\end{align}
    \item
      For $\tau_1 \le \tau_2$, $x, x' \in \mathbb{R}^k$, 
\begin{align}
  \bE|\xi^{x'}_{\tau_2, s, x_i} - \xi^{x}_{\tau_1, s, x_i}|^2 \le
  Ce^{-\frac{\lambda(s-\tau_2)}{\epsilon}} \left[ 1 + \frac{s -
  \tau_2}{\epsilon} \big(1 + |x|^2+|y|^2\big)\right] + C |x - x'|^2\,.
  \notag
\end{align}
  \end{enumerate}
  \label{lemma-stat-derivative-bound}
\end{lemma}
\begin{proof}
  \begin{enumerate}
    \item
      Using Ito's formula, Assumption \ref{assumption-1} (Lipschitz continuity
      of functions $g$ and $\alpha_2$),
      inequality (\ref{mixing}) in Remark~\ref{rmk-1}, as well as inequality
      (\ref{hodler-young-ineq}), we see that  
\begin{align}
  & \frac{d\bE|\xi^x_{\tau, s, x_i}|^4}{ds} \notag \\
  \le& \frac{1}{\epsilon} \bE\Big[|\xi^x_{\tau,
s, x_i}|^2\Big(4\langle D_{x_i} g(x,\xi^x_{\tau, s}) + \nabla_{y} g(x,\xi^x_{\tau,
s}) \xi^x_{\tau, s, x_i},  \xi^x_{\tau, s, x_i}\rangle + 6\|D_{x_i}
\alpha_2(x,\xi^x_{\tau, s}) + \nabla_{y} \alpha_2(x,\xi^x_{\tau, s})
\xi^x_{\tau, s, x_i}\|^2\Big)\Big] \notag \\
\le& \frac{1}{\epsilon} \bE\Big[|\xi^x_{\tau,
    s, x_i}|^2\Big(C|\xi^x_{\tau, s, x_i}| + 4\langle \nabla_{y} g(x,\xi^x_{\tau,
  s}) \xi^x_{\tau, s, x_i},  \xi^x_{\tau, s, x_i}\rangle + C + 12\|\nabla_{y} \alpha_2(x,\xi^x_{\tau, s})
\xi^x_{\tau, s, x_i}\|^2\Big)\Big] \notag \\
 \le& -\frac{2\lambda}{\epsilon} \bE|\xi^x_{\tau, s, x_i}|^4 + \frac{C}{\epsilon} \notag
\end{align}
and therefore $\bE|\xi^x_{\tau, s, x_i}|^4 \le C$ by Gronwall's inequality.

\item
Now consider $\xi^x_{\tau_1, s,x_i}, \xi^x_{\tau_2, s,x_i}$ with $\tau_1 \le \tau_2$.  
Using Lipschitz condition of functions $g, \alpha_2$, 
inequality (\ref{mixing}) in Remark~\ref{rmk-1}, as well as inequality
(\ref{hodler-young-ineq}), it follows 
\begin{align}
  & \frac{d\bE|\xi^x_{\tau_2, s, x_i} - \xi^x_{\tau_1, s, x_i}|^2}{ds} \notag \\
 =& \frac{2}{\epsilon} \bE\langle D_{x_i} g(x, \xi^x_{\tau_2, s}) - D_{x_i} g(x,
 \xi^x_{\tau_1, s}) + \nabla_{y} g(x, \xi^x_{\tau_2, s}) \xi^x_{\tau_2, s, x_i} - \nabla_{y}
 g(x, \xi^x_{\tau_1, s}) \xi^x_{\tau_1, s, x_i},  \xi^x_{\tau_2, s,
 x_i}-\xi^x_{\tau_1, s, x_i}\rangle \notag \\
 & + \frac{1}{\epsilon} \bE \|D_{x_i} \alpha_2(x, \xi^x_{\tau_2, s}) - D_{x_i}
 \alpha_2(x, \xi^x_{\tau_1, s}) + \nabla_{y} \alpha_2(x, \xi^x_{\tau_2, s})
 \xi^x_{\tau_2, s, x_i} - \nabla_{y}
\alpha_2(x, \xi^x_{\tau_1, s}) \xi^x_{\tau_1, s, x_i}\|^2 \notag \\
\le & \frac{C}{\epsilon}\bE\Big(|\xi^x_{\tau_2, s} - \xi^x_{\tau_1,
s}||\xi^x_{\tau_2, s, x_i} - \xi^x_{\tau_1, s, x_i}|\Big) + \frac{2}{\epsilon}
\bE\langle
\big(\nabla_{y} g(x, \xi^x_{\tau_2, s})-\nabla_{y} g(x, \xi^x_{\tau_1, s})\big)\xi^x_{\tau_1, s, x_i},
\xi^x_{\tau_2, s, x_i} - \xi^x_{\tau_1, s, x_i}\rangle \notag \\
& + \frac{2}{\epsilon} \bE\langle
\nabla_{y} g(x, \xi^x_{\tau_2, s})(\xi^x_{\tau_2, s, x_i} - \xi^x_{\tau_1, s, x_i}),
\xi^x_{\tau_2, s, x_i} - \xi^x_{\tau_1, s, x_i}\rangle + \frac{C}{\epsilon}
\bE|\xi^x_{\tau_2, s} - \xi^x_{\tau_1, s}|^2\notag \\
& + \frac{3}{\epsilon} \bE\|\big(\nabla_{y} \alpha_2(x, \xi^x_{\tau_2, s})-\nabla_{y}
\alpha_2(x, \xi^x_{\tau_1, s})\big)
\xi^x_{\tau_1, s, x_i}\|^2 
+ \frac{3}{\epsilon} \bE\|\nabla_{y} \alpha_2(x, \xi^x_{\tau_2, s})(\xi^x_{\tau_2, s,
x_i} - \xi^x_{\tau_1, s, x_i})\|^2 \notag \\
\le & -\frac{\lambda}{\epsilon} \bE|\xi^x_{\tau_2, s, x_i} - \xi^x_{\tau_1, s,
x_i}|^2 + \frac{C}{\epsilon} \big(\bE|\xi^x_{\tau_2, s} - \xi^x_{\tau_1,
s}|^4\big)^{\frac{1}{2}} \big(\bE|\xi^x_{\tau_1, s, x_i}|^4)^{\frac{1}{2}} + 
 \frac{C}{\epsilon}
\bE|\xi^x_{\tau_2, s} - \xi^x_{\tau_1, s}|^2 \notag \\
\le & -\frac{\lambda}{\epsilon} \bE|\xi^x_{\tau_2, s, x_i} - \xi^x_{\tau_1, s,
x_i}|^2 + \frac{C}{\epsilon} (1 + |x|^2+|y|^2)
e^{-\frac{2\lambda(s-\tau_2)}{\epsilon}}\,, \notag 
\end{align}
where the first assertion above and 
Lemma~\ref{lemma-stat-bound} have been used in the last inequality.
Then Gronwall's inequality entails 
\begin{align}
  \bE|\xi^x_{\tau_2, s, x_i} - \xi^x_{\tau_1, s, x_i}|^2 \le C\big(1 +
  |x|^2+|y|^2\big) e^{-\frac{\lambda(s-\tau_2)}{\epsilon}}\,.
  \notag 
\end{align}
\item
Consider $\xi^{x}_{\tau_1, s,x_i}, \xi^{x'}_{\tau_2, s,x_i}$ with $\tau_1
\le \tau_2$.  In a similar way, we have  
\begin{align}
  & \frac{d\bE|\xi^{x'}_{\tau_2, s, x_i} - \xi^{x}_{\tau_1, s, x_i}|^2}{ds} \notag
   \\
   =& \frac{2}{\epsilon} \bE\langle D_{x_i} g(x', \xi^{x'}_{\tau_2, s}) - D_{x_i} g(x,
   \xi^x_{\tau_1, s}) + \nabla_{y} g(x', \xi^{x'}_{\tau_2, s}) \xi^{x'}_{\tau_2, s, x_i} - \nabla_{y}
   g(x, \xi^x_{\tau_1, s}) \xi^x_{\tau_1, s, x_i},  \xi^{x'}_{\tau_2, s,
 x_i}-\xi^x_{\tau_1, s, x_i}\rangle \notag \\
 & + \frac{1}{\epsilon} \bE \|D_{x_i} \alpha_2(x', \xi^{x'}_{\tau_2, s}) - D_{x_i}
 \alpha_2(x, \xi^x_{\tau_1, s}) + \nabla_{y} \alpha_2(x', \xi^{x'}_{\tau_2, s})
 \xi^{x'}_{\tau_2, s, x_i} - \nabla_{y}
\alpha_2(x, \xi^x_{\tau_1, s}) \xi^x_{\tau_1, s, x_i}\|^2 \notag \\
\le & \frac{2}{\epsilon} \bE\langle D_{x_i} g(x', \xi^{x'}_{\tau_2, s}) -
D_{x_i} g(x',
   \xi^x_{\tau_1, s}) + \nabla_{y} g(x', \xi^{x'}_{\tau_2,
   s})(\xi^{x'}_{\tau_2, s, x_i} -\xi^x_{\tau_1, s, x_i}),  \xi^{x'}_{\tau_2, s,
 x_i}-\xi^x_{\tau_1, s, x_i}\rangle \notag \\
 & + \frac{2}{\epsilon} \bE\langle D_{x_i} g(x', \xi^{x}_{\tau_1, s}) -
D_{x_i} g(x, \xi^x_{\tau_1, s}) + \big(\nabla_{y} g(x', \xi^{x'}_{\tau_2,
   s}) - \nabla_{y} g(x, \xi^{x}_{\tau_1, s})\big) \xi^x_{\tau_1, s, x_i},  \xi^{x'}_{\tau_2, s,
 x_i}-\xi^x_{\tau_1, s, x_i}\rangle \notag \\
 & + \frac{3}{\epsilon} \bE \|D_{x_i} \alpha_2(x', \xi^{x'}_{\tau_2, s}) - D_{x_i}
 \alpha_2(x, \xi^x_{\tau_1, s})\|^2 
 + \frac{3}{\epsilon} \bE \|\nabla_{y} \alpha_2(x', \xi^{x'}_{\tau_2, s})
 (\xi^{x'}_{\tau_2, s, x_i} - \xi^x_{\tau_1, s, x_i})\|^2 \notag \\
 & + \frac{3}{\epsilon} \bE \|\big(\nabla_{y} \alpha_2(x', \xi^{x'}_{\tau_2, s})
 - \nabla_{y} \alpha_2(x, \xi^x_{\tau_1, s})\big) \xi^x_{\tau_1, s, x_i}\|^2 \notag \\
   \le &
-\frac{2\lambda}{\epsilon} \bE|\xi^{x'}_{\tau_2, s, x_i} - \xi^{x}_{\tau_1, s,
x_i}|^2 
+ \frac{C}{\epsilon} \bE\big(|\xi^{x'}_{\tau_2, s} - \xi^x_{\tau_1,
 s}||\xi^{x'}_{\tau_2, s, x_i}-\xi^x_{\tau_1, s, x_i}|\big) 
+ \frac{C}{\epsilon} \bE\big(|x' - x||\xi^{x'}_{\tau_2, s, x_i}-\xi^x_{\tau_1, s, x_i}|\big) \notag \\
& +\frac{C}{\epsilon} \bE\big[\big(|x'-x|+
|\xi^{x'}_{\tau_2, s}-\xi^x_{\tau_1, s}|\big)|\xi^x_{\tau_1, s,
x_i}||\xi^{x'}_{\tau_2, s, x_i}-\xi^x_{\tau_1, s, x_i}|\big]  
 + \frac{C}{\epsilon} |x-x'|^2+\frac{C}{\epsilon}\bE |\xi^{x'}_{\tau_2,
 s}-\xi^x_{\tau_1, s}|^2 \notag \\
 & + \frac{C}{\epsilon} \bE \big[(|x'-x| + |\xi^{x'}_{\tau_2, s}-\xi^x_{\tau_1,
s}|\big)^2|\xi^x_{\tau_1, s, x_i}|^2\big] \notag \\
   \le &
-\frac{\lambda}{\epsilon} \bE|\xi^{x'}_{\tau_2, s, x_i} - \xi^{x}_{\tau_1, s,
x_i}|^2 + \frac{C}{\epsilon}\Big(|x' - x|^2 + \bE|\xi^{x'}_{\tau_2, s} - \xi^{x}_{\tau_1,
s}|^2 + (\bE|\xi^{x'}_{\tau_2, s} - \xi^{x}_{\tau_1,
s}|^4)^{\frac{1}{2}}\Big) \notag \\
  \le &
-\frac{\lambda}{\epsilon} \bE|\xi^{x'}_{\tau_2, s, x_i} - \xi^{x}_{\tau_1, s,
x_i}|^2 + \frac{C}{\epsilon} \Big[(1 + |x|^2+|y|^2)
  e^{-\frac{\lambda(s-\tau_2)}{\epsilon}} + |x' - x|^2\Big]\,, \notag
\end{align}
and thus 
\begin{align}
  \bE|\xi^{x'}_{\tau_2, s, x_i} - \xi^{x}_{\tau_1, s, x_i}|^2 \le
  Ce^{-\frac{\lambda(s-\tau_2)}{\epsilon}} \Big[ 1 + \frac{s -
  \tau_2}{\epsilon} (1 + |x|^2+|y|^2)\Big] + C |x' - x|^2\,.
  \notag
\end{align}
\qed
\end{enumerate}
\end{proof}

The above results allow us to define the stationary process
$\xi_s^x=\xi_{-\infty,s}^{x}$ with $\xi_{s}^{x}\sim\rho_{x}(y)\,dy$ where $\rho_{x}$ is the 
stationary probability density \wrt Lebesgue measure, and also the derivative process $\xi^x_{s, x_i}$ 
for $1 \le i \le k$, satisfying that $\forall f \in C^1_b(\mathbb{R}^k\times
\mathbb{R}^l)$ and $\widetilde{f}(x)
= \bE(f(x, \xi_s^x)) = \int_{\mathbb{R}^l} f(x,y) \rho_x(y) dy$, it holds
\begin{align}
  D_{x_i} \widetilde{f}(x) = \bE\big(D_{x_i}f(x, \xi_s^x) +
  \nabla_{y}f(x,\xi_s^x)\xi^x_{s,x_i}\big)\,.
\end{align}

The processes $\xi^x_s$ and $\xi^x_{s, x_i}$ have the following properties:  
\begin{lemma}
Under Assumptions \ref{assumption-1} and \ref{assumption-2}, there is a constant
$C>0$, independent of $\epsilon$, $x$ and $y$, such that 
$\forall f \in C_b^1(\mathbb{R}^l)$: 
\begin{enumerate}
  \item
  \begin{align}
    \Big| \bE f(\xi^x_{0,s}) - \int_{\mathbb{R}^l} f(y) \rho_x(y) dy\Big| \le \sup |f'|
    \Big(|x| + |y| + 1\Big)
    e^{-\frac{\lambda s}{\epsilon}}\,.
  \end{align}
\item 
 \begin{align}
   \Big| \bE\Big(f(\xi^x_{0,s})\xi^x_{0,s,x_i}\Big)
   -\bE\Big(f(\xi^x_{s})\xi^x_{s,x_i}\Big) \Big| 
\le C\Big(\sup|f| + \sup|f'|\Big) \Big(1 + |x| + |y|\Big) e^{-\frac{\lambda
s}{2\epsilon}}\,.
  \end{align}
  \end{enumerate}
  \label{exp-converge}
\end{lemma}
\begin{proof}
  We only prove the second inequality, as the first one follows in a similar fashion. 
  Using Lemma~\ref{lemma-stat-bound} and Lemma~\ref{lemma-stat-derivative-bound}, we  readily conclude that 
  \begin{align}
    & \Big| \bE\big(f(\xi^x_{0,s})\xi^x_{0,s,x_i}\big)
    -\bE\big(f(\xi^x_{s})\xi^x_{s,x_i}\big)
    \Big| \notag \\
\le & \Big| \bE\big[f(\xi^x_{s})(\xi^x_{0,s,x_i} -\xi^x_{s,x_i})\big]\Big| +
\Big|\bE\big[(f(\xi^x_{0,s}) - f(\xi^x_{s}))\xi^x_{0, s,x_i}\big] \Big| \notag \\
\le & C\big(\sup|f| + \sup|f'|\big) \big(1 + |x| + |y|\big) e^{-\frac{\lambda s}{2\epsilon}}
\notag
\end{align}
\qed
\end{proof}

An analogous property for the stationary process $\xi^x_{s}$ is the following:  
\begin{lemma}
  Under Assumption~\ref{assumption-1} and \ref{assumption-2}, there exists constant
  $C > 0$, independent of $x, x'$, such that 
  \begin{enumerate}
    \item
      For $x \in \mathbb{R}^k$, $\bE|\xi^x_{s, x_i}|^4 \le C$.
    \item
      For $x , x'\in \mathbb{R}^k$, $\bE|\xi^{x'}_{s} - \xi^x_{s}|^4 \le C |x - x'|^4$.
    \item
      For $x, x' \in \mathbb{R}^k$, $\bE|\xi^{x'}_{s, x_i} - \xi^{x}_{s, x_i}|^2 \le C |x - x'|^2$.
  \end{enumerate}
  \label{lemma-stationary}
\end{lemma}
\begin{proof}
  The conclusions follow directly by letting $\tau_1, \tau_2\rightarrow
  -\infty$ in Lemma~\ref{lemma-stat-bound} and Lemma~\ref{lemma-stat-derivative-bound}. 
\end{proof}

\bibliographystyle{siam}
\bibliography{reference}
\end{document}